\documentclass[onefignum]{siamonline250211}

\usepackage{lipsum}
\usepackage{comment}
\usepackage{subcaption}
\usepackage{amsfonts}
\usepackage{graphicx}
\usepackage{epstopdf}
\usepackage{algorithmic}
\usepackage{mathrsfs}
\ifpdf
  \DeclareGraphicsExtensions{.eps,.pdf,.png,.jpg}
\else
  \DeclareGraphicsExtensions{.eps}
\fi

\usepackage{enumitem}
\setlist[enumerate]{leftmargin=.5in}
\setlist[itemize]{leftmargin=.5in}

\newsiamremark{remark}{Remark}
\newsiamremark{hypothesis}{Hypothesis}
\crefname{hypothesis}{Hypothesis}{Hypotheses}
\newsiamthm{claim}{Claim}
\newsiamremark{fact}{Fact}
\crefname{fact}{Fact}{Facts}

\headers{Observer design for Chemical Reaction Networks }{Animikh Biswas, Gargi Chaudhuri, and  Muruhan Rathinam}

\title{Data Assimilation for Chemical Reaction Networks and Population Models via a Tunable Observer}

\author{Animikh Biswas\thanks{Department of Mathematics and Statistics, University of Maryland Baltimore County, ID 
  (\email{abiswas@umbc.edu}).}
\and Gargi Chaudhuri\thanks{Department of Mathematics and Statistics, University of Maryland Baltimore County, ID 
  (\email{gargic1@umbc.edu}).}
  \and Muruhan Rathinam\thanks{Department of Mathematics and Statistics, University of Maryland Baltimore County, ID 
  (\email{muruhan@umbc.edu}).}}

\usepackage{amsopn}

\ifpdf
\hypersetup{
  pdftitle={Observer design for Chemical Reaction Networks and Population models},
  pdfauthor={Animikh Biswas, Gargi Chaudhuri, and Muruhan Rathinam}
}
\fi

\newcommand{\comments}[1]{}
\newcommand{\real}{{\mathbb R}}
\newcommand{\sR}{{\mathcal R}}
\begin{document}

\maketitle

\begin{abstract}
We consider the problem of state reconstruction for a nonlinear dynamical 
system from observations of a linear function of the state. We present a
design method for a tunable observer and provide a general theorem which under 
certain conditions guarantees exponential convergence of the observer regardless of initial error. Additional results are provided that apply this theorem to 
chemical reaction network models. Moreover, these results are illustrated 
via examples of mass action form of chemical reaction networks where a subset of the 
species concentrations are observed. Numerical results are provided to show 
the efficacy of our proposed observer. Numerical results are also shown 
for the case of noisy observations and our observer is compared favorably with the
particle filter when the observation noise is small. 
\end{abstract}



\section{Introduction}



Dynamic models in the form of evolution equations (ODEs, PDEs, etc.) derived from scientific principles play a vital role in the prediction and control of the behavior of natural and engineered systems. Even when a dynamic model is reasonably accurate, predicting the state of a system requires the knowledge of the entire state at some prior instant of time (i.e.\ initial conditions) and system parameters. Hence the task of forecasting using a dynamical model that exhibits complex and possibly chaotic dynamics is often hindered by the lack of precise measurements of the state variables. 
An example of this occurs in weather prediction, where one collects data from sparsely located weather stations. The goal in this context is to use these partial state  measurements to obtain an accurate estimate of the full state. In weather forecasting, particularly when the data are corrupted with observational noise, this is referred to broadly as {\em data assimilation} \cite{daley1991atmospheric, Kalnay2003, LSZ}. It is also closely related to observer design and the concept of the {\em Luenberger observer} or simply an {\em observer} in control theory \cite{luenberger1964observing, luenberger1966observers}. While control systems in engineering, atmospheric science, geoscience and meteorology have provided the initial impetus for the subject, it has now found widespread application, including, but not limited to, environmental sciences, systems biology and medicine \cite{kostelich2011accurate, mcdaniel2012data}, imaging science, traffic control and urban planning, economics and finance and oil exploration \cite{asch2016data}. 

Classically, data assimilation techniques are based on linear quadratic estimation, also known as the Kalman Filter, due to its phenomenal success. The Kalman Filter however, has the drawback of assuming that the underlying system and any corresponding observation models are linear and that 
the noise is Gaussian. The classical Kalman filter has been extended by practitioners to nonlinear models giving rise to various techniques such as the Ensemble Kalman Filter (EnKF), Extended Kalman Filter (EKF), the Unscented Kalman Filter and others; see \cite{ABN, LSZ} and references therein for a detailed account. However, unlike the Kalman filter,  these do not enjoy the optimality property and have other potential drawbacks such as lack of stability and accuracy and may exhibit {\em catastrophic filter divergence}, particularly for chaotic dynamical systems \cite{HM}. 
When the dynamics are nonlinear and the system and observation noise are not Gaussian, 
particle filters \cite{doucet2001sequential, bain2008fundamentals} provide an approach to state estimation that is better grounded in theory. However, despite theoretical 
guarantees of convergence in the large particles limit, the computational burden of particle filters tends to be prohibitive in several high dimensional problems.  

The Kalman filter and the particle filter described above take a probabilistic interpretation and apply to situations where the dynamics and/or the observations are noisy. 
An alternative approach to above mentioned filters is an {\em observer}. Especially, when the dynamics are deterministic and the observations are noiseless, the observers provide an easily implementable method to obtain a point estimate.  
The very first idea of an observer is due to Luenberger who developed them for state estimation of a linear autonomous system of ODEs from linear observations of the state \cite{luenberger1964observing, luenberger1966observers}. An observer is 
a dynamical system whose state is a proxy for the true state of the original dynamics. Observers are typically obtained by augmenting a copy of the original dynamical system with a term that provides a corrective feedback based on the observations. For linear autonomous systems of ODEs, the problem of designing an observer whose state converges to the system state exponentially fast is completely solved. Moreover, for such systems, the Kalman filter itself can be regarded as an optimal observer without using a stochastic interpretation \cite{sontag2013mathematical}. For nonlinear systems, the design of observers is a challenging problem and has been the subject of research among control theorists for decades with no single best approach. 
See \cite{bernard2022observer} for a recent survey. Generalizing the notion of observability to nonlinear systems necessitated a 
differential geometric approach \cite{hermann1977nonlinear}. 
For differential geometric approaches to observer design, see \cite{krener1985nonlinear, krener2002nonlinear} for instance. Moreover, several methods in the control literature only provide local error convergence. That is, convergence is guaranteed only if the initial observer state is sufficiently close to the true state. 

In disciplines outside of control engineering, such as weather forecasting, 
observers have been in use \cite{Anthes1974, HA, N}, sometimes without using that terminology, simply referring to the approach broadly as {\em data assimilation} or sometimes as {\em nudging} to refer to the augmented feedback term.   
A rigorous analytical framework for this approach for dissipative nonlinear PDEs, was first developed in \cite{AOT, AT}. In particular, it is shown there that the 
observer converges exponentially to the solution of the original system regardless of the initial data used to initialize it. This initiated an active field of research; see \cite{ADR, BOT, FJJ, FMT, LR, LPV} and the references therein.

In this work, we consider chemical reaction networks and population models, where one needs to estimate the concentration of the non-measured species using the data from the measured concentrations of a subset of species. 
For the case of stochastically modeled chemical reaction networks, particle filtering methods have been developed recently. See \cite{rathinam2021state, rathinam2024stochastic, fang2022stochastic, fang2023convergence, fang2024effective, d2026filtered} for instance. For the case of deterministic models of reaction networks observers were provided in \cite{farina2009observer, chaves2001observers}. The results provided in \cite{chaves2001observers, farina2009observer} are local in nature. That is, it is shown that if the observer starts sufficiently close to the system, then convergence is assured. Our results are applicable to arbitrary initial errors. If the initial error is large, the tunable parameter $\mu$ associated with the observer needs to be chosen sufficiently large. 

The rest of this paper is organized as follows. In Section \ref{sec-observer}, we review the notion of an observer in the context of a nonlinear dynamical system (ODEs) where a linear function of the states is observed.  We propose a specific form for our tunable observer (with a tunable parameter $\mu>0$) to be designed and we state our {\em design goal}. This goal 
basically requires that regardless of initial error of the observer, exponential error convergence is achieved for all sufficiently large values of $\mu$. We provide a new general result in Theorem \ref{genthm} which provides sufficient conditions that guarantee that the design goal shall be met. Theorem \ref{genthm} applies to any dynamical system and linear observations, and is not confined to chemical reaction models. Section \ref{sec-CRN} provides an overview of chemical reaction network models which are applicable to all kinds of population models as well. In Section \ref{sec-obs-CRN}, we consider the application of Theorem \ref{genthm} to chemical reaction systems. Two propositions, Proposition \ref{prop1} and Proposition \ref{prop2}, provide sufficient conditions 
under which our observer design goal can be satisfied for a chemical reaction network.  While these propositions do not assume mass action kinetics, they are most useful for the case of mass action kinetics where a subset of species concentrations are observed. We provide four examples to illustrate the application of these two propositions.  
In Section \ref{sec_num_res}, we show via numerical simulations that for the examples considered in Section \ref{sec-obs-CRN}, our proposed observer converges exponentially. We also provide comparisons of our observer with what we call the {\em usual nudging} method 
for which no theoretical convergence guarantees are available. We also include numerical results for the case where the observations are corrupted by Brownian motion. 
In this case, we compare our proposed observer with a particle filter. When the observation noise is small, our observer seems to perform better than the particle filter which suffers from higher computational burden. In Section \ref{sec-conclusion}, we provide some concluding remarks.

\section{Data assimilation via a tunable observer}\label{sec-observer}
In this section we review the concept of an {\em observer} or a {\em nudged system} 
for the asymptotic reconstruction of the state of a 
partially observed dynamical system. We also present 
a main new result, Theorem \ref{genthm}, which will be used throughout the rest of this paper. 

Consider a dynamical system
\begin{equation}\label{eq_sys}
 \dot{x}(t)= f(x(t))
\end{equation}
where $f:\real^n \to \real^n$ is $C^1$. Suppose 
we observe $y(t) \in \real^m$ continuously in time $t \geq 0$ which is a linear function of the state $x(t)$ given by \begin{equation}
y(t) = C \, x(t)
\end{equation}
where $C\in \real^{m\times n}$. Our goal is to estimate $x(t)$ from the observed signal $y(t)$. In order to facilitate this, one designs another dynamical system 
with state $z(t) \in \real^n$ which is a proxy for $x(t)$ and evolves according to an equation of the form
\begin{equation}\label{eq-obs}
\dot{z}(t) = f(z(t)) + G \, (y(t)-C \,z(t))    
\end{equation}
where $G \in \real^{n \times m}$. The basic idea behind this is that $z(t)$ experiences the same vector field $f$ as the original system, but this vector field is augmented by the {\em nudging term} $G \, (y(t)- C\, z(t))$. We note that, the term $C \, z(t)$ corresponds to the observed signal if the actual state was $z(t)$ instead of $x(t)$. Thus $y(t)-C\, z(t)$ is an indication of the error between $z(t)$ and $x(t)$ that 
we can measure without the knowledge of $x(t)$. The term $G$ (known as the {\em gain} in control literature) basically maps the error $y(t)- C \, x(t)$ into an augmentation to the vector field. When $G$ depends on a scalar parameter $\mu$ which is to be 
determined, the observer is called a {\em tunable observer}. In general, the goal is to choose $G$ such that the error $e(t)=z(t)-x(t)$ converges to zero 
as $t \to \infty$, ideally, exponentially fast. 
We note that, exponential error convergence means that exists $\lambda>0$ and $M>0$ such that 
\[
|e(t)| \leq M \, e^{-\lambda t} \quad \forall t \geq 0,
\]
where $|x|$ is the standard Euclidean norm of $x \in \real^n$. 

Typically $m < n$ and without loss of generality we suppose that $C$ is full rank (surjective). It also makes sense to 
choose $G$ to be full rank, i.e.\ injective, for it to be most effective. When the vector field $f$ is linear (of the form $f(x)=Ax$), there exists comprehensive theory 
that helps choose $G$ under certain conditions to achieve exponential convergence of the error, see \cite{sontag2013mathematical} for instance. When $f$ is nonlinear, one may linearize $f$ around an equilibrium to 
apply this theory. However, this will only help when 
the initial states $x(0)$ and $z(0)$ are sufficiently close. 

In this paper, we aim to provide a tunable observer design which under certain conditions will achieve exponential convergence of error regardless of the size of the initial error $z(0)-x(0)$. Our approach will be to ensure that the error 
$z(t)-x(t)$ is contractive in a norm that arises from a 
suitable inner product $(\cdot,\cdot)_N$. 

Two subspaces of $\real^n$ will play an important role in the behavior of the observer. First is the null space 
of the observation which we denote by $K$ so that  
$K = \ker C$. We note that $K$ depends only on the observation model. We denote the range of $G$ by $E$ 
and note that $E$ depends on the choice of $G$. 
If we denote the observer error by $e(t)$, so that $e(t)=z(t)-x(t)$, then $e(t)$ and $x(t)$ evolve according to the combined dynamics
\begin{equation}
\begin{aligned}
\dot{x}(t) &= f(x(t)),\\
\dot{e}(t) &= [f(x(t)+e(t)) - f(x(t))] - H \, e(t), 
\end{aligned}    
\end{equation}
where $H=GC$. Note that $H:\real^n \to \real^n$ and 
that $\ker H = K$ and the range of $H$ is $E$. We also note that $\dim K = n-m$ and $\dim E=m$ where $C \in \real^{m \times n}$. 
Let $(\cdot,\cdot)_N$ be an inner product on $\real^n$. 
Then
\[
\frac{1}{2} \frac{d}{dt} (e(t), e(t))_N = (e(t), f(x(t)+e(t))-f(x(t)))_N - (e(t), H e(t))_N. 
\]
In order to ensure that $(e(t),e(t))_N$ is decreasing in time $t$, our goal is to ensure that the quadratic form 
$(e,H e)_N$ for $e \in \real^n$ is positive semidefinite 
and is of maximum rank as possible. We note that $(e,He)_N=0$ 
for $e \in K$ (as $He=0$) and also for $e \in E^{\perp_N}$ 
since $H e \in E$. Here, $E^{\perp_N}$ is the orthogonal space of $E$ with respect to the inner product $(\cdot,\cdot)_N$. If $H$ is positive semidefinite, it follows that $(e, He)_N=0$ for all $e \in K + E^{\perp_N}$ (see Lemma \ref{lem-H-rank}). Thus, to maximize the rank of $(e,He)_N$, one must have that $E^{\perp_N}=K$. This
implies that $E \cap K=\{0\}$ and $\real^n = K \oplus E$. 

\begin{lemma}\label{lem-H-rank}
Suppose $(\cdot,\cdot)_N$ is an inner product on $\real^n$ and $H:\real^n \to \real^n$ satisfies 
\[
(e,He)_N \geq 0 \quad \forall e \in \real^n.
\]
Suppose further that $K=\ker H$ and 
$E$ is the range of $H$. Then for all $e \in K + E^{\perp_N}$ it holds that $(e, He)_N=0$.     
\end{lemma}
\begin{proof}
Let $u \in K$ and $v \in E^{\perp_N}$ be arbitrary. Then 
\[
0 \leq (u + v, H (u + v) )_N  = (u, H v)_N.
\]
So $(u,Hv)_N \geq 0$ for all $u \in K$ and all $v \in E^{\perp_N}$. By replacing $u$ with $-u$, we obtain that 
$(u, Hv)_N \leq 0$ and hence $(u, Hv)_N=0$ for all $u \in K$ and $v \in E^{\perp_N}$. From this, the result follows. 
\end{proof}

Thus going forward, we shall pick $E$ to be complementary to $K$. We make the further choice that $H$ is of the form  $H= \mu \Pi_E$ where $\Pi_E$ is the projection from $\real^n$ onto $E$ along $K$ and $\mu$ is a positive parameter to be tuned. We note that it is possible  
to ensure $GC = \mu \Pi_E$ as follows. Given 
$y \in \real^m$, since $C$ is surjective and $E$ is complementary to the null space $K$ of $C$, there exists unique $x \in E$ such that $y = Cx$. One assigns $Gy=\mu x$. 
From this, it follows that $GC = \mu \Pi_E$. 
With this choice, our observer equation takes the form 
\begin{equation}
z(t) = f(z(t)) - \mu \, \Pi_E (z(t)-x(t)).    
\end{equation}
The combined system and error equations take the form 
\begin{equation}\label{eq-sys-error}
\begin{aligned}
\dot{x} &= f(x),\\
\dot{e} &= f(x+e)-f(x) - \mu \Pi \, e,
\end{aligned}    
\end{equation}
where we have suppressed time dependence for brevity.  

When $E=K^\perp$, that is $E$ is the orthogonal 
space of $K$ in the standard inner product, we refer 
to the resulting tunable observer as the {\em usual} or {\em standard nudging} method. As an example, it is common in the literature to only nudge those variables that 
are observed. For instance, if we observe the first $m$ components $x_1, \dots, x_m$, then the usual nudging observer takes the form 
\[
\begin{aligned}
\dot{z}_i &= f_i(z) - \mu (z_i - x_i) \quad i=1,\dots,m,\\
\dot{z}_i &= f_i(z) \quad i=m+1,\dots,n.
\end{aligned}
\]
While this usual nudging is tied to the standard inner product, our observer is more general as it is not tied 
to the standard inner product. We also note that once we choose a subspace $E$ complementary to $K$, the form of the tunable observer is fixed. 
With this in mind, we state our design goal as follows. 

\paragraph{Design goal}\label{design-goal}
Choose the $m$ dimensional subspace $E$ 
complementary to $K$ such that the following property satisfied: {\em for every compact set $\Gamma \subset \real^n$ there exists $\mu_0>0$ such that for all $\mu > \mu_0$,
if the system trajectory $x(t) \in \Gamma$ for all $t \geq 0$ then the error of the observer initialized in $\Gamma$
converges to zero exponentially fast.} Roughly speaking, this property asserts that regardless of the size of the initial error one may tune the observer to achieve exponential error convergence. 

Before we state a theorem that provides sufficient conditions under which our design goal will be met, 
we introduce some notation. 
Let $(\cdot,\cdot)_N$ be an inner product on $\real^n$. We denote a ball in $\real^n$ centered 
at the origin and of radius $R>0$ in this inner product by $B_{R,N}$. We drop the subscript $N$ if the standard inner product is used. 

\begin{theorem}\label{genthm}
Suppose the solution $x(t)$ lies in a compact set $\Gamma \subset \mathbb{R}^n$ for $t \geq 0$. 
Suppose $f$ is locally Lipschitz and that there exists an 
inner product $(\cdot, \cdot)_N$ on $\mathbb{R}^n$ such that $E \perp_N K$ and there exists $\alpha_0>0$ such that 
\begin{equation}
		 (e, f(x+e) - f(x))_N \leq -\alpha_0 |e|_N^2 
\end{equation}
for all $x \in \Gamma$ and $e \in K \cap B_{2R,N}$, where $R>0$ is such that 
 $\Gamma \subset B_{R,N}$. 
Then, there exists $\mu_0>0$
such that for all $\mu > \mu_0$ the error $e(t)$ of the observer initialized in $\Gamma$ tends to zero exponentially. Moreover, in the exponential rate, any exponent $\lambda \in (0,\alpha_0)$ may be achieved for all sufficiently large $\mu$. 
\end{theorem}

\begin{proof}
Assume that the observer is initialized in $\Gamma \subset B_{R,N}$. Then $e(0)=z(0)-x(0) \in B_{2R,N}$. 
Let $T=\inf\{t>0 \, | \, e(t) \notin B_{2R,N}\}$, with the infimum of the empty set defined by $\infty$. Clearly $T>0$ as $e(0) \in B_{2R,N}$. Suppose $T<\infty$. 
Then, by continuity $|e(T)|_N=2 R$. For $t \in [0,T)$ we consider the time evolution of 
$|e(t)|_N$. Let $L$ be the Lipschitz constant for $f$ on $B_{3R,N}$ in the 
norm $|\cdot|_N$. Suppressing time dependence, and writing $x(t)=x$ and $e(t)=e=e_E + e_K$ where $e_E \in E$ and $e_K \in K$ 
are the projections, we may write for $0 \leq t <T$ 
		\begin{align*}
			\frac{1}{2} \frac{d}{dt} |e|_N^2 &= \left(e, \frac{de}{dt}\right)_N\\
			&= (e, f(x+e) - f(x) - \mu e_E)_N \\
			&= (e, f(x+e_E + e_K) - f(x+e_K))_N + (e, f(x+e_K) - f(x))_N - \mu |e_E|_N^2 \\
			&\leq (e_E, f(x+e_K + e_E) - f(x+e_K))_N + (e_E, f(x+e_K)-f(x))_N\\
			&\quad + (e_K, f(x+e_K+e_E) - f(x+e_K))_N + (e_K, f(x+e_K) - f(x))_N - \mu |e_E|_N^2 \\
			&\leq 2L |e_E|_N |e_K|_N + L |e_E|_N^2 - \alpha_0 |e_K|_N^2 - \mu |e_E|_N^2,
\end{align*} 
where we note that, since $e(t) \in B_{2R,N}$, we also have that $e_E(t), e_K(t) \in B_{2R,N}$. Hence we have that $x(t) + e(t), x(t) + e_E(t), x(t) + e_K(t) \in B_{3R,N}$ for $t \in [0,T)$.  
Thus we may write
\[
			\frac{1}{2} \frac{d}{dt} |e|_N^2 
\leq \Bar{e}^T \, Q \, \Bar{e},
\]
where $\Bar{e}=(|e_E|_N,|e_K|_N)$ (column vector) and 
\[
Q = \left[ \begin{array}{rr} -\mu + L & L \\L & -\alpha_0\\
\end{array} \right].
\]
The symmetric matrix $Q$ is negative definite 
provided $\mu>\mu_0$ where $\mu_0 = L + \frac{L^2}{\alpha_0}$. Let $-\lambda$ be the largest eigenvalue 
of $Q$, so that $\lambda>0$ when $\mu>\mu_0$. 
Hence, we get
		\[ \frac{1}{2} \frac{d}{dt} |e|_N^2 \leq -\lambda |e|_N^2 \quad 0 \leq t < T.
        \]
By Gronwall's inequality, for $t \in [0,T)$
		\[ |e(t)|_N \leq e^{-\lambda t} |e(0)|_N. \]
Thus by continuity $|e(T)|_N < |e(0)|_N < 2R$, reaching a contradiction. 	
Thus $T=\infty$, and hence $|e(t)|_N \leq |e(0)|_N e^{-\lambda t}$ for all $t \geq 0$, 
showing exponential convergence of the error. 
We note that, due to norm equivalence, in the standard Euclidean norm, the error satisfies a bound of the form
\[
|e(t)| \leq M \, e^{-\lambda t}, \forall t \geq 0,
\]
where $M>0$ is independent of $t$. 

Under the assumption $\mu > \mu_0$, we may write $\lambda=2 \alpha_0 \, g(\mu)$ where 
\[
g(\mu) = \frac{\mu-\mu_0}{(\mu - \mu_1) + \sqrt{(\mu - \mu_2)^2 + 4 L^2}} = \frac{1}{(\frac{\mu-\mu_1}{\mu-\mu_0}) + \sqrt{(\frac{\mu-\mu_2}{\mu-\mu_0})^2 + (\frac{2 L}{\mu-\mu_0}})^2},
\]
with $\mu_1= L-\alpha_0$ and $\mu_2 = L + \alpha_0$. 
We observe that as $\mu \to \infty$, $\lambda \to \alpha_0$. Also as $\mu \to \mu_0$ from above, $\lambda \to 0$. Thus, any exponent in the range $(0,\alpha_0)$ 
may be achieved by sufficiently large $\mu$. 
We also note that if $L \geq \alpha_0$ then $\mu_0 \geq \mu_2>\mu_1>0$, and hence $F(\mu)$ is monotonically increasing for $\mu \in (\mu_0,\infty)$. 
\end{proof}

\begin{remark}\label{rem-bigthm}
We note that in Theorem \ref{genthm} the space $E$, the inner product $(\cdot,\cdot)_N$ and $\alpha_0>0$ are allowed to depend on the compact set $\Gamma$. However, our design goal requires that the space $E$ be independent of $\Gamma$. This is assured if $E$ may be chosen independent of $\Gamma$. Our subsequent results provide sufficient conditions under which this holds. 
\end{remark}

\section{Chemical reaction networks (CRNs)}\label{sec-CRN}
A Chemical Reaction Network (CRN) model consists of $n$ species interacting via 
$r$ reactions. If we denote the species by $S_1, \dots, S_n$, then a reaction 
consists of a nonnegative integer linear combination of the species 
being converted to a nonnegative integer linear combination of the species, and may be written as follows: 
\[
\sum_{i=1}^n \nu^-_i S_i \longrightarrow \sum_{i=1}^n \nu^+_i S_i
\]
where $\nu^+=(\nu^+_1,\dots,\nu^+_n), \nu^-=(\nu^-_1,\dots,\nu^-_n) \in \mathbb{Z}_+^n$. If $\nu^-=0$, we denote the reaction by 
\[
\emptyset \longrightarrow \sum_{i=1}^n \nu^+_i S_i.
\]
Likewise for $\nu^+=0$. 

Thus, for each reaction $j$, we have the quantities $\nu_j^+, \nu_j^- \in \mathbb{Z}_+^n$ and $\nu_j \in \mathbb{Z}^n$ where $\nu_j$ is called the {\em stoichiometric vector} of reaction $j$. If the $i$th component of $\nu_j^-$ is nonzero, then 
specie $i$ is a {\em reactant} in the reaction $j$. We shall refer to the $i$th component (when nonzero) of $\nu_j^-$ as the {\em stoichiometric coefficient} 
of the reactant specie $i$ in reaction $j$. For example, in the case of the Lotka-Volterra model \cite{volterra1928variations} considered later in Section \ref{sec-prop1-eg} there are two species $X_1$ (the prey) and $X_2$ (the predator) which undergo three reactions which may be written as 
		\begin{align*}
			X_1 &\xrightarrow{k_1} 2X_1, \\
			X_1 + X_2 &\xrightarrow{k_2} 2X_2, \\
			X_2 &\xrightarrow{k_3} \emptyset.
		\end{align*}
The first reaction represents prey reproduction and has one reactant specie $X_1$. 
The second reaction represents predator-prey interaction and the third reaction represents death of predator. As an example, in reaction 2, both species are reactants and each of them have a stoichiometric coefficient $1$.  

The dynamic model of a CRN consists of the state space $\real_+^n$ 
which stands for the set of possible species concentration vectors. 
(We shall use $\real_+$ to denote the set of nonnegative real numbers). 
We denote by 
$x_i(t)$ the concentration of the $i$th species and by $x(t)$ the vector of species concentrations. 
Associated to each reaction $j$ is a {\em reaction rate} function $a_j(x)$ which 
maps the species concentration vector $x$ to a nonnegative scalar which is regarded as 
the {\em reaction flux}. Then the dynamic model of the CRN is given by
\begin{equation}\label{eq-RRE}
\dot{x}(t) = \sum_{j=1}^r \nu_j \, a_j(x(t)).     
\end{equation}
The reaction rate functions $a_j$ are modeled in various ways, but the most 
common model is the so-called {\em mass action} form. In the mass action form, 
$a_j(x)$ is given by
\begin{equation}
a_j(x) = k_j \, x^{\nu_j^-},    
\end{equation}
where $k_j>0$ is a parameter and we have used the following notation for a monomial. 
Given $x \in \real^n$ and a multi-index $\gamma \in \mathbb{Z}_+^n$ 
\[
x^\gamma = x_1^{\gamma_1} \dots x_n^{\gamma_n}.
\]
For a multi-index $\gamma \in \mathbb{Z}_+^n$ we define $|\gamma|=\gamma_1+\dots+\gamma_n$. 

Thus, for the Lotka-Volterra example, the mass action reaction rate functions are given by
       \begin{align*}
			a_1(x) &= k_1 x_1, & 
			a_2(x) &= k_2 x_1 x_2, & 
			a_3(x) &= k_3 x_2.
		\end{align*}
The system equations may be written as
        \[
		\dot{x} = k_1 x_1 \begin{pmatrix} 1 \\ 0 \end{pmatrix} + k_2 x_1 x_2 \begin{pmatrix} -1 \\ 1 \end{pmatrix} + k_3 x_2 \begin{pmatrix} 0 \\ -1 \end{pmatrix} 
		\]
The {\em stoichiometric subspace} $S \subset \mathbb{R}^n$ of a reaction network is defined by
\begin{equation}
S = \text{span}\{ \nu_j \, | \, j=1,\dots,r\}.    
\end{equation}
It is clear that $S$ is invariant under the dynamics. 
The dynamic model \eqref{eq-RRE} is physically meaningful only if the  non-negative orthant $\real_+^n$ is forward invariant. This depends on the nature of the reaction rate functions $a_j$. In the case of mass action kinetics,  
the non-negative orthant $\real_+^n$ is indeed forward invariant; see for instance,  Proposition 2.4.1 in \cite{johnston2011thesis} which shows that the (strictly) positive orthant $(0,\infty)^n$ is forward invariant for mass-action systems and this proof also shows that $\real_+^n$ is forward invariant as well. 

\section{Observer design for CRNs}
\label{sec-obs-CRN}


Theorem \ref{genthm} together with the fact that the state space of a CRN is a subset of $\real_+^n$ suggests that it is instructive to 
study the map $F : \real_+^n \times \real^n \rightarrow \mathbb{R}^n$ defined by
\begin{equation}
F(x, e) = f(x+e) - f(x),
\end{equation}
and $F_K:\real_+^n \times K \to \real^n$, the restriction of $F$ to $\real_+^n \times K$. In particular, we wish to guarantee that the following contractivity condition
is satisfied. 

\paragraph{Contractivity Condition (CC)}
There exists an inner product $(\cdot,\cdot)_N$ on $\real^n$ such that $E \perp_N K$ and moreover, 
for every compact set $\Gamma \subset \real_+^n$ there exists $\alpha_0>0$ such that 
\begin{equation}\label{eq-contract}
(e, F_K(x,e))_N \leq -\alpha_0 (e, e)_N, \quad \forall (x,e) \in \Gamma \times K.    
\end{equation}

We shall present two propositions, Proposition \ref{prop1} and Proposition \ref{prop2}, that apply to CRNs and 
provide sufficient conditions under which CC holds. 
These propositions do not assume mass action form of rate functions. However, these propositions are most useful when the CRNs have mass action form of rate functions and 
a subset of the species concentrations are observed. 

Proposition \ref{prop1} involves writing $F_K(x,e) = A_K \, e + N_K(x,e)$ and considering the case where the range of $N_K(x,e)$ lies in a subspace complementary to $K$. In this case, one  chooses $E$ to contain the range of $N_K(x,e)$. This effectively eliminates having to consider the nonlinear terms and simply focus on the linear part $A_K \, e$ of $F_K(x,e)$. 

Proposition \ref{prop2} involves writing $F_K= B_K(x) \, e + \Tilde{N}_K(x,e)$ and considering the case where the range of $\Tilde{N}_K(x,e)$ is complementary to $K$. 
Thus the second result is in fact a generalization of the first result. We present the first result before discussing the second as it will be easier to follow the development of our ideas. 

In the case of CRNs, the vector field is of the form
 $f(x) = \sum_{j=1}^r \nu_j a_j(x)$.
Then 
\[
F(x, e) = \sum_{j=1}^r \nu_j (a_j(x+e) - a_j(x)).
\]
We also assume that $a_j$ are $C^1$. We define $\psi_j:\real_+^n \times \real^n \to \real$
by
\begin{equation}
\psi_j(x,e) = a_j(x+e)-a_j(x).    
\end{equation}
Thus, the functions $\psi_j$ and their restrictions to 
$\real_+^n \times K$ are important to study.  
\subsection{Decomposition $F_K(x,e) = A_K \, e + N_K(x,e)$}\label{sec-prop1}
We consider the Taylor expansion of $\psi_j(x,e)$ around $(0,0)$ to define the linear and nonlinear parts of $\psi_j$ and hence those of $F(x,e)$ and $F_K(x,e)$.  
Focusing on $(x,e) \in \real_+^n \times K$, we define the sets of reactions $\sR_{0,K}$ and $\sR_{1,K}$ by 
\begin{equation}
\begin{aligned}
\sR_{0,K} &= \{ j \, | \psi_j(x,e) = 0 \quad \forall (x,e) \in \real_+^n \times K \},\\
\sR_{1,K} &= \{ j \, | j \notin \sR_{0,K} \text{ and } \psi_j(x,e) =  \frac{d a_j}{dx}(0)\,  e \quad \forall (x,e) \in \real_+^n \times K \}.  
\end{aligned}
\end{equation}
Thus $\sR_{0,K}$ is the set of reactions $j$ for which $\psi_j(x,e)$ is zero on $\real_+^n \times K$ and $\sR_{1,K}$ is the set of reactions $j$ for which the restriction of $\psi_j(x,e)$ to $\real^n \times K$ is linear and nonzero. (We note that, from the definition of $\psi_j(x,e)$, the linear term in its Taylor expansion around $(0,0)$ is independent of $x$.) 

We define the stoichiometric subspace $S_{1,K}$ as follows:
\begin{equation}
S_{1,K} = \text{span}\{ \nu_j \, | \, j \notin (\sR_{1,K} \cup \sR_{0,K})\}.
\end{equation}
Thus $S_{1,K}$ contains the range of the nonlinearity in $F_K$.  
Denote the linear part of $F$ by $A e$, so that 
$A=\sum_{j=1}^r \nu_j \, \frac{da_j}{dx}(0)$. Then we may write 
\begin{equation}
F_K(x,e) = A_K \, e +  N_K(x,e)  
\end{equation}
where $A_K:K \to \real^n$ is the restriction of $A$ to $K$ and $N_K$ is the nonlinear part of $F_K$ which lies in $S_{1,K}$. We note that it may happen that 
$N_K(x,e)=0$ for all $(x,e) \in \real_+^n \times K$ 
due to cancellation of nonlinear terms when summing over 
$j$ to obtain $F_K(x,e)$. 

Assuming a decomposition $\real^n = E \oplus K$, we define $\Bar{A}_K:K \to K$ by 
\begin{equation}\label{eq-barAK}
\Bar{A}_K = \Pi_K \, A_K,    
\end{equation}
where $\Pi_K:\real^n \to K$ is the projection onto $K$ along $E$.

The next proposition provides a sufficient condition that guarantees the applicability of Theorem \ref{genthm}. This condition requires that either $S_{1,K} \cap K = \{0\}$  or $N_K=0$
and that $\Bar{A}_K$ is Hurwitz. 

\begin{proposition}\label{prop1}
Suppose we observe a CRN for which 
$S_{1,K} \cap K = \{0\}$ or $N_K(x,e)=0$ for all $(x,e) \in \real_+^n \times K$. Then we may pick $E$ to satisfy $\real^n = E \oplus K$ and $\text{range}(N_K) \subset E$. Suppose further that
$\Bar{A}_K$ is Hurwitz. 
Then the contractivity condition \eqref{eq-contract} is satisfied and hence the design goal is achieved. 
Suppose further that the symmetric part $(\Bar{A}_K + {\Bar{A}_K}^T)/2$ of $\Bar{A}_K$ is also Hurwitz. 
Then, the exponent $\alpha_0>0$ in \eqref{eq-contract} satisfies the bound 
\[
\alpha_0 \geq |\lambda_{\text{max}}(\Bar{A}_K + {\Bar{A}_K}^T)|/2,
\]
where $\lambda_{\text{max}}(M)$ is the greatest eigenvalue of a real symmetric matrix $M$.
		\end{proposition}
\begin{proof}
As  $\Bar{A}_K$ is Hurwitz, there exists a positive definite linear map $N:K \to K$ such that
		\[ {\Bar{A}_K}^T \, N + N  \, \Bar{A}_K  = -2 I\]
		where $I$ is the identity map on $K$. Thus $N$ defines an inner product on $K$. 
We extend this inner product $(\cdot, \cdot)_N$ to $\mathbb{R}^n $ as follows:
		\[
		(a,b)_N = 
		\begin{cases} 
			(a,b) & \text{if } a,b \in E, \text{ (the standard inner product), }\\ 
			0 & \text{if } a \in E, \, b \in K \text{ or } a \in K, \, b \in E, \\ 
			(a,b)_N & \text{if } a,b \in K. 
		\end{cases}
		\]
Let $e\in K$ and $x\in \real_+^n$. By our hypothesis, either the nonlinear part $N_K(x,e) \in S_{1,K} \subset E$ or $N_K(x,e)=0$. It follows that 
$(e, N_K(x,e))_N=0$. Hence 
\begin{align*}
\left(e,F_K(x,e) \right)_N
			&= \left(e, A_K e \right)_N\\
			= \left(e, \Bar{A}_K e \right)_N
			&= - |e|^2
            \leq -\frac{1}{\lambda_{\text{max}}(N)}|e|_N^2,
		\end{align*}
where $\lambda_{\text{max}}(N)$ is the greatest eigenvalue of $N$.  
Hence with
\[
\alpha_0= \frac{1}{\lambda_{\text{max}}(N)}, 
\]
the result follows. From \cite{smith1965bounds, lancaster1970explicit}, if $\Bar{A}_K + {\Bar{A}_K}^T$ is negative definite, 
then we may obtain the bound 
\[
\lambda_{\text{max}}(N) \leq \frac{2}{|\lambda_{\text{max}}(\Bar{A}_K + {\Bar{A}_K}^T)|}, 
\]
and hence $\alpha_0 \geq |\lambda_{\text{max}}(\Bar{A}_K + {\Bar{A}_K}^T)|/2$. 
\end{proof}

\paragraph{Mass action CRNs with a subset of species observed}
Before we show examples of application of Proposition \ref{prop1}, we shall consider the special case of 
mass action CRNs where a subset of the species concentrations are observed. 
 
We may write the species concentration vector $x \in \real_+^n$ as $x=(x_o,x_u)$ where $x_o$ is the vector concentration of the observed species and $x_u$ is the vector concentration of the unobserved species. Likewise 
we may write $\nu_j^\pm=(\nu_{j,o}^\pm,\nu_{j,u}^\pm)$ and 
$\nu_j=(\nu_{j,o},\nu_{j,u})$. 

Consider a reaction $j$. The mass action form of reaction rate is given by $a_j(x) = k_j \, x^{\nu^-_j}$. 
Let $e \in K$. Then we may write $e=(0,e_u)$ as the observed component $e_o=0$. Then for $x \in \real_+^n$ and $e \in K$ we have that 
\begin{equation}\label{eq-psij-mass-action}
\psi_j(x,e) = k_j [(x+e)^{\nu_j^-} - x^{\nu_j^-}] = k_j \, {x_o}^{\nu_{j,o}^-} \, [(x_u + e_u)^{\nu_{j,u}^-} - {x_u}^{\nu_{j,u}^-}].
\end{equation}
From this we may deduce the following:
\begin{enumerate}
\item $j \in \sR_{0,K}$ if and only if $|\nu_{j,u}^-|=0$.
In words, $j \in \sR_{0,K}$ if and only if all the reactants in $j$ are observed species. 
In this case $\psi_j(x,e)=0$ for $(x,e) \in \real_+^n \times K$. 
\item $j \in \sR_{1,K}$ if and only if $|\nu_{j,u}^-|=1$  and $|\nu_{j,o}^-|=0$. 
In words,  $j \in \sR_{1,K}$ if and only if there is exactly one reactant specie $i$ in reaction $j$, that specie $i$ is an unobserved specie and its stoichiometric coefficient is $1$.   
In this case $\psi_j(x,e) = k_j \, {e_u}^{\nu_{j,u}^-}$. 
\end{enumerate}
We also obtain that the linear part $A_K \, e$ of $F_K(x,e)$ is given by
\begin{equation}\label{eq-AK-mass-action}
A_K \, e = \sum_{j \in \sR_{1,K}} \nu_j \, k_j \, {e_u}^{\nu_{j,u}^-} \quad \forall e \in K.    
\end{equation}

\subsection{Examples showing the application of Proposition \ref{prop1}}\label{sec-prop1-eg}

\paragraph{Example $1$:  Lotka-Volterra model for population dynamics}\mbox 
We revisit the Lotka-Volterra example considered in 
Section \ref{sec-CRN}. 
We assume the mass action form of reaction rate functions.
\begin{figure}[htb]
     \centering
     \includegraphics[width=0.8\linewidth]{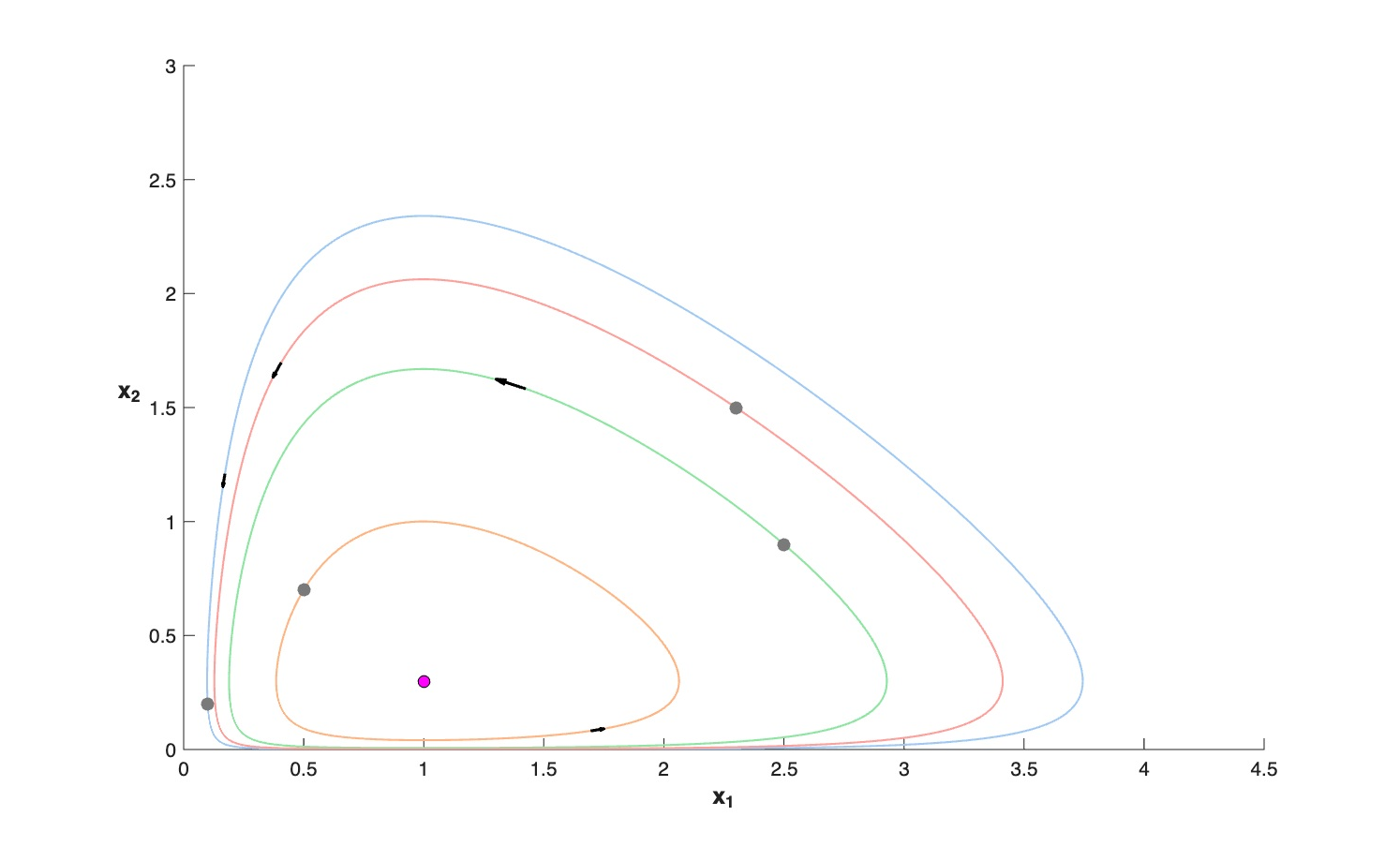}
     \caption{Phase-portrait: Lotka-Volterra model.}
     \label{lv_phase}
\end{figure}
        
The stoichiometric vectors and rate functions are given by
		\begin{align*}
			\nu_1 &= \begin{pmatrix} 1 \\ 0 \end{pmatrix}, & 
			\nu_2 &= \begin{pmatrix} -1 \\ 1 \end{pmatrix}, & 
			\nu_3 &= \begin{pmatrix} 0 \\ -1 \end{pmatrix} ,\\
			a_1(x) &= k_1 x_1, & 
			a_2(x) &= k_2 x_1 x_2, & 
			a_3(x) &= k_3 x_2
		\end{align*}
    where $x_i$ is the $i$th species concentration.
The system equations may be written as
        \[
		\dot{x} = k_1 x_1 \begin{pmatrix} 1 \\ 0 \end{pmatrix} + k_2 x_1 x_2 \begin{pmatrix} -1 \\ 1 \end{pmatrix} + k_3 x_2 \begin{pmatrix} 0 \\ -1 \end{pmatrix} 
		\]
It is well known that the solutions corresponding to strictly positive initial conditions are periodic orbits (one of them being an equilibrium). Moreover, $x_1=0$ 
and $x_2=0$ are invariant sets. If $x_1(0)=0$, 
then the solution limits to $(0,0)$ and if $x_2(0)=0$ then $x_1(t) \to \infty$. 
Hence, we shall assume that $x_2(0)>0$. 

The linear part $A$ is given by
		\[
		A = \begin{bmatrix} k_1 & 0 \\ 0 & -k_3 \end{bmatrix}. 
		\]
Suppose we observe $x_1$ (the prey population). Then it follows that $\sR_{0,K} = \{1\}$ (as reaction 1 consists only of observed species as reactants) and $\sR_{1,K} = \{3\}$.         
\[
		\begin{aligned}
			S_{1,K} &= \text{span}\left\{ \begin{pmatrix} -1 \\ 1 \end{pmatrix} \right\}, \\
			K &= \text{span}\left\{ \begin{pmatrix} 0 \\ 1 \end{pmatrix} \right\}.
		\end{aligned}
		\]
Thus, $S_{1,K} \cap K = \{0\}$ and we must choose $E=S_{1,K}$.        
        Now we have
        \begin{align*}
			\begin{bmatrix} k_1 & 0 \\ 0 & -k_3 \end{bmatrix}\begin{pmatrix} -1 \\ 1 \end{pmatrix}=\begin{pmatrix} -k_1 \\ -k_3 \end{pmatrix} &= k_1 \begin{pmatrix} -1 \\ 1 \end{pmatrix} + (-k_1-k_3) \begin{pmatrix} 0 \\ 1 \end{pmatrix},\\
			\begin{bmatrix} k_1 & 0 \\ 0 & -k_3 \end{bmatrix}\begin{pmatrix} 0 \\ 1 \end{pmatrix}=\begin{pmatrix} 0 \\ -k_3 \end{pmatrix} &= 0 \begin{pmatrix} -1 \\ 1 \end{pmatrix} + (-k_3) \begin{pmatrix} 0 \\ 1 \end{pmatrix}.
		\end{align*}
Then, we may write $A$ with respect to the decomposition $\real^2 = E \oplus K$ as		
		\[
		A = \left[ \begin{array}{c|c}
			k_1 & 0 \\
			\hline
			  \ \ -k_1-k_3 & -k_3
		\end{array} \right] \quad \text{w.r.to } E \oplus K.
		\]
	Therefore, $\Bar{A}_K = -k_3$ is Hurwitz.
Thus Proposition \ref{prop1} is applicable for this model
 with $\alpha_0 \geq k_3$. 

\paragraph{Example 2: The Willamowski-Rössler chaotic system}\mbox \\

This system \cite{gaspard2005rossler} relies on the interactions of three chemical species $X_1,\ X_2$ and $X_3$ governed by the following five reversible reactions:				
		\begin{align*}
			X_1 &\xleftarrow[k_2]{}\xrightarrow{k_1} 2X_1 \\
			X_1+X_2 &\xleftarrow[k_4]{}\xrightarrow{k_3} 2X_2 \\
			X_2 &\xleftarrow[k_6]{}\xrightarrow{k_5} \emptyset \\
			X_1+X_3 &\xleftarrow[k_8]{}\xrightarrow{k_7} \emptyset\\
			X_3 &\xleftarrow[k_{10}]{}\xrightarrow{k_9} 2X_3
		\end{align*}
where $k_j$ denote the rate constants and $k_j>0$ for $j\in\{1,...,10\}$. Here the reaction rate functions follow the mass action form by assumption.				
\begin{figure}[htb]
     \centering
     \includegraphics[width=0.8\linewidth]{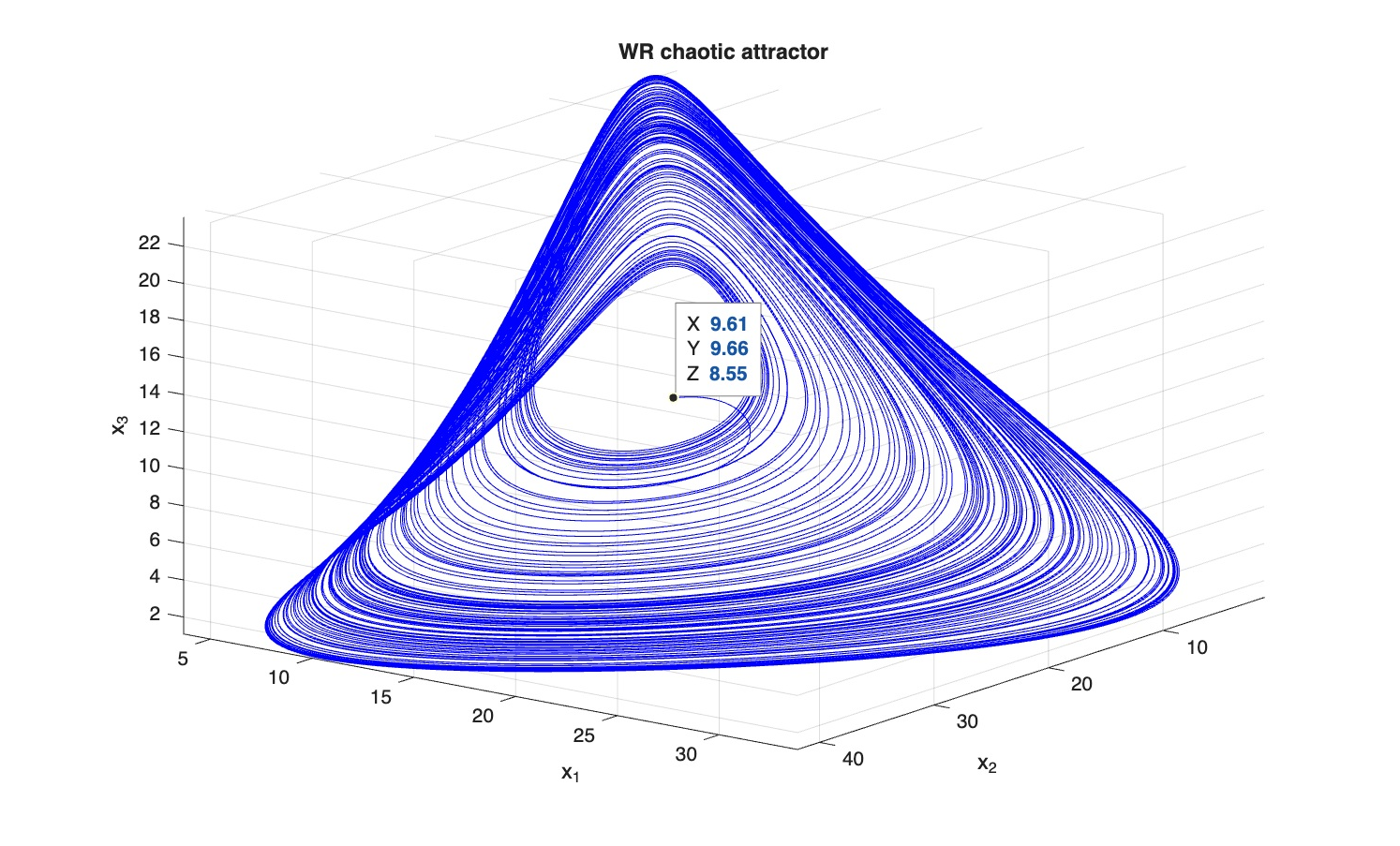}
     \caption{Chaotic attractor: Willamowski-Rössler model.}
     \label{wr2}
\end{figure}
This system is known to be chaotic \cite{gaspard2005rossler}. 
The stoichiometric vectors ($\nu_j$) and rate functions ($a_j(x)$) are given by				
				\[
				\nu_1 = \begin{pmatrix} 1 \\ 0 \\ 0 \end{pmatrix}, \quad
				\nu_2 = \begin{pmatrix} -1 \\ 0 \\ 0 \end{pmatrix}, \quad
				\nu_3 = \begin{pmatrix} -1 \\ 1 \\ 0 \end{pmatrix}, \quad
				\nu_4 = \begin{pmatrix} 1 \\ -1 \\ 0 \end{pmatrix}, \quad
                \nu_5 = \begin{pmatrix} 0 \\ -1 \\ 0 \end{pmatrix},
				\]
				\[
				\nu_6 = \begin{pmatrix} 0 \\ 1 \\ 0 \end{pmatrix}, \quad
				\nu_7 = \begin{pmatrix} -1 \\ 0 \\ -1 \end{pmatrix}, \quad
				\nu_8 = \begin{pmatrix} 1 \\ 0 \\ 1 \end{pmatrix},
				\nu_9 = \begin{pmatrix} 0 \\ 0 \\ 1 \end{pmatrix}, \quad
				\nu_{10} = \begin{pmatrix} 0 \\ 0 \\ -1 \end{pmatrix},
				\]
				
\begin{align*}
			a_1(x) &= k_1 x_1, & a_2(x) &= k_2 x_1^2, \\
			a_3(x) &= k_3 x_1 x_2, & a_4(x) &= k_4 x_2^2, \\
			a_5(x) &= k_5 x_2, & a_6(x) &= k_6, \\
			a_7(x) &= k_7 x_1 x_3, & a_8(x) &= k_8, \\
			a_9(x) &= k_9 x_3, & a_{10}(x) &= k_{10} x_3^2
\end{align*}
where $x_i(\geq 0)$ is the concentration of species $X_i$ and $i\in\{1,2,3\}$.\\
Therefore, by the mass-action kinetics, the system of differential equations becomes
\begin{align*}
    \dot{x}_1 &= k_1 x_1 - k_2 x_1^2 - k_3 x_1 x_2 + k_4 x_2^2 - k_7 x_1 x_3 + k_8, \\
    \dot{x}_2 &= k_3 x_1 x_2 - k_4 x_2^2 - k_5 x_2 + k_6, \\
    \dot{x}_3 &= -k_7 x_1 x_3 + k_8 + k_9 x_3 - k_{10} x_3^2.
\end{align*}

Now, we show how to use Proposition \ref{prop1} for this chaotic model.
The linear part is given by	
				\[
				A = \begin{bmatrix}
					k_1 & 0 & 0 \\
					0 & -k_5 & 0 \\
					0 & 0 & k_9
				\end{bmatrix}.
				\]
Suppose we observe the concentration of species $X_1$ and $X_3$ i.e., $x_1$ and $x_3$. Then $\sR_{0,K} =\{1, 2, 6, 7, 8, 9, 10\}$ 
and $\sR_{1,K} = \{ 5\}$. We also have

				\[
				K = \left\{ \begin{pmatrix} 0 \\ 1 \\ 0 \end{pmatrix} \right\}:=\{v_1\},\quad S_{1,K} = \text{span}\left\{ \begin{pmatrix}
				    -1\\1\\0
				\end{pmatrix}\right\}
                \]
where $v_1$ is defined and $K \cap S_{1,K}=\{0\}$. We select
                \[
                E=\left\{ \begin{pmatrix} -1 \\ 1 \\ 0 \end{pmatrix},  \begin{pmatrix} 0 \\ 0 \\ 1 \end{pmatrix} \right\}:=\{v_2,v_3\}
				\]
so that $S_{1,K} \subset E$ and $\real^3= E \oplus K$. 
(Note that, there are other choices for $E$.)  Then we can write the linear part of $F_K$ i.e., $A_Ke$ as
\[
A_Ke = \nu_5 k_5 e_2
\]
Then the calculations for finding $\bar{A}_K$ w.r.to $\{v_1\}$ as the basis for $K$ and $\{v_2, v_3\}$ as the basis for $E$ become as follows:
\[
A_Kv_1 = (-k_1-k_5)v_1+k_1v_2, \quad A_Kv_2 = k_9v_3, \quad A_Kv_3 = -k_5v_1,
\]

				Therefore,

				\[
				A = \left[\begin{array}{c c|c}
					k_1 & 0 & 0 \\
                    
					0 & k_9 & 0 \\
                    \hline
					-k_1-k_5 & 0 & -k_5
				\end{array}\right] \ \ \text{w.r.to \ \ $E\oplus K$}.
				\] 
			As $\bar{A}_K= -k_5$ is Hurwitz, we can use Proposition \ref{prop1} in this case and we obtain $\alpha_0 \geq k_5$.

\subsection{Decomposition $F_K(x,e) = B_K(x) \, e + \tilde{N}_K(x,e)$}\label{sec-prop2}
The Proposition \ref{prop1} considered the situation in which the nonlinear terms of $F_K$ were complementary to $K$, which allowed the possibility of effectively canceling out the nonlinear terms. This relies on the dimension of the nonlinearity not exceeding 
the rank of $C$. Here we provide a more general strategy which relies on 
canceling out only those terms 
of $F_K$ which are nonlinear in $e$. This relies on the idea that the nonlinear terms in $F_K$ that are linear in $e$ may be controlled.   Recalling the definition $\psi_j(x,e)=a_j(x+e)-a_j(x)$, we define 
the sets of reactions $\sR_{2,K}$ as follows:
\begin{equation}
\sR_{2,K} = \{ j \, | \, j \notin \sR_{0,K} \text{ and }\psi_j(x,e) = \frac{d a_j}{dx}(x) \, e \quad \forall (x,e) \in \real^n \times K \}.
\end{equation}
Thus, $\sR_{2,K}$ are the reactions $j$ for which the restriction of $\psi_j$ to $\real^n \times K$ is nonzero and linear in $e$ but not necessarily linear in $(x,e)$. 
It is clear that $\sR_{1,K} \subset \sR_{2,K}$. We define the subspace 
$S_{2,K}$ by
\begin{equation}
S_{2,K} = \text{span}\{ \nu_j \, | \, j \notin (\sR_{0,K} \cup \sR_{2,K})\}.   
\end{equation}
Thus $S_{2,K}$ contains the range of the terms in $F_K$ that are nonlinear in $e$. Note that $S_{2,K} \subset S_{1,K}$. 
We may write
\begin{equation}
F_K(x,e) = \sum_{j=1}^r \nu_j \, \frac{d a_j}{d x}(x)\, e + \Tilde{N}_K(x,e)
\end{equation}
where $\Tilde{N}_K(x,e)$ is nonlinear in $e$. 
We note that it may happen that $\Tilde{N}_K$ is zero due to cancellations. Define $B(x) \in \real^{n \times n}$ (for $x \in \real^n$) by
\begin{equation}\label{eq-B}
B(x) = \sum_{j=1}^r \nu_j \, \frac{d a_j}{dx}(x).    
\end{equation}
Denote by $B_K(x)$ the restriction of the linear map $B(x):\real^n \to \real^n$ to $K$. Then we may write
\begin{equation}
F_K(x,e) = B_K(x) \, e + \Tilde{N}(x,e) \quad \forall (x,e) \in \real_+^n \times K.    
\end{equation}
Given a decomposition $\real^n = E \oplus K$ define $\Bar{B}_K(x):K \to K$ 
by $\Bar{B}_K(x) = \Pi_K \, B_K(x)$. 

\begin{proposition}\label{prop2}
Suppose we observe a CRN for which either $S_{2,K} \cap K = \{0\}$ or $\Tilde{N}_K$ is zero. 
Pick $E$ such that $\real^n = E \oplus K$ with $\text{range}(\Tilde{N}_K) \subset E$.   Let $\Gamma \subset \real_+^n$ be a compact set and suppose that there exist finite number of linear maps $V_i:K \to K$ for $i=1,\dots,p$ such that 
$\Bar{B}_K(\Gamma)$ belongs to their convex hull.
Suppose further that 
$V_i$ are Hurwitz for each $i=1,\dots,p$ and the following 
Lyapunov inequalities have a common symmetric positive definite solution $N:K \to K$:
\begin{equation}\label{eq-Lyap-ineq-common}
N \, V_i + {V_i}^T \, N \leq - 2 I, \quad 1 \leq i \leq p.  
\end{equation}
Then the contractivity condition \eqref{eq-contract} is satisfied and hence the design goal is met. 
 \end{proposition}
\begin{proof}
The common solution $N$ defines an inner product $(\cdot,\cdot)_N$ inside $K$ 
such that for all $e \in K$ and for $1 \leq i \leq p$ it holds that
\[
(e, V_i \, e)_N \leq - (e,e) \leq -\frac{1}{\lambda_{\text{max}}(N)} (e,e)_N,
\]
where $\lambda_{\text{max}}(N)$ is the greatest eigenvalue of $N$. 
As in Proposition \ref{prop1} we extend the inner product $(\cdot,\cdot)_N$ defined by $N$ in $K$ to $\real^n$ so that $E \perp_N K$ and the extension coincides with the standard inner product inside $E$. We also observe that for $(x,e) \in \real_+^m \times K$, $(e,\Tilde{N}(x,e))_N=0$.  
Hence for $x \in \Gamma$ and $e \in K$
\[
(e,F_K(x,e)_N = (e, B_K(x) \, e)_N = (e,\Bar{B}_K(x) \, e)_N.  
\]
Moreover, for $x \in \Gamma$, there exist $c_1(x), \dots, c_p(x) \geq 0$ with $\sum_{i=1}^p c_i(x)=1$ such that 
\[
\Bar{B}_K(x) = \sum_{i=1}^p c_i(x) \, V_i.
\]
Hence 
\[
(e,\Bar{B}_K(x) e)_N = \sum_{i=1}^p c_i(x) \, (e,V_i e)_N \leq -\frac{1}{\lambda_{\text{max}}(N)} (e,e)_N
\]
for all $x \in \Gamma$ and all $e \in K$. 
\end{proof}

Proposition \ref{prop2} requires a common (positive definite) solution $N$ to the set of Lyapunov inequalities 
\eqref{eq-Lyap-ineq-common}. A useful sufficient condition in an analytical form 
may be found in \cite{liberzon1999} which we restate here for completeness. 
\begin{lemma}[Theorem 2, \cite{liberzon1999}.]
Let $V_i \in \real^{d \times d}$ for $i=1,\dots,p$ be Hurwitz. Suppose further that
the Lie algebra generated by $\{V_i\}$ is solvable (see \cite{liberzon1999} for a definition). Then, the following system of Lyapunov inequalities have a common 
symmetric positive definite solution $N \in \real^{d \times d}$:
\[
N \, V_i + {V_i}^T \, N \leq -2 I, \quad 1 \leq i \leq p.
\]
\end{lemma}

This Lie Algebraic  approach is attractive since the matrices $V_i$ will be functions of the parameters $k_j$ and one may be able to design an observer that works for all choices of $k_j>0$. However, since this is not a necessary condition, if the Lie algebra is not solvable, we cannot conclude that there is no common solution to \eqref{eq-Lyap-ineq-common}. 
We also note that, the problem of determining whether 
there is a (positive definite) solution $N$ to \eqref{eq-Lyap-ineq-common} for a given set of numerical square Hurwitz matrices $V_i$ reduces to the solution of a convex minimax problem with linear constraints \cite{horisberger1976regulators} and MATLAB tools are available under linear matrix inequalities (LMI). See also \cite{boyd1994linear}. This approach is useful when we have specific numerical values for the parameters $k_j$.

\paragraph{Mass action CRNs where concentrations of a subset of species are observed}In this case, finding the set $\sR_{2,K}$ and determining the set of matrices $V_i$ is 
relatively simple as we shall illustrate. 
From \eqref{eq-psij-mass-action} we can conclude that 
$j \in \sR_{2,K}$ (that is $\psi_j(x,e)$ is nonzero and is linear in $e$) if and only if $|\nu_{j,u}^-| = 1$. 
 In other words, 
$\sR_{2,K}$ consists of reactions for which there is exactly one unobserved reactant specie and its stoichiometric coefficient is $1$. We also have that for $j \in \sR_{2,K}$ 
\[
\psi_j(x,e) = k_j \, {x_o}^{\nu_{j,o}^-} \, {e_u}^{\nu_{j,u}^-}.
\]
Then $B_K(x) \, e$, the part of $F_K(x,e)$ that is linear in $e$ 
 defined earlier will be given by
\[
B_K(x) \, e = \sum_{j \in \sR_{2,K}} \nu_j \, k_j \, {x_o}^{\nu_{j,o}^-} \, {e_u}^{\nu_{j,u}^-} \quad \forall e \in K.
\]
We may write $B_K(x)$ in terms of $A_K$ as follows:
\[
B_K(x) \, e = A_K \, e + \sum_{j \in \sR_{2,K} \setminus \sR_{1,K}} x_o^{\nu_{j,o}^-} \, \nu_j \, k_j \, {e_u}^{\nu_{j,u}^-}.
\]
We define the linear maps $B_{j,K}:K \to \real^n$ 
by 
\[
B_{j,K} \, e = \sum_{j \in \sR_{2,K} \setminus \sR_{1,K}} \nu_j \, k_j \, {e_u}^{\nu_{j,u}^-}.
\]
Defining $\Bar{B}_{j,K}=\Pi_K \, B_{j,K}$, we may write 
\[
\Bar{B}_K(x) = \Bar{A}_K + \sum_{j\in \sR_{2,K} \setminus \sR_{1,K} } x_o^{\nu_{j,o}^-} \, \Bar{B}_{j,K}.
\]
Given any compact set $\Gamma \subset \real_+^n$, 
we may bound each species $i$ in an interval $[\ell_i,u_i]$. We may write the vectors $\ell$ and $u$ as $\ell=(\ell_o,\ell_u)$ and $u=(u_o,u_u)$. 
It follows that for $x \in \Gamma$, $x_o^{\nu_{j,o}^-}$ 
is a convex combination of $\ell_o^{\nu_{j,o}^-}$ 
and $u_o^{\nu_{j,o}^-}$. Consequently, the convex hull of the finite set 
\begin{equation}\label{eq-setV}
\mathcal{V} = \left\{ \Bar{A}_K + \sum_{j \in \Tilde{\sR}_{2,K}} m_j \, \Bar{B}_{j,K} \, \Big{|} \, m_j \in \{\ell_o^{\nu_{j,o}^-}, u_o^{\nu_{j,o}^-}\}\right\},
\end{equation}
contains $\Bar{B}_K(\Gamma)$. Here $\Tilde{\sR}_{2,K}=\sR_{2,K} \setminus \sR_{1,K}$. We observe that $\mathcal{V}$ contains at most $2^k$ points where $k$ is the number of reactions in $\Tilde{\sR}_{2,K}$. Frequently, two or more reactions in $\Tilde{\sR}_{2,K}$ may share the same monomial $x_o^{\nu_{j,o}^-}$, leading to a reduction in the number of elements needed to form the convex set.

\subsection{Examples for application of Proposition \ref{prop2}}\label{sec-prop2-eg}
\paragraph{Example 3: Oscillator}\mbox
This system relies on the interactions of five chemical species $X_1$, $X_2$, $X_3$, $X_4$ and $X_5$  governed by the following eight reactions:
        \begin{align*}
			\emptyset &\xrightarrow{k_1} X_1  \\
			X_1+2X_2 &\xrightarrow{k_2} 2X_3  \\
			X_3 &\xrightarrow{k_3} X_4  \\
			X_4 &\xrightarrow{k_4} X_2 \\
			 X_2 &\xrightarrow{k_5} X_5 \\
			X_1 + X_4 &\xrightarrow{k_6} X_3+X_5 \\
			X_1+ X_5 &\xrightarrow{k_7} X_3 \\
			X_5 &\xrightarrow{k_8} \emptyset
		\end{align*}
where $k_j$ denotes the rate constant for $j$th reaction and $k_j>0$ for $j\in \{1,2,...8\}$.       
	Stoichiometric vectors $(\nu_j)$ and assuming mass action form the rate functions $(a_j(x))$ are given by	
		\begin{align*}
			\nu_1 &= \begin{pmatrix} 1 \\ 0 \\ 0 \\ 0 \\ 0 \end{pmatrix}, &
			\nu_2 &= \begin{pmatrix} -1 \\ -2 \\ 2 \\ 0 \\ 0 \end{pmatrix}, &
			\nu_3 &= \begin{pmatrix} 0 \\ 0 \\ -1 \\ 1 \\ 0 \end{pmatrix}, &
			\nu_4 &= \begin{pmatrix} 0 \\ 1 \\ 0\\ -1 \\ 0 \end{pmatrix}, \\
			\nu_5 &= \begin{pmatrix} 0 \\ -1 \\ 0 \\ 0 \\ 1 \end{pmatrix}, &
			\nu_6 &= \begin{pmatrix} -1 \\ 0 \\ 1 \\ -1 \\ 1 \end{pmatrix}, &
			\nu_7 &= \begin{pmatrix} -1 \\ 0 \\ 1 \\ 0 \\ -1 \end{pmatrix}, &
			\nu_8 &= \begin{pmatrix} 0 \\ 0 \\ 0 \\ 0 \\ -1 \end{pmatrix},
		\end{align*}
		
		\begin{align*}
			a_1 &= k_1 , & a_2 &= k_2 x_1 x_2^2, & a_3 &= k_3 x_3, & a_4 &= k_4 x_4, \\
			a_5 &= k_5 x_2, & a_6 &= k_6 x_1 x_4, & a_7 &= k_7 x_1 x_5, & a_8 &= k_8 x_5
		\end{align*}
    where $x_i(\geq 0)$ is the concentration of species $X_i$ with $i\in\{1,2,...5\}$.    
        \paragraph{Observation model 1}
		Suppose that we observe concentrations of species $X_1$ and $X_3$ i.e. $x_1$ and $ x_3 $. 
	Then we have
		\[
		K = \text{span} \left\{ \begin{pmatrix} 0 \\ 1 \\ 0 \\ 0 \\ 0 \end{pmatrix}, \begin{pmatrix} 0 \\ 0 \\ 0 \\ 1 \\ 0 \end{pmatrix}, \begin{pmatrix} 0 \\ 0 \\ 0 \\ 0 \\ 1 \end{pmatrix} \right\} := \{v_1, v_2, v_3\},
        \]
where $v_1, v_2$ and $v_3$ defined. We note that $\sR_{0,K}=\{1, 3\}$ since reactions 1 and 3 do not have unobserved reactant species. Also $\sR_{1,K} = \{4, 5, 8\}$ 
as $4, 5$ and $8$ are the only reactions with exactly one specie as a reactant 
and that specie is unobserved and has stoichiometric coefficient $1$. We also note that $\sR_{2,K}=\{4, 5, 6, 7, 8\}$ since $4, 5, 6, 7$ and $8$ are the only reactions that have exactly one 
unobserved specie with stoichiometric coefficient $1$ as a reactant. 
Hence $S_{2,K} = \text{span}\{\nu_2\}$.   
Here we can see that $K \cap \ S_{2,K}={0}$. 
Since $\dim E=2$, there is freedom in choosing $E$ subject to $S_{2,K} \subset E$. We discuss two possibilities. 
     \paragraph{Observer 1 (for observation model 1)} We choose   
		\[
		E = \text{span} \left\{ \begin{pmatrix} -1 \\ -2 \\ 2 \\ 0 \\ 0 \end{pmatrix}, \begin{pmatrix} 1 \\ -2 \\ 0 \\ 0 \\ 0 \end{pmatrix} \right\} := \{v_4, v_5\}
		\]
where $v_4$ and $v_5$ are defined. 
Let $A_K \, e$ be the linear part of $F_K(x,e)$. Then for $e \in K$, $A_K \, e$ is obtained by considering the reactions in $\sR_{1,K}=\{4, 5, 8\}$. This leads to   
\[
A_K \, e = \nu_4 \, k_4 \, e_4 + \nu_5 \, k_5 \, e_2 + \nu_8 \, k_8 \, e_5,
\]
for $e \in K$. 
Moreover, letting $B_K(x) \, e$ be the part 
of $F_K(x,e)$ that is linear in $e$ for $e \in K$, then 
\[
B_K(x) \, e = A_K \, e + \nu_6 \, k_6 \, x_1 e_4 + \nu_7 \, k_7 \, x_1 e_5.
\]
Thus, for $e \in K$, we can write $B_K(x) e = A_K \, e + x_1 \, B_{1,K} \, e$ where 
\[
B_{1,K} \, e = \nu_6 \, k_6 \, e_4 + \nu_7 \, k_7 \,  e_5.
\]
From the definition of $\Bar{B}_K$ it follows that 
$\Bar{B}_K(x) = \Bar{A}_K + x_1 \, \Bar{B}_{1,K}$ 
where $\Bar{B}_{1,K} = \Pi_K \, B_{1,K}$ and $\Bar{A}_K = \Pi_K A_K$. In order to obtain the matrix representations 
of $\Bar{A}_K$ and $\Bar{B}_{1,K}$ with respect to the basis $\{v_1, v_2, v_3\}$ for $K$, we note the following calculations:
\[
\begin{aligned}
A_K v_1 &= -k_5 v_1 + k_5 v_3, \quad A_K v_2 = k_4 v_1 - k_4 v_2,\quad 
A_K v_3 = -k_8 v_3,\\
B_{1,K} v_1 &= 0, \quad
B_{1,K} v_2 = -k_6 v_2 + k_6 v_3 + \frac{k_6}{2} v_4 - \frac{k_6}{2} v_5, \quad
B_{1,K} v_3 = -k_7 v_3 + \frac{k_7}{2} v_4 - \frac{k_7}{2} v_5.
\end{aligned}
\]

Hence, in matrix form (w.r.t.\ the basis $\{v_1, v_2, v_3\}$ for $K$)
		\[
		\bar{A}_{K} = \begin{pmatrix}
			-k_5 & k_4 & 0 \\
			0 & -k_4 & 0 \\
			k_5 & 0 & -k_8
		\end{pmatrix},
		\]
		\[
		\bar{B}_{1K} = \begin{pmatrix}
			0 & 0 & 0 \\
			0 & -k_6 & 0 \\
			0 & k_6 & -k_7
		\end{pmatrix}.
		\]
Assuming the solution lies in a compact set $\Gamma$, 
we may choose an upper bound $u$ for $x_1(t)$ and take $0$ as its lower bound. Then, $\Bar{B}_K(x)$ for $x \in \Gamma$ lies in the convex combination of $\Bar{A}_K$ 
and $\Bar{A}_K + u \, \Bar{B}_{1,K}$. 
It is easy to verify that $\bar{A}_{K}$ and $(\bar{A}_{K}+u\bar{B}_{1K})$ are Hurwitz for all parameter values ($k_j>0$). It is shown that they generate a solvable Lie algebra in Lemma \ref{lie_lem}. Hence the conditions of Proposition \ref{prop2} are satisfied and 
  the design goal is met.       

        \paragraph{Observer 2 (for observation model 1)} If we choose 
        \[
		E = \text{span} \left\{ \begin{pmatrix} -1 \\ -2 \\ 2 \\ 0 \\ 0 \end{pmatrix}, \begin{pmatrix} -1 \\ 0 \\ 1 \\ 0 \\ 0 \end{pmatrix} \right\}:=\{v_4,v_5'\},
		\]
        where $v_4$ and $v_5'$ are defined. Changing $E$ does not change $B_K(x)$, $A_K$ and $B_{1,K}$. However, it changes the projection $\Pi_K$ and hence
        may change $\Bar{A}_K$ and $\Bar{B}_{1,K}$. 
                
        Now, the calculations to derive $\bar{A}_K$ and $\bar{B}_{1,K}$ are,
        \[
        A_K v_1 = -k_5 v_1 + k_5 v_3, \ A_K v_2 = k_4 v_1 - k_4 v_2,\
        A_K v_3 = -k_8 v_3,
        \]
and        
        \[
        B_K v_1 = 0, \
        B_K v_2 = -k_6 v_2 + k_6 v_3 + k_6 v_5', \
        B_K v_3 = -k_7 v_3 + k_7 v_5'.
        \]
        Hence we get the same $\bar{A}_K$ and $\bar{B}_{1K}$ as in Observer 1. So, with the same argument stated in case of Observer 1, 
        the conditions of Proposition \ref{prop2} are satisfied and the design goal is met. In Section \ref{sec_num_res} we see that while both observers achieve error convergence, the second observer seems to perform better. 
        
\paragraph{Observation model 2} If we observe the concentrations of species $X_1$ and $X_2$ i.e., $x_1$ and $x_2$, we have
\[
		K = \text{span} \left\{ \begin{pmatrix} 0 \\ 0 \\ 1 \\ 0 \\ 0 \end{pmatrix}, \begin{pmatrix} 0 \\ 0 \\ 0 \\ 1 \\ 0 \end{pmatrix}, \begin{pmatrix} 0 \\ 0 \\ 0 \\ 0 \\ 1 \end{pmatrix} \right\}:= \{ v_1, v_2, v_3\}
        \]
where $v_1$, $v_2$, $v_3$ are defined and $S_{2,K} = \{0\}$ since $|\nu_{j,u}^+|\leq 1$ for all the reactions $j$. 
     Hence $K \cap \ S_{2,K}={0}$. There is freedom in choosing $E$. We choose   
		\[
		E = \text{span} \left\{ \begin{pmatrix} 1 \\ 0 \\ -1 \\ 0 \\ 0 \end{pmatrix}, \begin{pmatrix} 0 \\ 1 \\ 0 \\ 0 \\ 0 \end{pmatrix} \right\}:=\{v_4,v_5\}
		\]
where $v_4$, $v_5$ are defined in this way and clearly $S_{2,K} \subset E$. 
With this observation model we get $\sR_{0,K}=\{1,2,5\}$, $\sR_{1,K}=\{3, 4, 8\}$ 
and $\sR_{2,K} = \{3, 4, 6, 7, 8\}$. 
Let $A_K \, e$ be the linear part of $F_K(x,e)$. Then for $e \in K$, $A_K \, e$ is obtained by considering the reactions in $\sR_{1,K}=\{3, 4, 8\}$. This leads to   
\[
A_K \, e = \nu_3\,k_3\,e_3 +\nu_4 \, k_4 \, e_4 + \nu_5 \, k_5 \, e_2 + \nu_8 \, k_8 \, e_5,
\]
for $e \in K$. 
Moreover, letting $B_K(x) \, e$ be the part 
of $F_K(x,e)$ that is linear in $e$ for $e \in K$, then by considering the reactions in $\sR_{2,K}$ we obtain 
\[
B_K(x) \, e = A_K \, e + \nu_6 \, k_6 \, x_1 e_4 + \nu_7 \, k_7 \, x_1 e_5.
\]
Thus, for $e \in K$, we can write $B_K(x) e = A_K \, e + x_1 \, B_{1,K} \, e$ where 
\[
B_{1,K} \, e = \nu_6 \, k_6 \, e_4 + \nu_7 \, k_7 \,  e_5.
\]
From the definition of $\Bar{B}_K$ it follows that 
$\Bar{B}_K(x) = \Bar{A}_K + x_1 \, \Bar{B}_{1,K}$ 
where $\Bar{B}_{1,K} = \Pi_K \, B_{1,K}$ and $\Bar{A}_K = \Pi_K A_K$. In order to obtain the matrix representations 
of $\Bar{A}_K$ and $\Bar{B}_{1,K}$ with respect to the basis $\{v_1, v_2, v_3\}$ for $K$, we note the following calculations:
\[
\begin{aligned}
A_K v_1 &= -k_3 v_1 + k_3 v_2, \quad A_K v_2 = -k_4 v_2 + k_4 v_5,\quad 
A_K v_3 = -k_8 v_3,\\
B_K v_1 &= 0, \quad
B_K v_2 = -k_6 v_2 + k_6 v_3 - k_6 v_4, \quad
B_K v_3 = -k_7 v_3 - k_7 v_4.
\end{aligned}
\]

Then we get the matrix representations w.r.t.\ the basis $\{v_1, v_2, v_3\}$ for $K$: 
		\[
		\bar{A}_{K} = \begin{pmatrix}
			-k_3 & 0 & 0 \\
			k_3 & -k_4 & 0 \\
			0 & 0 & -k_8
		\end{pmatrix}, \ \
		\bar{B}_{1,K} = \begin{pmatrix}
			0 & 0 & 0 \\
			0 & -k_6 & 0 \\
			0 & k_6 & -k_7
		\end{pmatrix}.
		\]

        As we assume the solution stays in a compact set $\Gamma$, we can say that 
        $0 \leq x_1 \leq u$, where $u$ is the upper bound for $x_1(t)$. So, $\bar{B}_K(x) $ stays in the convex combination of $\bar{A}_{K}$ and $(\bar{A}_{K}+u\bar{B}_{1,K})$.
		$\bar{A}_{K}$ and $(\bar{A}_{K}+u\bar{B}_{1,K})$ are Hurwitz and lower triangular matrices. So, they have solvable Lie algebra \cite{humphreys1978}. So, the system of Lyapunov inequalities stated in equation \ref{eq-Lyap-ineq-common} with $V_1 = \bar{A}_{K}$ and $V_2 = (\bar{A}_{K}+u\bar{B}_{1,K})$ has a common symmetric positive definite solution $N$. Hence the conditions of Proposition \ref{prop2} are satisfied and our required design goal is attained.

\paragraph{Example 4:}\mbox
		This system consists of three species $X_1$, $X_2$, and $X_3$ involved in six reactions given below:
		\begin{align*}
			& X_1 + X_2 \xrightarrow{k_1} X_1+X_3  \\
			& X_1 + X_3 \xrightarrow{k_2} 2X_1  \\
			& 2X_1 \xrightarrow{k_3} X_1 + X_2  \\
			& X_1  \xrightarrow{k_4} X_2  \\
			& X_2 \xrightarrow{k_5} X_3  \\
            & X_3  \xrightarrow{k_6} X_1
		\end{align*}

   where $k_j$ denotes the rate constant for $j$th reaction in this reaction network and $k_j>0$ for $j\in\{1,2,...,6\}$. Here the
stoichiometric vectors $(\nu_j)$ and rate functions $(a_j(x))$ using the mass action kinetics are
             \begin{align*}
			\nu_1 &= \begin{pmatrix} 0 \\ -1 \\ 1  \end{pmatrix}, &
			\nu_2 &= \begin{pmatrix} 1 \\ 0 \\ -1  \end{pmatrix}, &
			\nu_3 &= \begin{pmatrix} -1 \\ 1 \\ 0  \end{pmatrix}, &
			\nu_4 &= \begin{pmatrix} -1 \\ 1 \\ 0 \end{pmatrix}, \\
			\nu_5 &= \begin{pmatrix} 0 \\ -1 \\ 1  \end{pmatrix}, &
			\nu_6 &= \begin{pmatrix} 1 \\ 0 \\ -1 \end{pmatrix},\\
		  a_1 &= k_1 x_1 x_2, & a_2 &= k_2 x_1 x_3, & a_3 &= k_3 x_1^2, & a_4 &= k_4 x_1, \\
			a_5 &= k_5 x_2, & a_6 &= k_6 x_3
		\end{align*}
         where $x_i(\geq 0)$ is the concentration of species $X_i$ with $i\in\{1,2,3\}$.    
Therefore, by mass-action kinetics, the system of differential equations becomes
\begin{align*}
    \dot{x}_1 &=   k_2  x_1  x_3 - k_3  x_1^2 - k_4  x_1 + k_6  x_3 ,\\
    \dot{x}_2 &=  - k_1  x_1  x_2  - k_5  x_2 + k_3  x_1^2 + k_4  x_1,\\
    \dot{x}_3 &=  k_1  x_1  x_2 - k_2  x_1  x_3 + k_5  x_2 - k_6  x_3 .
\end{align*}
Here, $x_i \in \mathbb{R}^+$ for each $i\in \{1,2,3\}$. Since $[1 \; 1\; 1] \, \nu_j=0$ for all $j$, $\dot{x}_1 + \dot{x}_2 + \dot{x}_3 = 0$. So, we have
\[ x_1 + x_2 + x_3 = \text{constant} = x_1(0) + x_2(0) + x_3(0), \]
and thus the solution lies in a compact set since $x_i(t) \geq 0$. 

Suppose we observe $x_1(t)$ i.e., the concentration of species $X_1$. Then we have

		\[
			K = \text{span} \left\{ \begin{pmatrix} 0 \\ 1 \\ 0  \end{pmatrix}, \begin{pmatrix} 0 \\ 0 \\ 1 \end{pmatrix} \right\}:=\{v_1, v_2\}  
        \]
     where $v_1$ and $v_2$ are defined and $S_{2,K}=\{0\}$ since $|\nu_{j,u}^+| \leq 1$  for all reactions $j$. 
We also note that $\sR_{0,K}=\{3, 4\}$ , $\sR_{1,K} = \{5, 6\}$ and $\sR_{2,K} = \{1, 2, 5, 6\}$. 
Since $S_{2,K}$ is trivial, there is freedom in choosing $E$ subject to $E \cap K=\{0\}$. We explore two choices. 

\paragraph{Observer 1} We choose
        \[
			E = \text{span} \left\{ \begin{pmatrix} 1 \\ 0 \\ 0 \end{pmatrix} \right\}:=\{v_3\}
		\]
	where $v_3$ is defined.	Since $E = K^\perp$, this choice is the same as the usual nudging. 
    Considering $A_K \, e$ as the linear part of $F_K(x,e)$, we can get  $A_K \, e$ by using the reactions in $\sR_{1,K}=\{  5, 6\}$. This gives   
\[
A_K \, e = \nu_5\,k_5\,e_2 +\nu_6 \, k_6 \, e_3, 
\]
for $e \in K$. 
Moreover, considering $B_K(x) \, e$ as the part 
of $F_K(x,e)$ that is linear in $e$ for $e \in K$, then using the reactions in $\sR_{2,K}$ we obtain 
\[
B_K(x) \, e = A_K \, e + \nu_1 \, k_1 \, x_1 e_2 + \nu_2 \, k_2 \, x_1 e_3.
\]
Thus, for $e \in K$, we can write $B_K(x) e = A_K \, e + x_1 \, B_{1,K} \, e$ where 
\[
B_{1,K} \, e = \nu_1 \, k_1 \, e_2 + \nu_2 \, k_2 \,  e_3.
\]
From the definition of $\Bar{B}_K$ it follows that 
$\Bar{B}_K(x) = \Bar{A}_K + x_1 \, \Bar{B}_{1,K}$ 
where $\Bar{B}_{1,K} = \Pi_K \, B_{1,K}$ and $\Bar{A}_K = \Pi_K A_K$. In order to obtain the matrix representations 
of $\Bar{A}_K$ and $\Bar{B}_{1,K}$ with respect to the basis $\{v_1, v_2\}$ for $K$, we note the following calculations:
\[
\begin{aligned}
A_K v_1 &= -k_5 v_1 + k_5 v_2, \quad A_K v_2 = -k_6 v_2 + k_6 v_3,\\ 
B_K v_1 &= -k_1 v_1 + k_1 v_2, \quad
B_K v_2 = -k_2 v_2 + k_2 v_3.
\end{aligned}
\]
So, we get the matrix representations 
\[
			\bar{A}_{K} = \begin{pmatrix} -k_5 & 0 \\ k_5 & -k_6 \end{pmatrix} , \quad 
			\bar{B}_{1,K} = \begin{pmatrix} -k_1 & 0 \\ k_1 & -k_2 \end{pmatrix}.
\]
Let the solution stay in a compact set $\Gamma$. Then we can say that $0 \leq  x_1 \leq u$, where $u$ is the upper bound for $x_1$. So, $\bar{B}_K(x) $ stays in the convex combination of $\bar{A}_{K}$ and $(\bar{A}_{K}+u\bar{B}_{1,K})$.
		$\bar{A}_{K}$ and $(\bar{A}_{K}+u\bar{B}_{1,K})$ are Hurwitz and lower triangular matrices. Hence they generate a solvable Lie algebra \cite{humphreys1978}. Therefore the system of Lyapunov inequalities stated in equation \ref{eq-Lyap-ineq-common} with $V_1 = \bar{A}_{K}$ and $V_2 = (\bar{A}_{K}+u\bar{B}_{1,K})$ has a common symmetric positive definite solution $N$. Hence the conditions of Proposition \ref{prop2} are satisfied and our required design goal is attained.

\paragraph{Observer 2} Next, we choose
        \[
			E = \text{span} \left\{ \begin{pmatrix} 1 \\ 0 \\ 1 \end{pmatrix} \right\}:={v_3'}
		\]
        where $v_3'$ is defined.

As the subspace $E$ is changed, the calculations for obtaining $\bar{A}_K$ and $\bar{B}_{1,K}$ with respect to the basis $\{v_1, v_2\}$ for $K$ will change. We show 
the computations here:
\[
\begin{aligned}
A_K v_1 &= -k_5 v_1 + k_5 v_2, \quad A_K v_2 = -2k_6 v_2 + k_6 v_3',\\
B_K v_1 &= -k_1 v_1 + k_1 v_2, \quad
B_K v_2 = -2k_2 v_2 + k_2 v_3'.
\end{aligned}
\]

        Then we get the matrix representations 
		\[
			\bar{A}_{K} = \begin{pmatrix} -k_5 & 0 \\ k_5 & -2k_6 \end{pmatrix} ,\ \ 
			\bar{B}_{1K} = \begin{pmatrix} -k_1 & 0 \\ k_1 & -2k_2 \end{pmatrix}. \\
		\]
	Given a compact set $\Gamma$ in which the solution lies, there exists $u>0$ such that $0 \leq x_1 \leq u$.  Thus, $\bar{B}_K(x) $ stays in the convex combination of $\bar{A}_{K}$ and $(\bar{A}_{K}+u\bar{B}_{1,K})$.
		$\bar{A}_{K}$ and $(\bar{A}_{K}+u\bar{B}_{1,K})$ are Hurwitz and lower triangular matrices. Hence they generate a solvable Lie algebra \cite{humphreys1978}. Thus, the system of Lyapunov inequalities stated in equation \ref{eq-Lyap-ineq-common} with $V_1 = \bar{A}_{K}$ and $V_2 = (\bar{A}_{K}+u\bar{B}_{1,K})$ has a common symmetric positive definite solution $N$. Hence the conditions of Proposition \ref{prop2} are satisfied and our required design goal is met.

\section{Numerical results}\label{sec_num_res}
In this section we demonstrate the efficiency of our proposed framework through numerical simulations of the examples discussed in sections \ref{sec-prop1-eg} and \ref{sec-prop2-eg}. These are performed in MATLAB with ODE solver {\tt ode45}. 
In Section \ref{sec-noisy}, we include simulations of the application of our proposed observer to the case of noisy observations of Example 2, the chaotic WR model. We compare the performance of our observer with that of a particle filter for state estimation. 

\begin{remark}\label{rem-num-acc}
When the time trajectory of the norm of the observer error $|e(t)|$ is plotted, the error typically decays down to very small values and then 
fluctuates around a constant value. When these fluctuations are at values below $10^{-12}$ (often below $10^{-15}$) we take it as an indication that the error is within the numerical accuracy of the solver. In some examples, the error becomes numerically zero after some time point and remains zero. In the log scale plots the decay of the error is seen more clearly 
and the plots end abruptly if the error reaches (numerical) zero. 
\end{remark}

\subsection{Simulation for Example $1$}
Here we consider Lotka-Volterra model with rate constants $k_1= 0.3,k_2=1,k_3=1$ where $k_j$ is the rate constant for $j$th reaction.
\begin{figure}[htbp]
     \centering
     \includegraphics[width=0.4\linewidth]{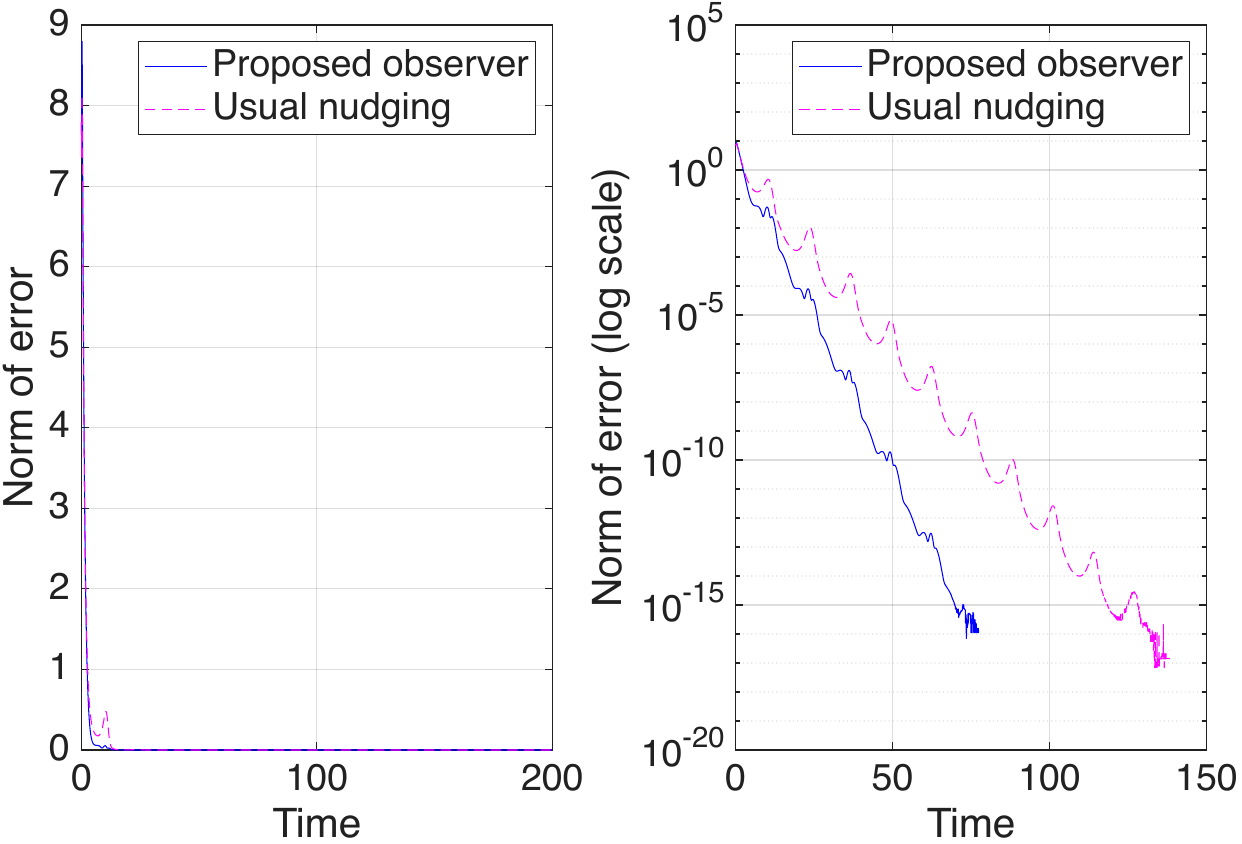}
     \caption{Comparison between the errors of the proposed observer and the usual nudging with $ x_0=[0.5, 0.7]^T, \ z_0 = [7, 5]^T $ and $\mu = 1$: Lotka-Volterra model.
     In the log scale plot, the error of the proposed observer is not shown after some time point because it is zero (numerically).}
     \label{lv1}
\end{figure}
\begin{figure}[htbp]
     \centering
     \includegraphics[width=0.4\linewidth]{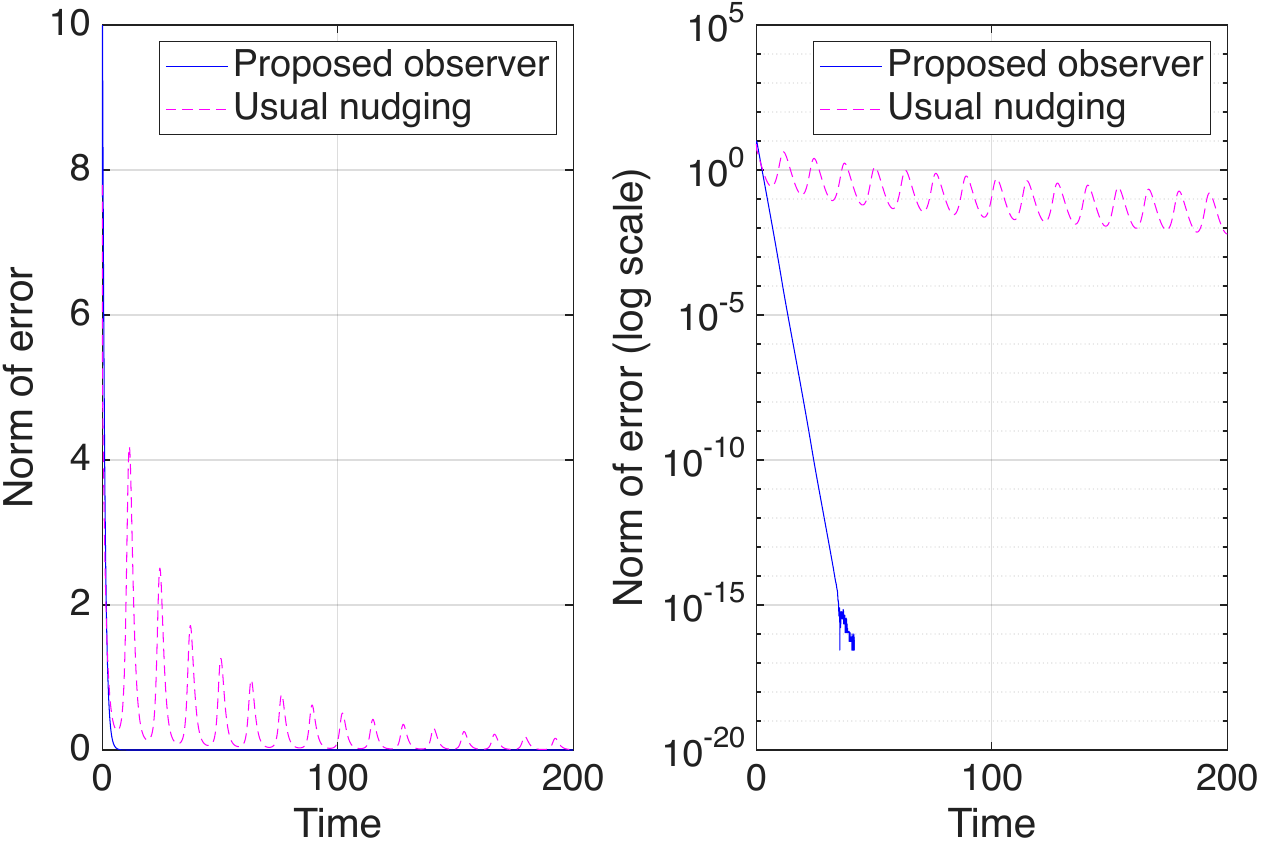}
     \caption{Comparison between the errors of the proposed observer and the usual nudging with $ x_0=[0.5, 0.7]^T, \ z_0 = [7, 5]^T $ and $\mu = 30$: Lotka-Volterra model.
     In the log scale plot, we note (numerical) zero values 
     of the error norm are not shown, abruptly ending the plot.}
     \label{lv2}
\end{figure}
\begin{figure}[htbp]
     \centering
     \includegraphics[width=0.4\linewidth]{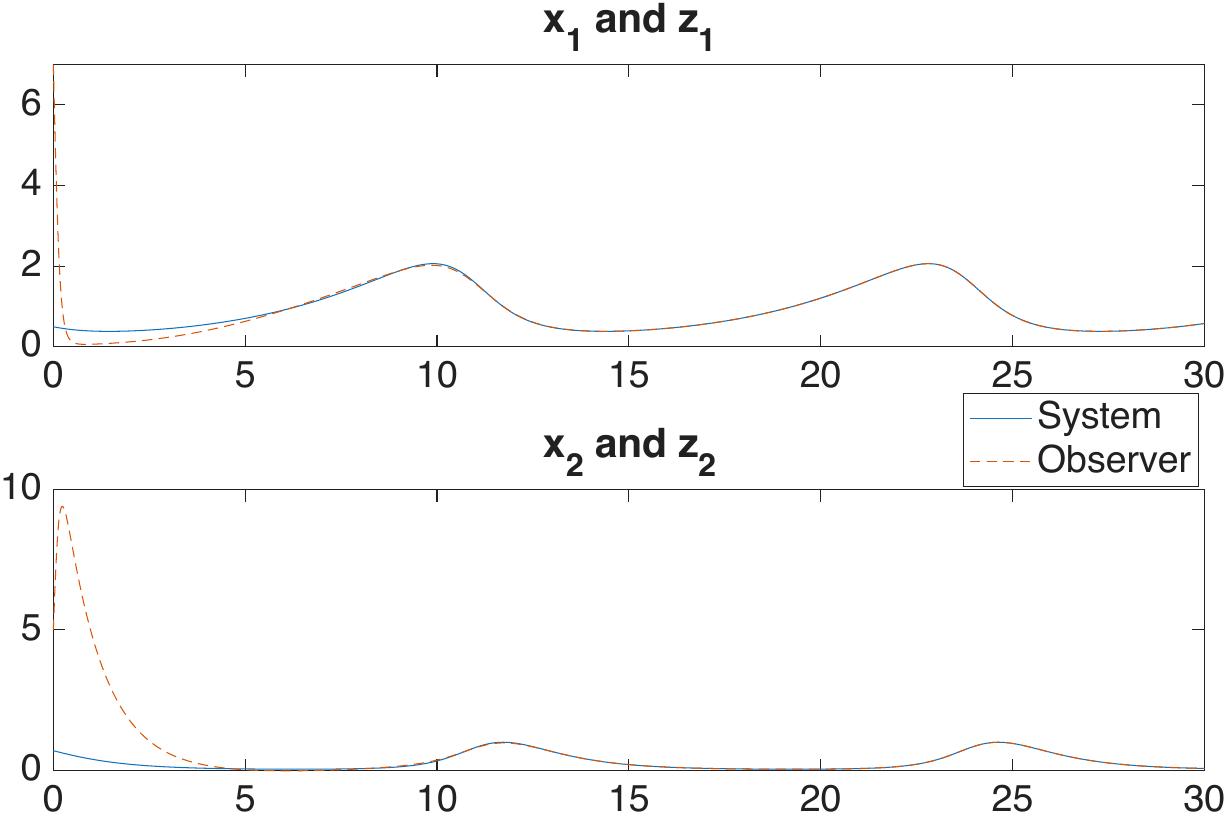}
     \caption{Comparison between each component of state($x$)and observer($z$) $ x_0=[0.5, 0.7]^T, \ z_0 = [7, 5]^T $ and $\mu = 1$: Lotka-Volterra model with proposed observer. }
     \label{lv3}
\end{figure}

Figures \ref{lv1} and \ref{lv2} compare the convergence of the error (in Euclidean norm) of both the proposed observer and the usual nudging for the parameter values $\mu=1$ and $\mu=30$. Both methods achieve convergence with $\mu=1$ and $\mu=30$. But we can see that the proposed observer attains a faster exponential decay of the error. The plots in logarithmic scale show this more clearly. 
More importantly, we note that the usual nudging observer performs poorly for the larger parameter 
value $\mu=30$ making the usual nudging observer unreliable.  
 Figure \ref{lv3} shows the components of the trajectories of the proposed observer ($\mu=1$) and the system. We can see that the estimated trajectories $z_1$ and $z_2$ capture the oscillatory system dynamics $x_1$ and $x_2$ quickly even with large initial error.

 We also note that the exponential decay rate of the error for the proposed observer with $\mu=30$ may be roughly estimated from Figure \ref{lv2} to be around $1$. 
 The lower bound for $\alpha_0$ from Proposition \ref{prop1} for this example is $k_3=1$
as mentioned earlier. For large $\mu$ the decay rate is expected to be $\alpha_0$ or greater, and this is consistent with the theory. 

\subsection{Simulation for Example $2$}
In this example we apply the proposed observer to the chaotic Willamowski-Rössler model with parameters (rate constants) taken from \cite{gaspard2005rossler}:
$   
    k_1 = 30,
    k_2 = 0.5,
    k_3 = 1,
    k_4 = 0.001,
    k_5 = 10,
    k_6 = 0.001,
    k_7 = 1,
    k_8 = 0.001,
    k_9 = 16.5 \ \ \text{and}\ \ 
    k_{10} = 0.5
$.

\begin{figure}[htbp]
     \centering
     \includegraphics[width=0.5\linewidth]{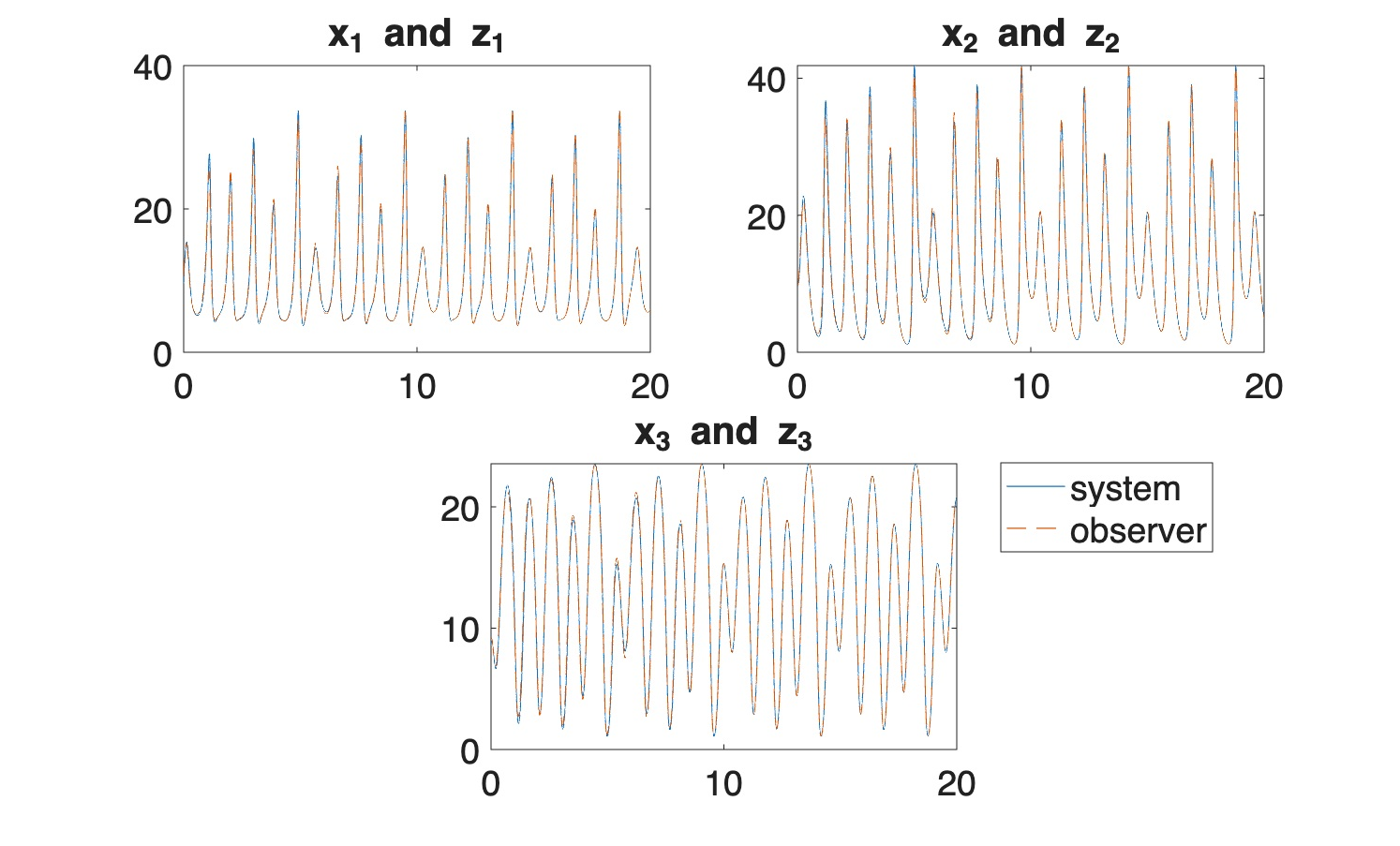}
     \caption{Comparison between each component of state ($x$) and observer ($z$) with $ x_0=[9.61, 9.66, 8.55]^T, \ z_0 = [10.53, 9.70, 9.04]^T $ and $\mu = 1$ : WR model with proposed observer.}
     \label{wr_each species}
\end{figure}

\begin{figure}[htbp]
     \centering
     \includegraphics[width=0.4\linewidth]{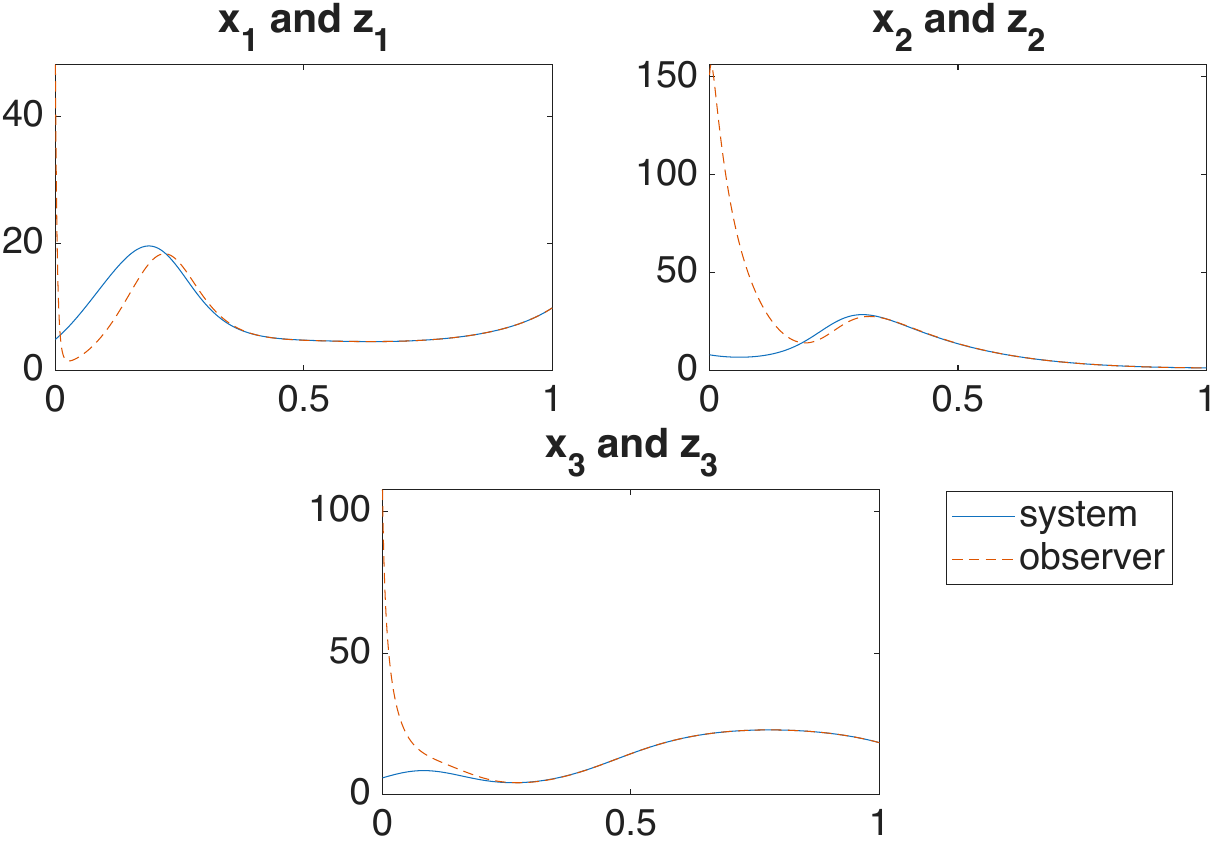}
     \caption{ Comparison between each component of system and observer $x_0=[4.95, 8.17, 5.97]^T, \ z_0=[48.18, 150.97, 107.79]^T$ and $\mu=30 $ : WR model with proposed observer}
     \label{wr_each species_large error}
\end{figure}

\begin{figure}[htbp]
     \centering
     \begin{subfigure}[b]{0.35\textwidth}
     \includegraphics[width=\textwidth]{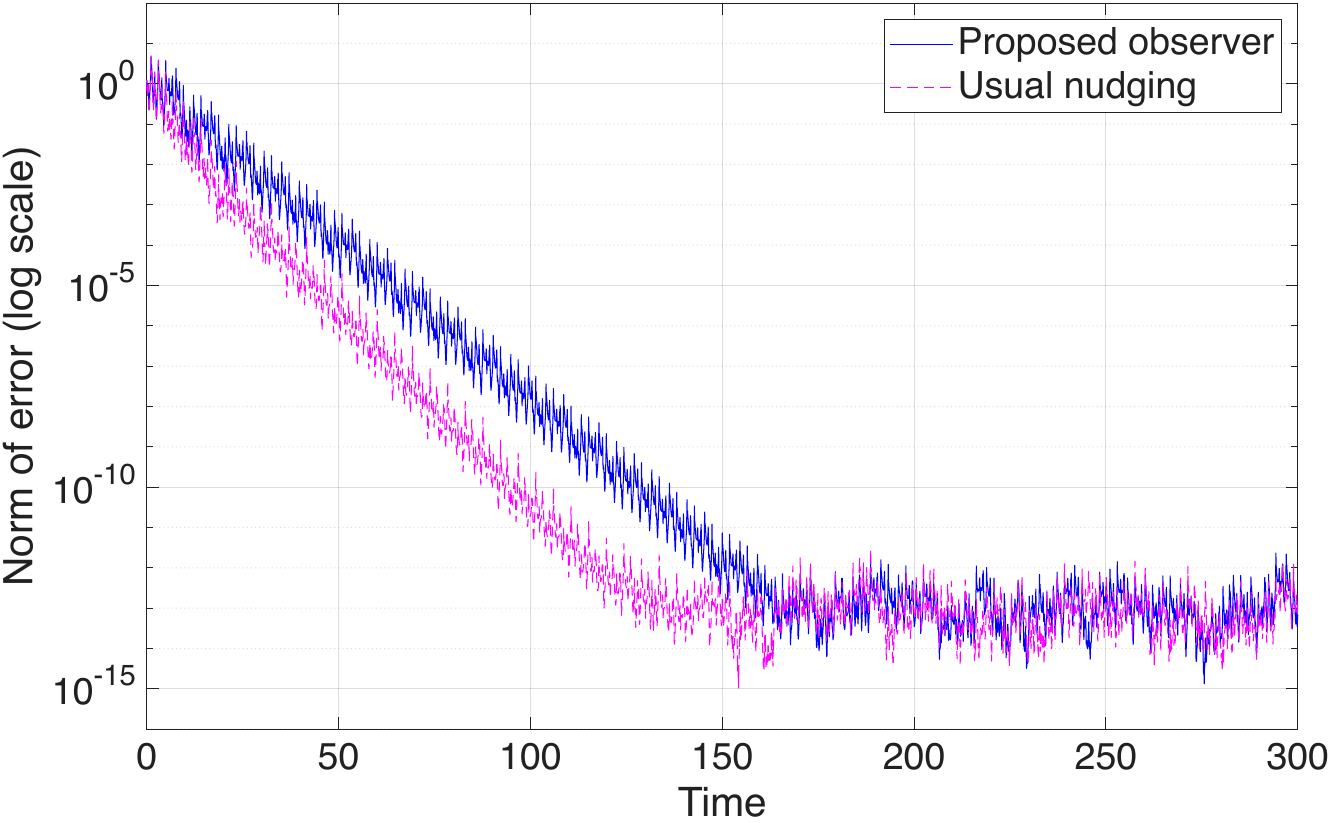}
     \caption{ $\mu=1 $  }
     \label{wr_mu1}
     \end{subfigure}
     \hspace{2cm}
    \begin{subfigure}[b]{0.35\textwidth}
     \includegraphics[width=\textwidth]{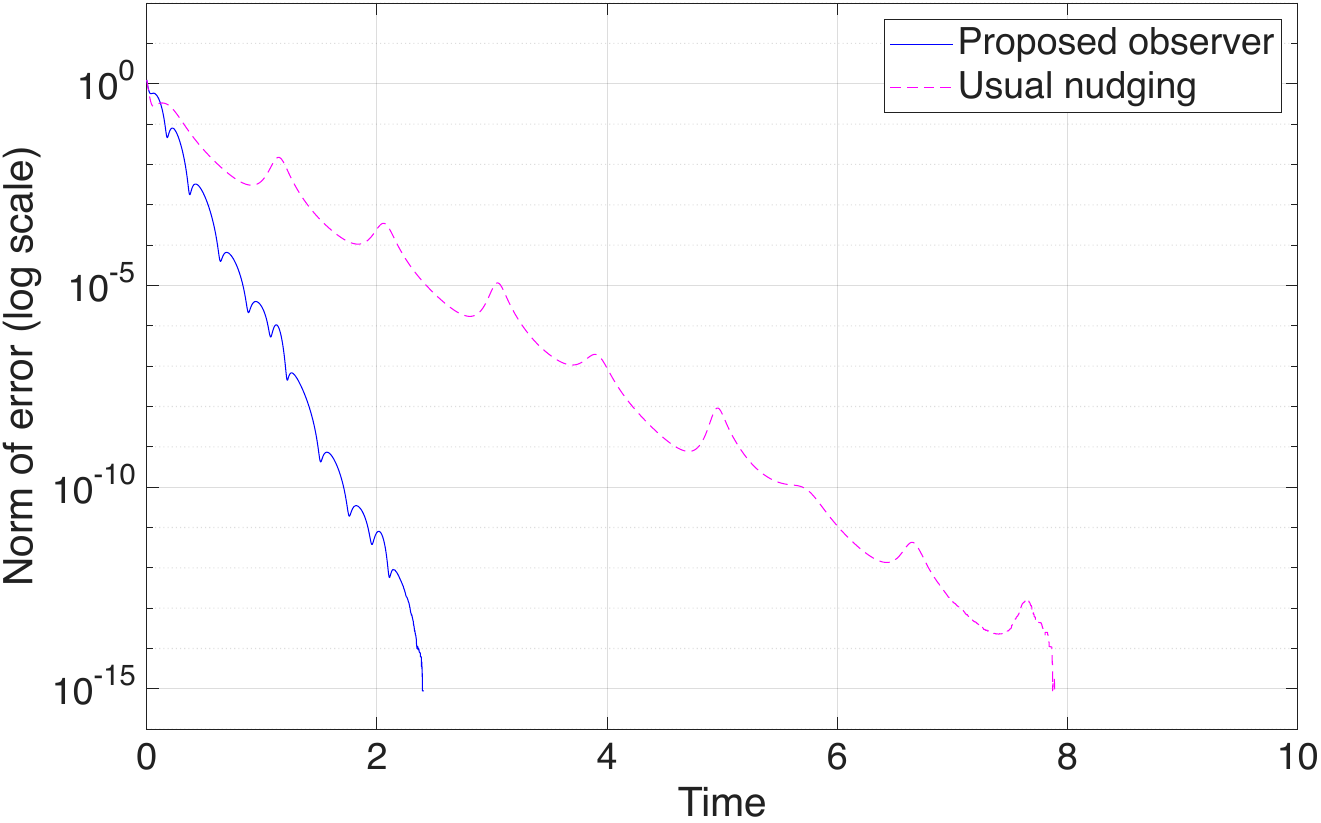}
     \caption{ $\mu=30 $  }
     \label{wr_err_mu30}
     \end{subfigure}
     \caption{ Comparison between norm of error with proposed observer and usual nudging in logarithmic scale using $x_0=[9.61, 9.66, 8.55]^T, \ z_0=[10.53, 9.70, 9.04]^T$ and $\mu=1 $ and $\mu=30 $  : WR model.}
     \label{wr_comp_log}
\end{figure}

\begin{figure}[htbp]
     \centering
     \begin{subfigure}[b]{0.35\textwidth}
     \includegraphics[width=\textwidth]{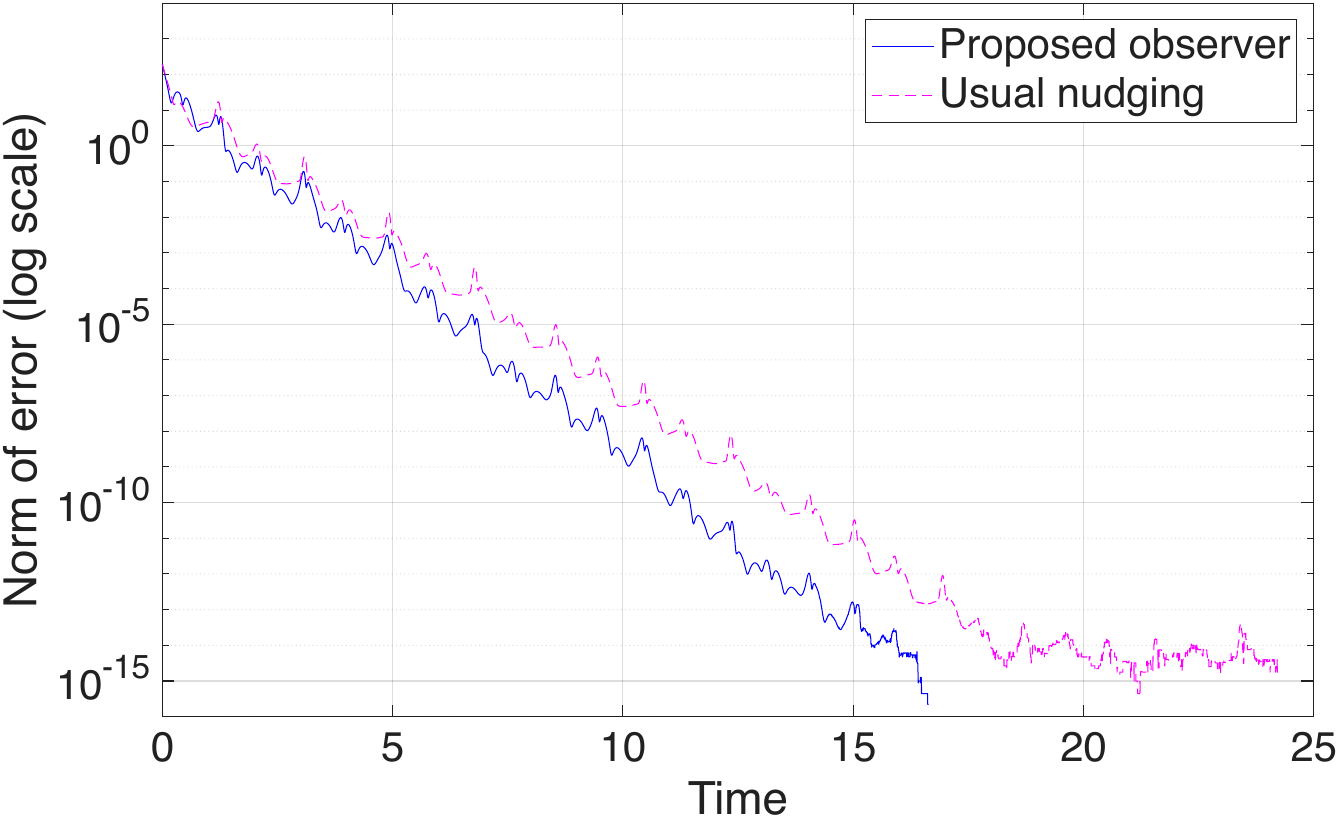}
     \caption{ $\mu=6 $  }
     \label{wr_mu6}
     \end{subfigure}
     \hspace{2cm}
    \begin{subfigure}[b]{0.35\textwidth}
     \includegraphics[width=\textwidth]{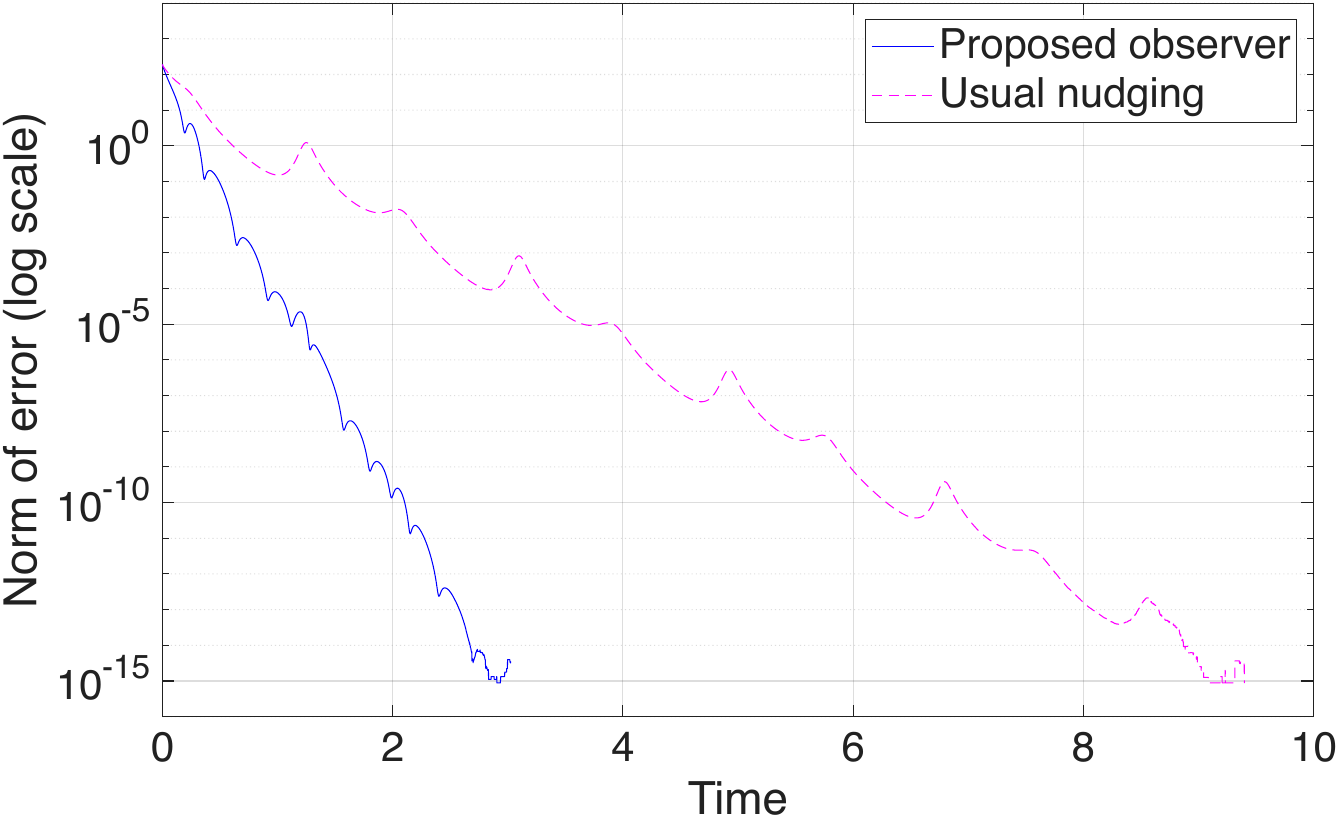}
     \caption{ $\mu=30 $  }
     \label{wr_large_err_mu30}
     \end{subfigure}
     \caption{ Comparison between norm of error with proposed observer and usual nudging in logarithmic scale using $x_0=[4.95, 8.17, 5.97]^T, \ z_0=[48.18, 150.97, 107.79]^T$ with $\mu=6$ and $\mu=30 $  : WR model.}
     \label{wr_comp_log_large error}
\end{figure}

Figure \ref{wr_each species} displays the state and observer trajectories 
when the initial error $z_0 - x_0$ is small. We see convergence with $\mu=1$. 
In Figure \ref{wr_each species_large error} a larger initial error is considered and this necessitated 
a larger $\mu$ (approximately $6$ or greater) and rapid convergence is seen in a short time scale with $\mu=30$.  
Figures \ref{wr_comp_log} and \ref{wr_comp_log_large error} show the errors (in log scale) of both the proposed observer and the usual nudging for two different parameter values of $\mu=1$ and $\mu=30$. For parameter $\mu$ values sufficiently large (about $\mu \geq 6$) the proposed observer converges faster even though both observers reach zero error numerically.  

From Figure \ref{wr_comp_log_large error} we can estimate the exponential decay rate 
for the proposed observer with $\mu = 30$ to be around $14$. The theoretical lower bound for $\alpha_0$ from Proposition \ref{prop1} is $k_5$ as discussed earlier 
and $k_5=10$. Thus, the results are consistent with the theory. 

\subsection{Simulation for Example $3$}
For the oscillator example, we use the parameter values
$
k_1 = 1, k_2 = 0.5, k_3 = 2, k_4 = 2, k_5 = 1, k_6 = 0.01, k_7 = 0.01, k_8 = 0.5.
$

\begin{figure}[htbp]
     \centering
     \includegraphics[width=0.4\linewidth]{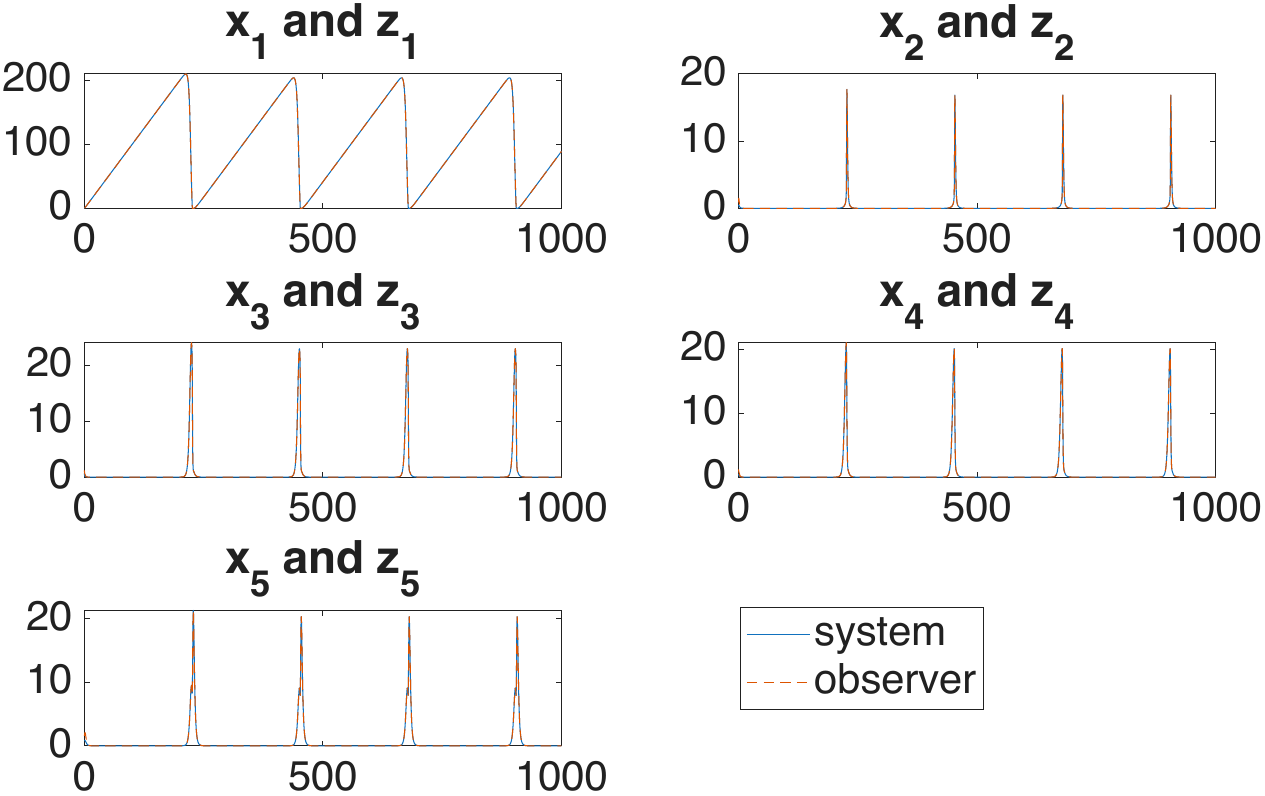}
     \caption{ Comparison between each component of system and observer with $ x_0 = [0.04, 0.30, 0.69, 0.28, 0.83]^T , \ z_0 = [0.62, 0.89, 1.20, 1.15, 1.57]^T $ and $\mu = 1$: Oscillator model with proposed observer 1 when $x_1$ and $x_3$ are observed.}
     \label{os_each species_long time span}
\end{figure}

\begin{figure}[htbp]
     \centering
     \includegraphics[width=0.4\linewidth]{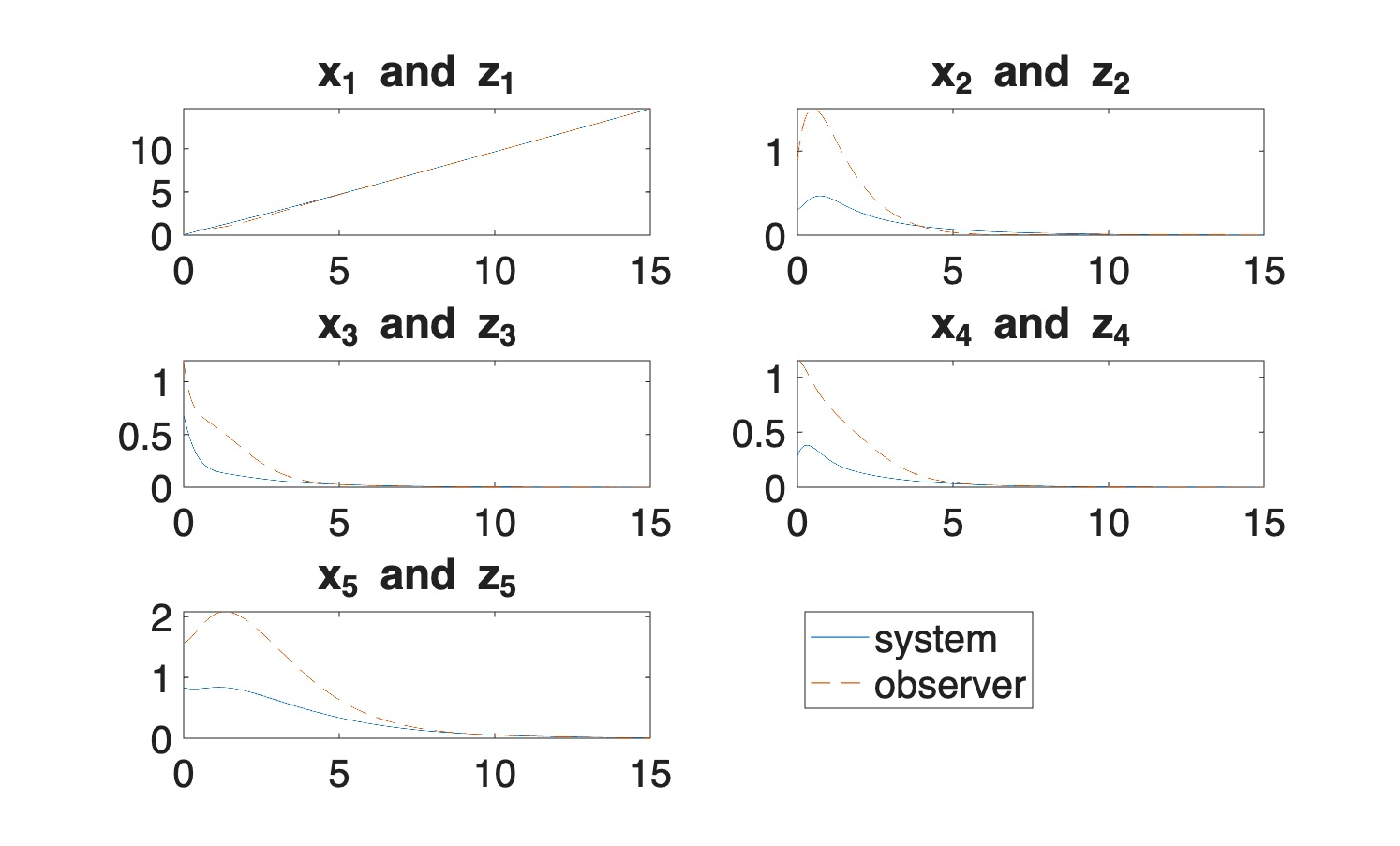}
     \caption{ Comparison between each component of system and observer in shorter time span with $ x_0 = [0.04, 0.30, 0.69, 0.28, 0.83]^T , \ z_0 = [0.62, 0.89, 1.20, 1.15, 1.57]^T $ and $\mu = 1$ : Oscillator model with proposed observer 1 when $x_1$ and $x_3$ are observed.}
     \label{os_each species_short time span}
\end{figure}

\begin{figure}[htbp]
     \centering
     
\begin{subfigure}[b]{0.35\textwidth}
     
     \includegraphics[width=\textwidth]{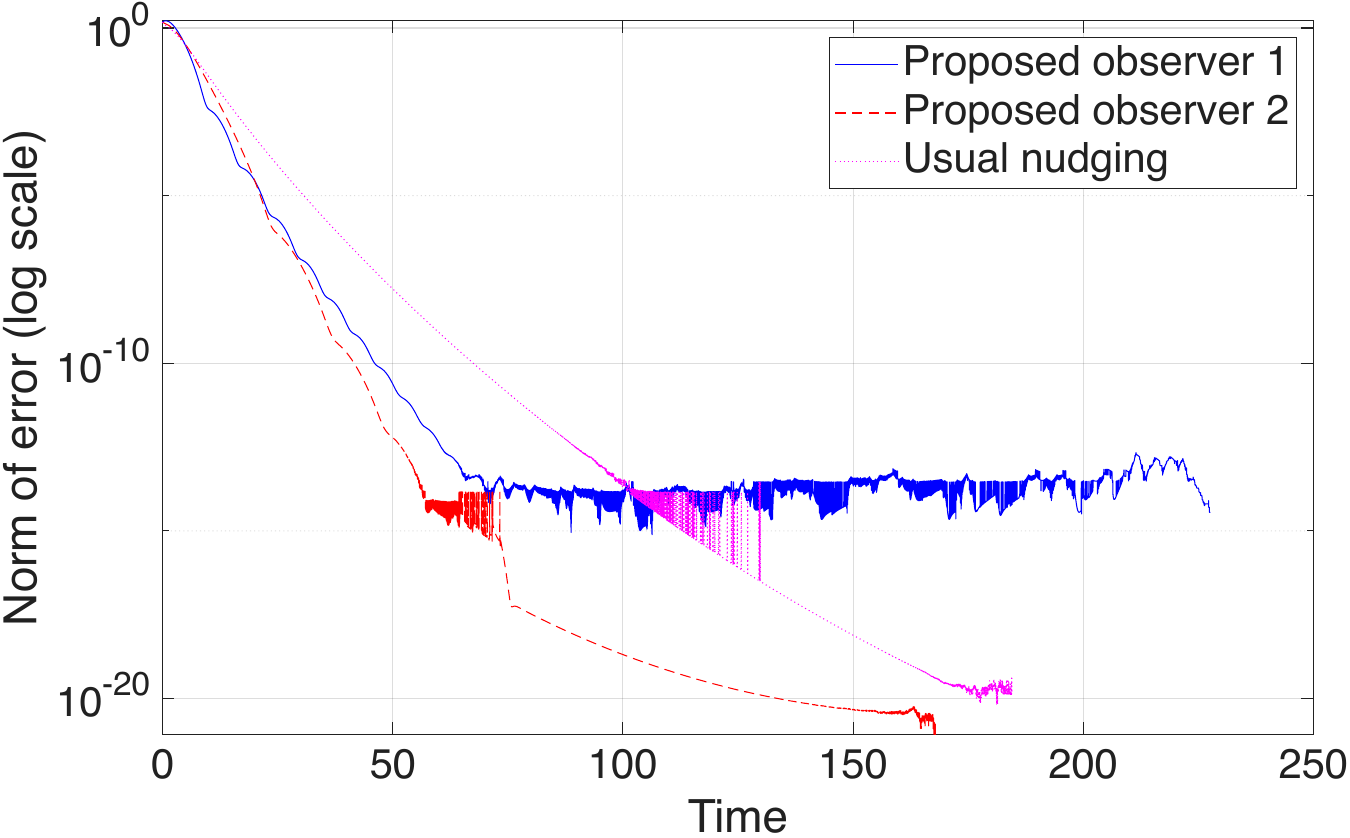}
     \caption{$\mu=1$}
     \label{osci_err_mu1}
\end{subfigure}
\hspace{2cm}
\begin{subfigure}[b]{0.35\textwidth}
     \includegraphics[width=\textwidth]{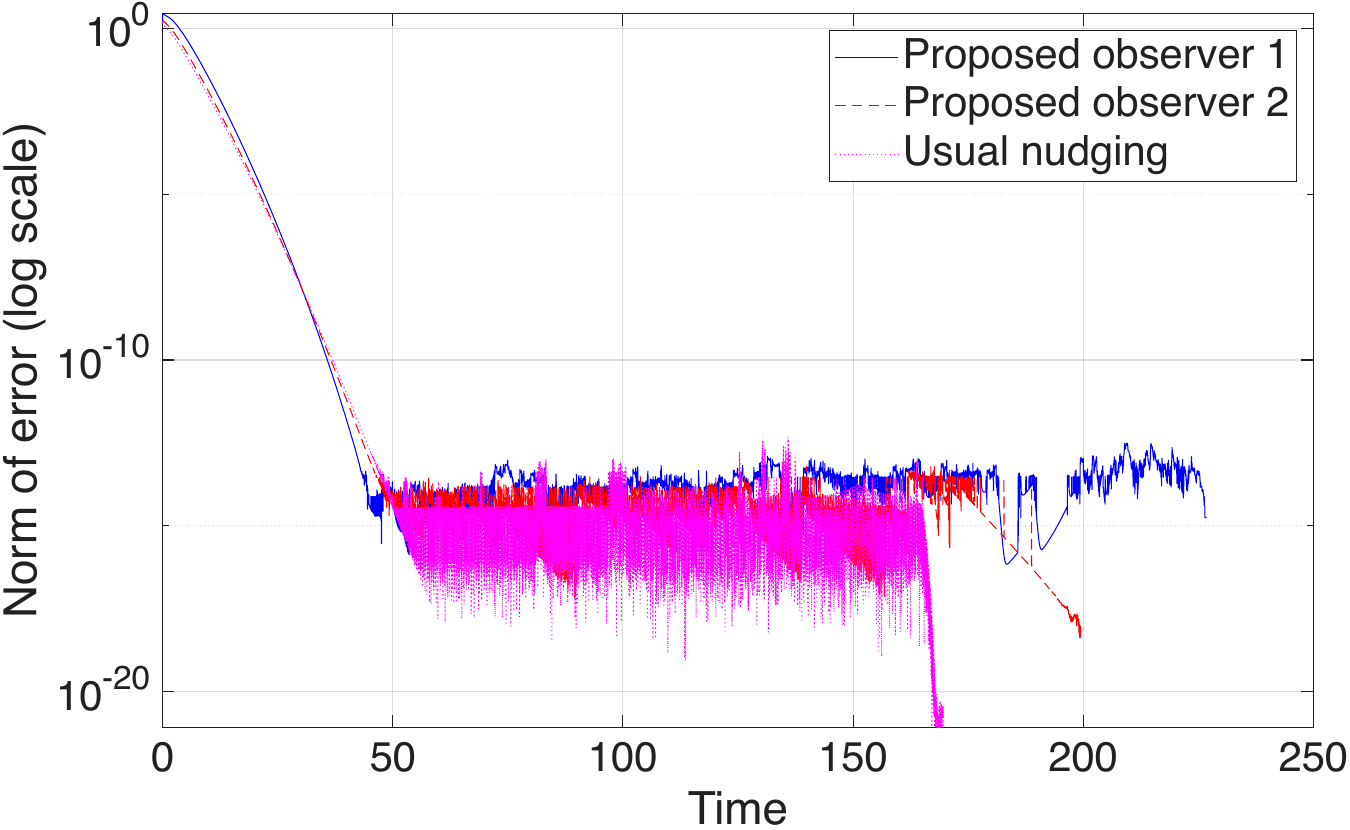}
     \caption{$\mu=30$}
     \label{osci_err_mu30}
\end{subfigure}
     \caption{Comparison between norm of error using proposed observers and usual nudging in logarithmic scale with $ x_0 = [0.04, 0.30, 0.69, 0.28, 0.83]^T , \ z_0 = [0.62, 0.89, 1.20, 1.15, 1.57]^T $ using  $\mu=1$ and $\mu=30$: Oscillator model where $x_1$ and $x_3$ are observed.}
     \label{osci_err}
\end{figure}


\begin{figure}[htbp]
     \centering
    \begin{subfigure}[b]{0.35\textwidth}
     \includegraphics[width=\textwidth]{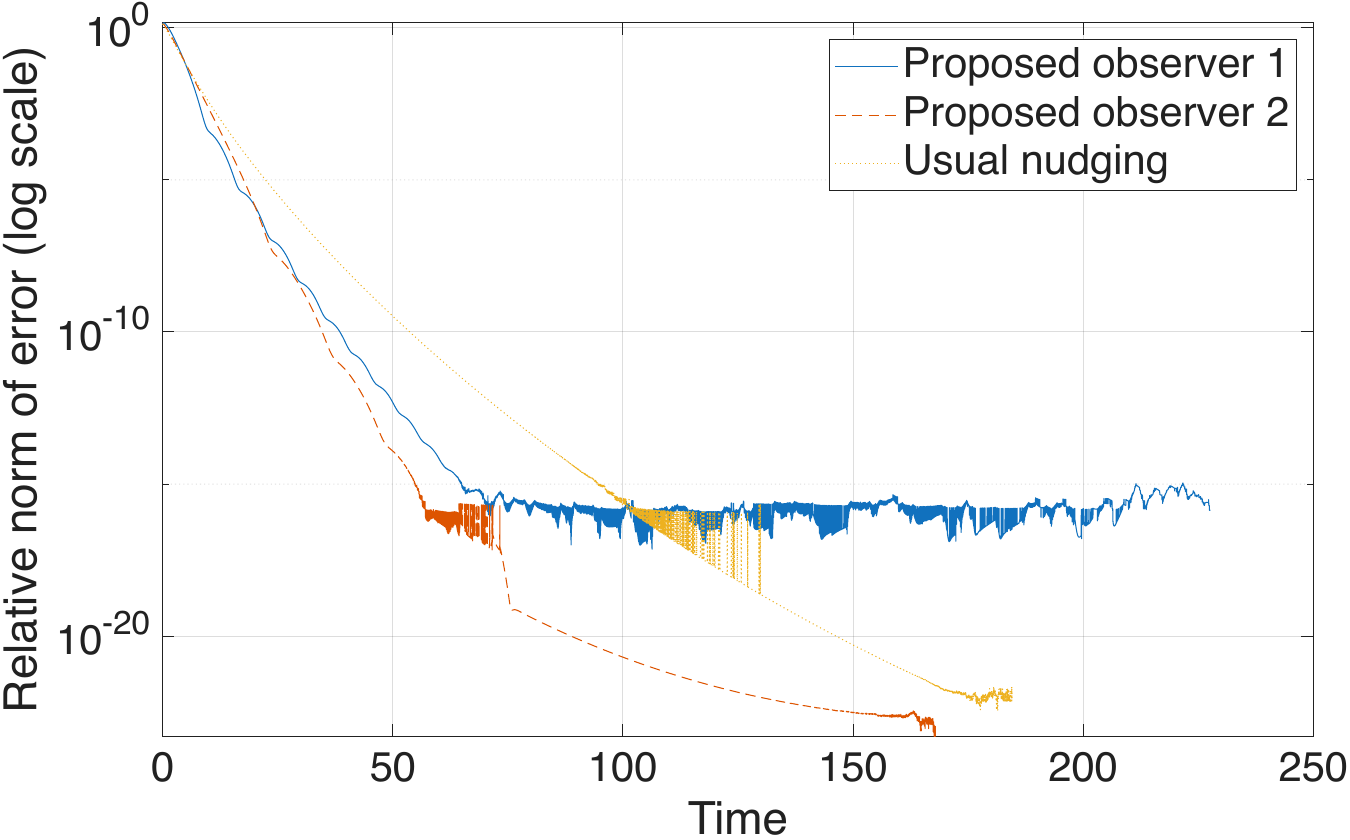}
     \caption{ $\mu=1 $  }
     \label{osci_rel_err_mu1}
     \end{subfigure}
     \hspace{2cm}
      \begin{subfigure}[b]{0.35\textwidth}
     \includegraphics[width=\textwidth]{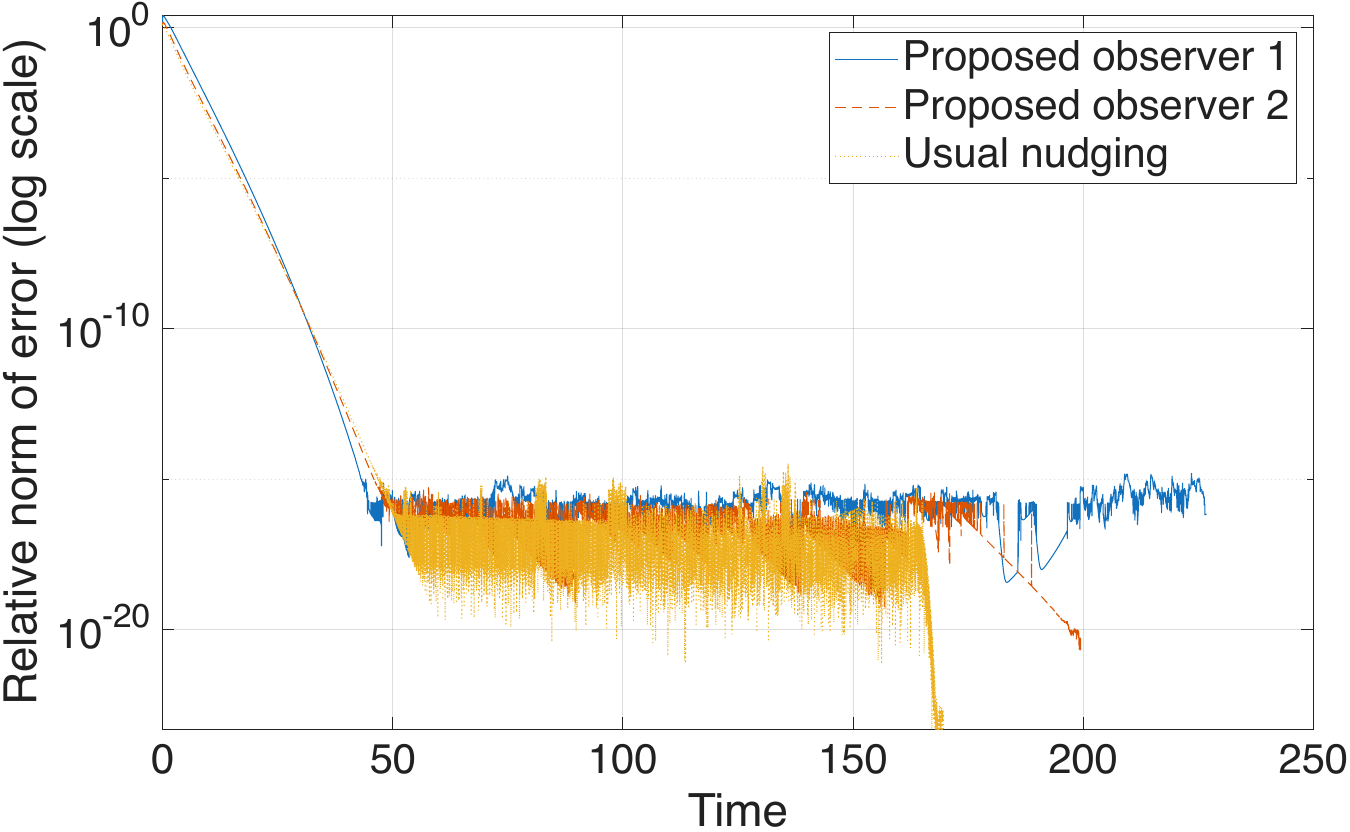}
     \caption{ $\mu=30 $  }
     \label{osci_rel_err_mu30}
     \end{subfigure}
     \caption{ Comparison between relative norm of error with proposed observers and usual nudging in logarithmic scale using $x_0=[0.04, 0.30, 0.69, 0.28, 0.83]^T, \ z_0=[0.62, 0.89, 1.20, 1.15, 1.57]^T$ with  $\mu=1$ and $\mu=30 $  : Oscillator model where $x_1$ and $x_3$ are observed.}
     \label{osci_rel_err_obs1}
\end{figure}

\begin{figure}[htbp]
     \centering
    \begin{subfigure}[b]{0.35\textwidth}
     \includegraphics[width=\textwidth]{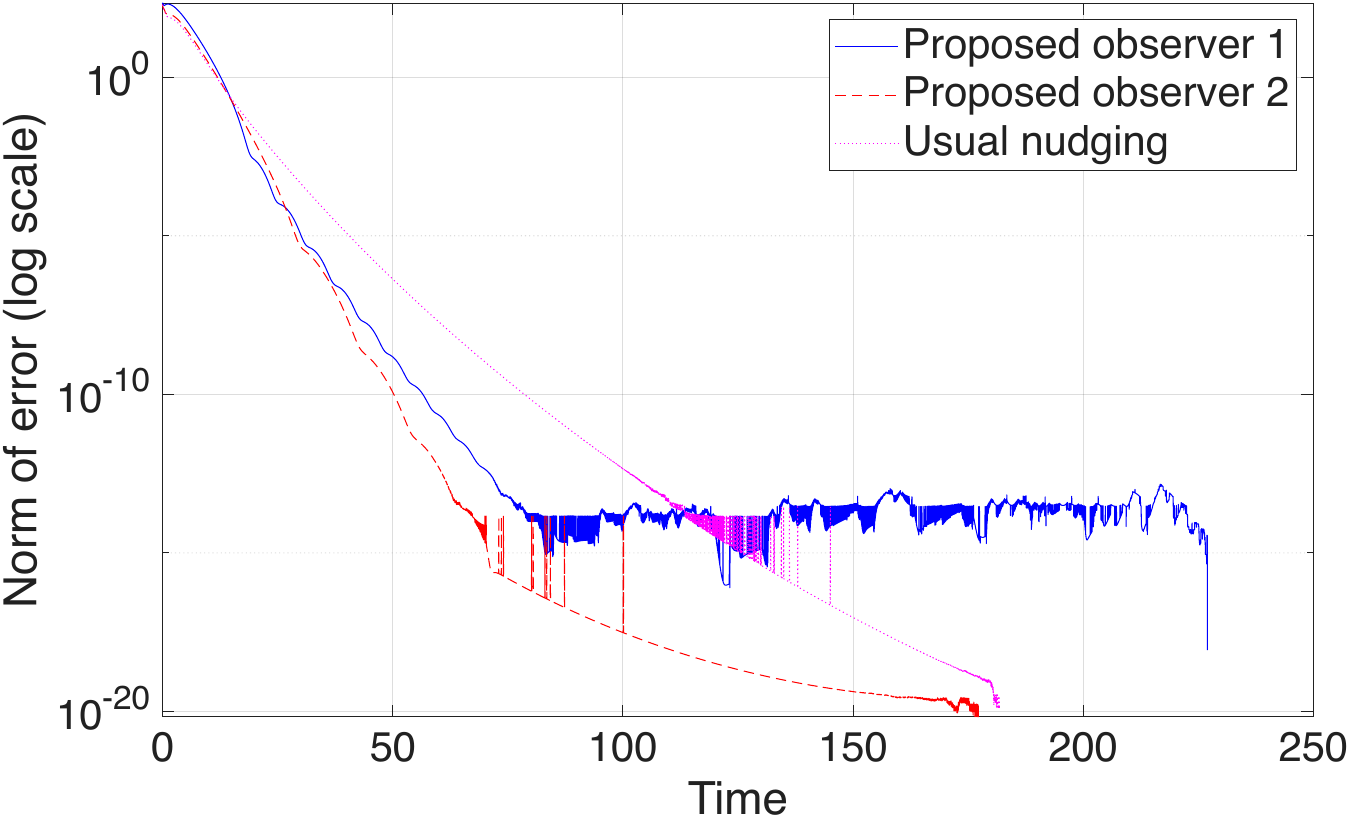}
     \caption{ $\mu=1 $  }
     \label{osci_err_large_mu1}
     \end{subfigure}
     \hspace{2cm}
     \begin{subfigure}[b]{0.35\textwidth}
     \includegraphics[width=\textwidth]{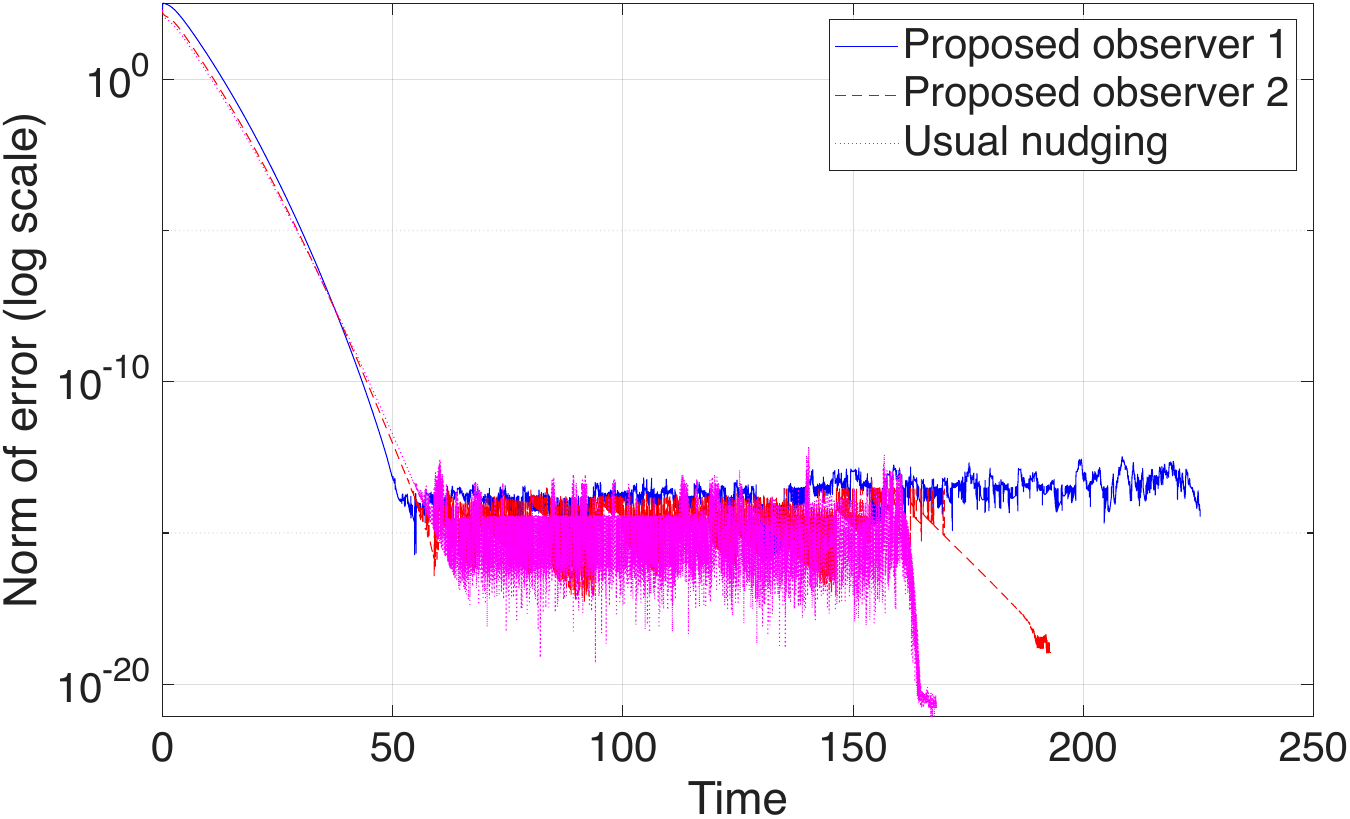}
     \caption{ $\mu=30 $  }
     \label{osci_err_large_mu30}
     \end{subfigure}
     \caption{ Comparison between norm of error with proposed observers and usual nudging in logarithmic scale using $x_0=[1.04,35.30, 7.69, 12.28, 10.83]^T, \ z_0=[100.62, 10.89, 15.20, 100.15, 150.57]^T$ with $\mu=1 $ and $\mu=30$  : Oscillator model where $x_1$ and $x_3$ are observed.}
     \label{osci_err_large}
\end{figure}

\begin{figure}[htbp]
     \centering
     \begin{subfigure}[b]{0.35\textwidth}
     \includegraphics[width=\textwidth]{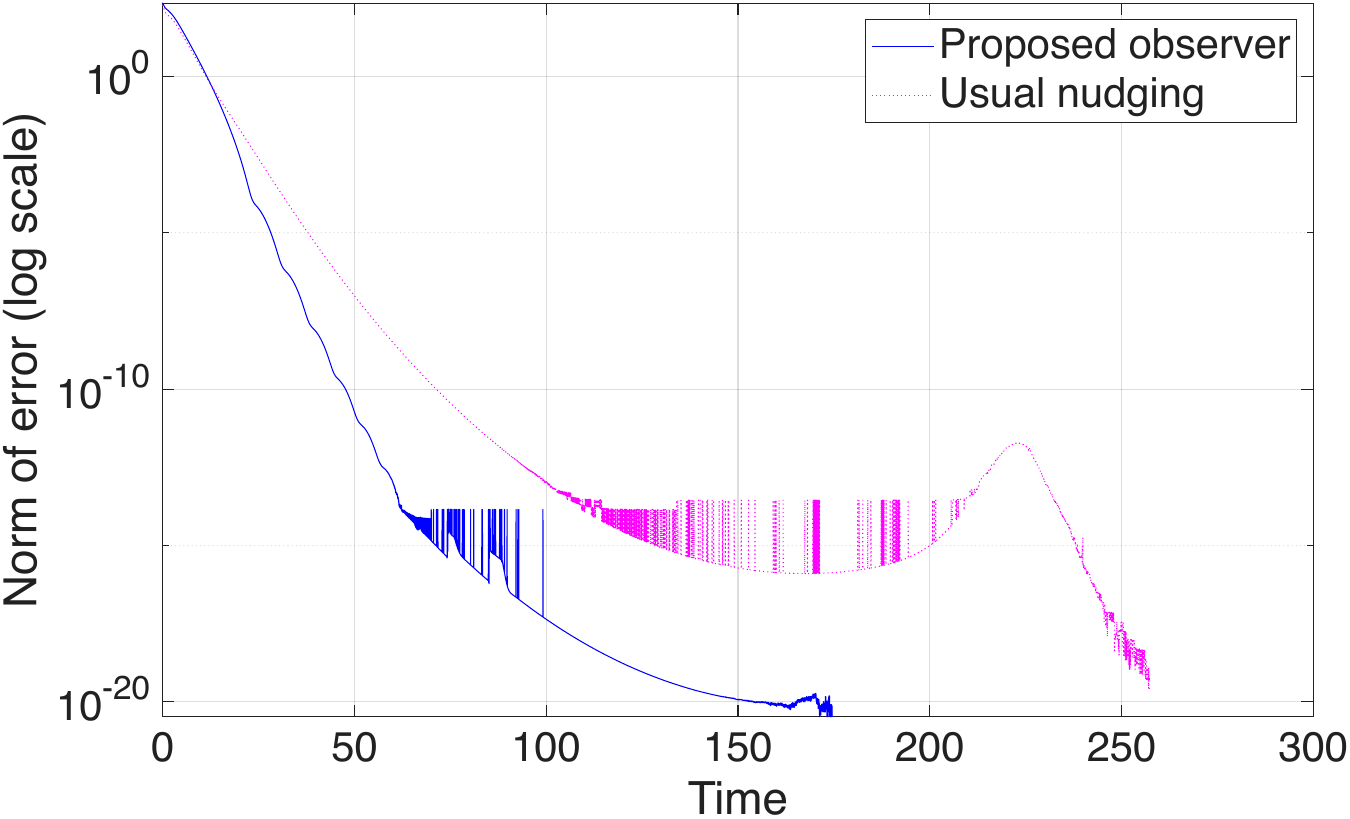}
     \caption{ $\mu=1 $  }
     \label{osci_err(2nd obs)_mu1}
     \end{subfigure}
     \hspace{2cm}
    \begin{subfigure}[b]{0.35\textwidth}
     \includegraphics[width=\textwidth]{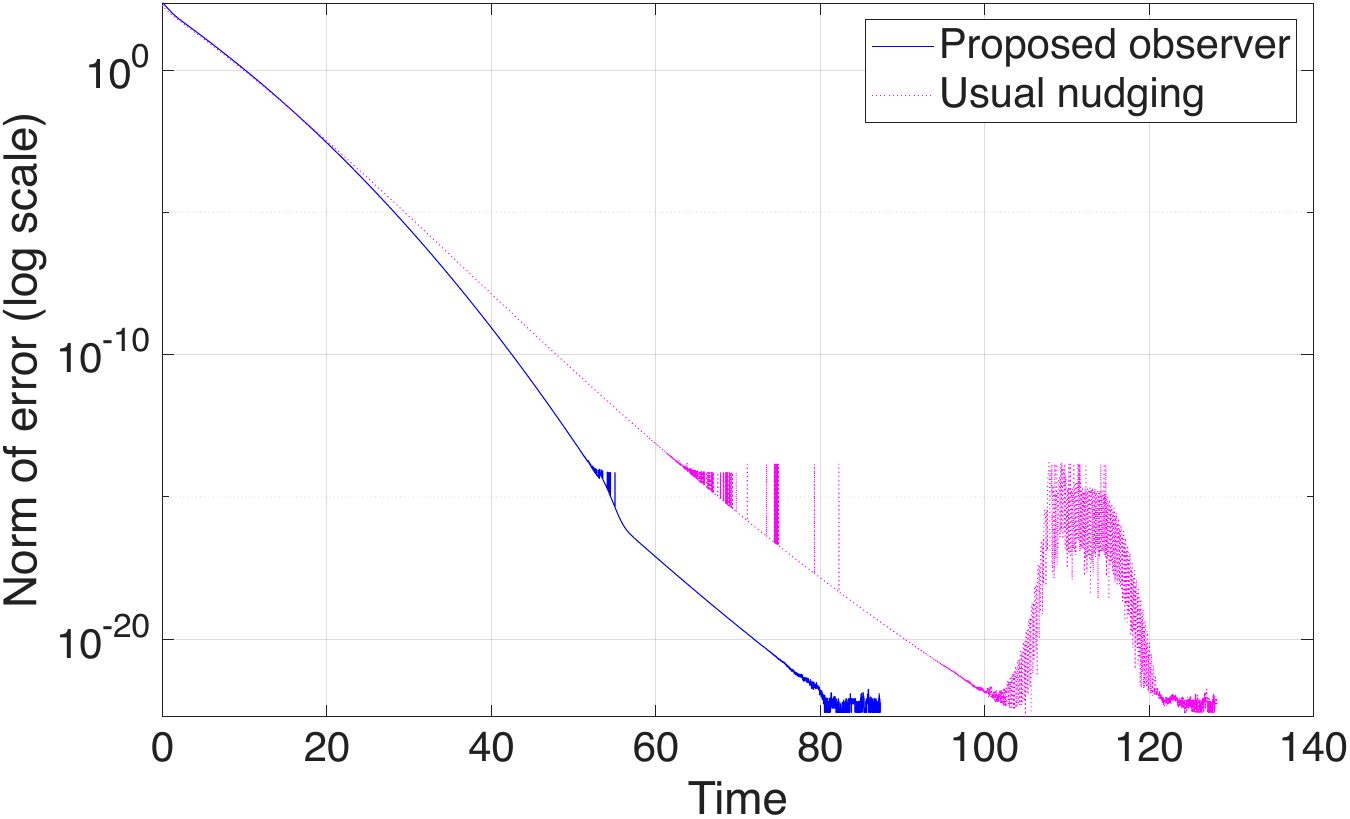}
     \caption{ $\mu=10 $  }
     \label{osci_err(2nd_obs)_mu30}
     \end{subfigure}
     \caption{ Comparison between norm of error with proposed observer and usual nudging in logarithmic scale using $x_0=[1.04,35.30, 7.69, 12.28, 10.83]^T, \ z_0=[100.62, 10.89, 15.20, 100.15, 150.57]^T$ with $\mu=1$ and $\mu=10 $  : Oscillator model where $x_1$ and $x_2$ are observed.}
     \label{osci_err_2nd obs}
\end{figure}

Figures \ref{os_each species_long time span}, \ref{os_each species_short time span}, \ref{osci_err}, \ref{osci_rel_err_obs1} and \ref{osci_err_large} illustrate the efficiency of our proposed observers 1 and 2 for the case where we observe $x_1$ and $x_3$.  
Figures \ref{os_each species_long time span} and \ref{os_each species_short time span} show the trajectories of the proposed observer 1 against the system states. Figure \ref{os_each species_short time span} illustrates the initial behavior of these trajectories. We can see that the observer state variables converge to the system state variables  within the time interval $t<15$. Figure \ref{os_each species_long time span} displays the behavior over a longer interval $t=0$ to $t=1000$, showing that our proposed observer tracks the oscillating system dynamics as the orange dotted line overlaps the blue solid line. In Figures \ref{osci_err} and  \ref{osci_err_large}, we show the comparison of the error norm on a logarithmic scale between the two proposed observers and the usual nudging method with 
$\mu=1$ and $\mu=30$. In Figure \ref{osci_err} we consider small initial error while in Figure \ref{osci_err_large} we consider large initial error. In Figure \ref{osci_err}, we see that while all observers converge, when $\mu=1$, proposed observer 2 (dashed orange line) is the fastest and the usual nudging observer (dotted magenta) is the slowest. For larger value of $\mu=30$, the differences in convergence rates are negligible. 
Since the oscillator trajectories display a wide dynamic range as seen in 
Figure \ref{os_each species_long time span}, it is important to consider a measure of relative error of the observers. We use the following formula for relative error $r(t)$:
\[
r(t) = \frac{|e(t)|}{\epsilon + |x(t)|},
\]
with $\epsilon=10^{-16}$. 
Figure \ref{osci_rel_err_obs1} shows the comparison between the relative errors of proposed observers and the usual nudging technique. Here also we can see that our all observers attain convergence of the relative error and the proposed observers performing better for the case of $\mu=1$. 

Figure \ref{osci_err_2nd obs}
shows the comparison of the error norm on logarithmic scale between the proposed observer and the usual nudging method when $x_1$ and $x_2$ are observed using $\mu=1$ and $\mu=10$. In this figure, proposed observer (blue solid line) exponentially decays more rapidly than usual nudging method (dotted magenta line).

\subsection{Simulation for Example $4$}
For this example we use parameter values
$   
    k_1 = 0.1,
    k_2 = 0.2,
    k_3 = 0.1,
    k_4 = 2,
    k_5 = 1,
    k_6 = 2
$. For this example we have two proposed observers. Our proposed observer 1  is the same as usual nudging.

\begin{figure}[htbp]
     \centering
     \begin{subfigure}[b]{0.35\textwidth}
     \includegraphics[width=\textwidth]{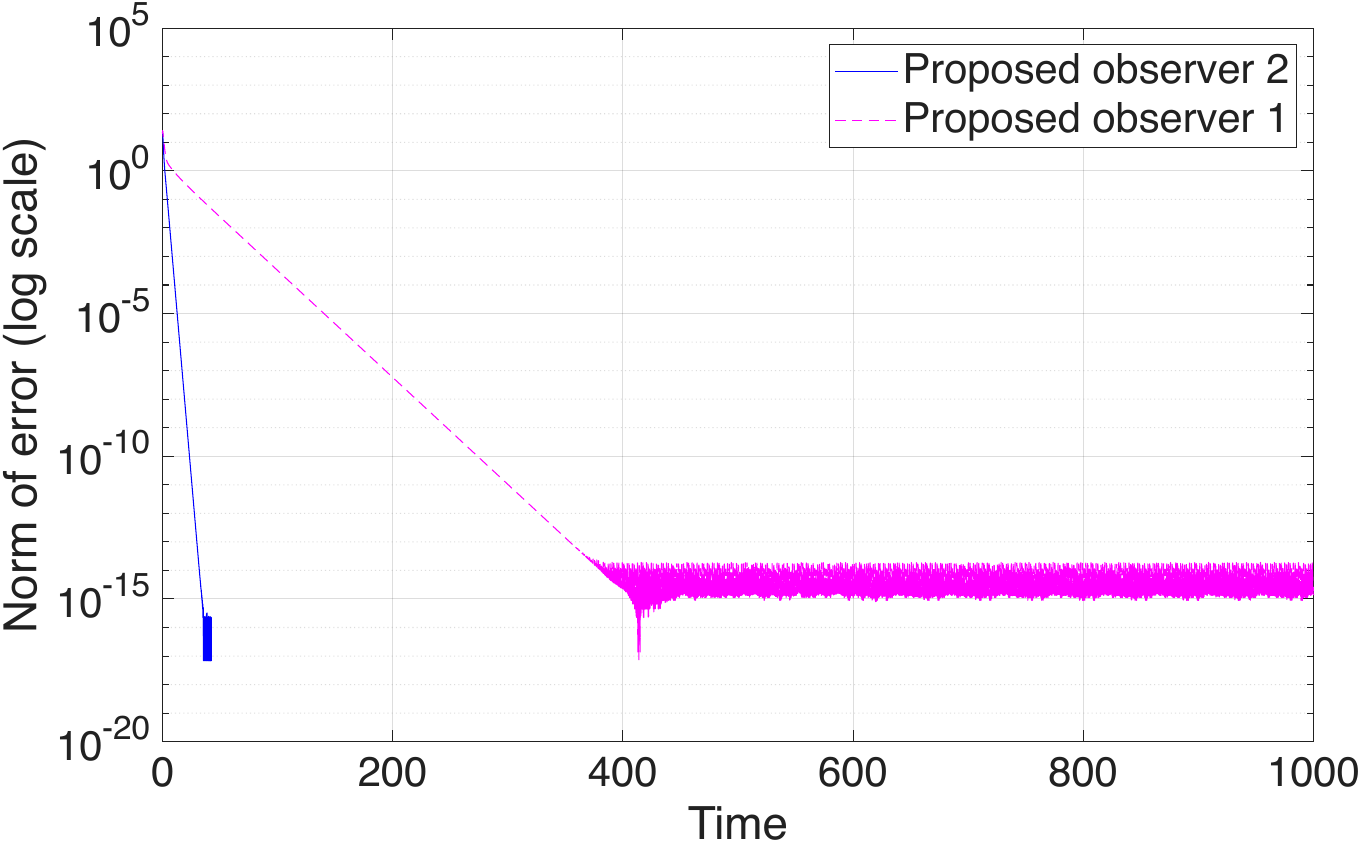}
     \caption{ $\mu=1 $  }
     \label{ex4_err_mu1}
     \end{subfigure}
     \hspace{2cm}
     \begin{subfigure}[b]{0.35\textwidth}
         \includegraphics[width=\textwidth]{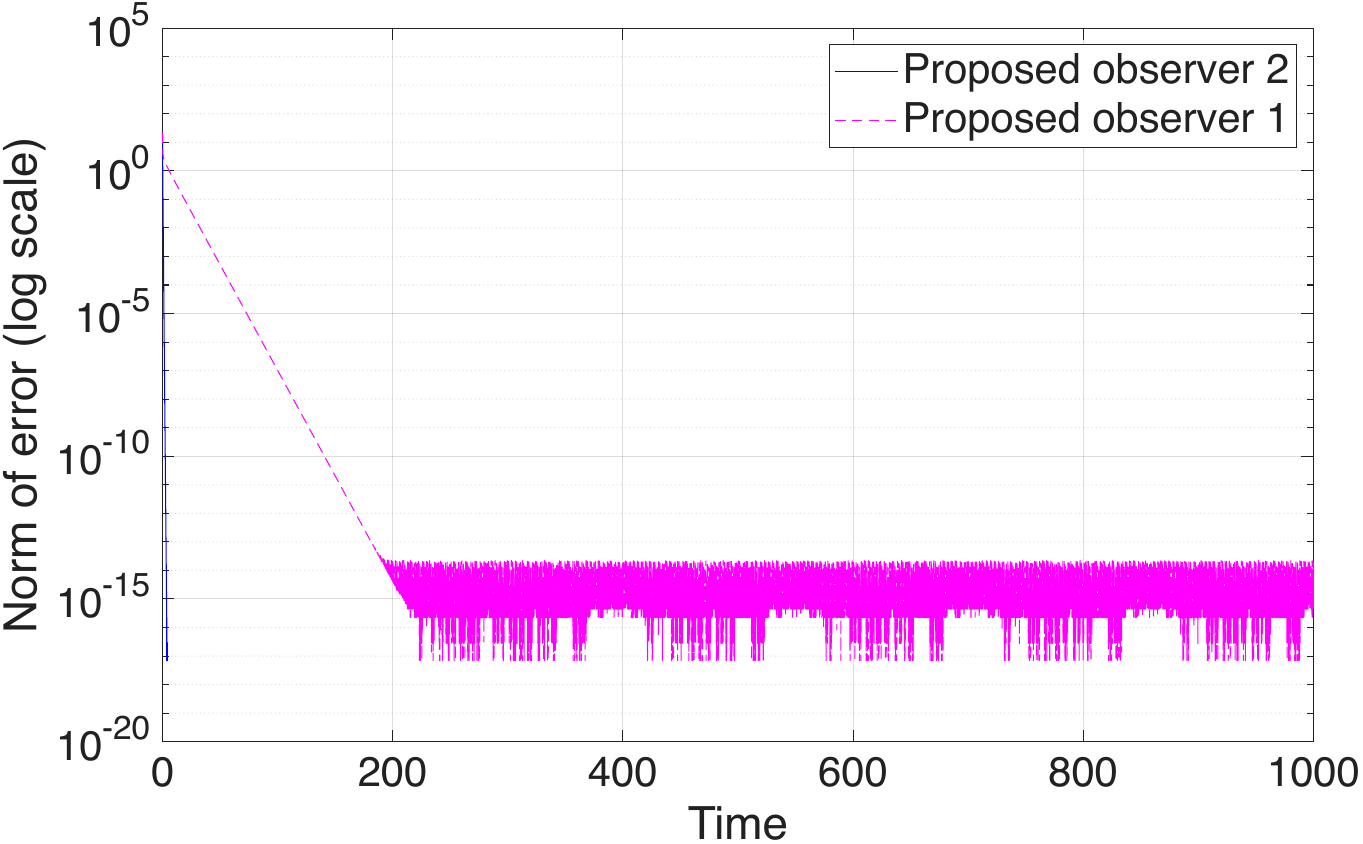}
         \caption{ $\mu=10$}
         \label{ex4_err_mu10}
     \end{subfigure}
     \caption{Comparison between norm of error with proposed observer 1 and observer 2 using $x_0=[0.9,0.5,0.7]^T, \ z_0 = [9,8,20]^T$ with $\mu=1$ and $\mu=10$ : Model in Example $4$.}
     \label{ex4_err}
\end{figure}
\begin{figure}[htbp]
     \centering
     \includegraphics[width=0.4\linewidth]{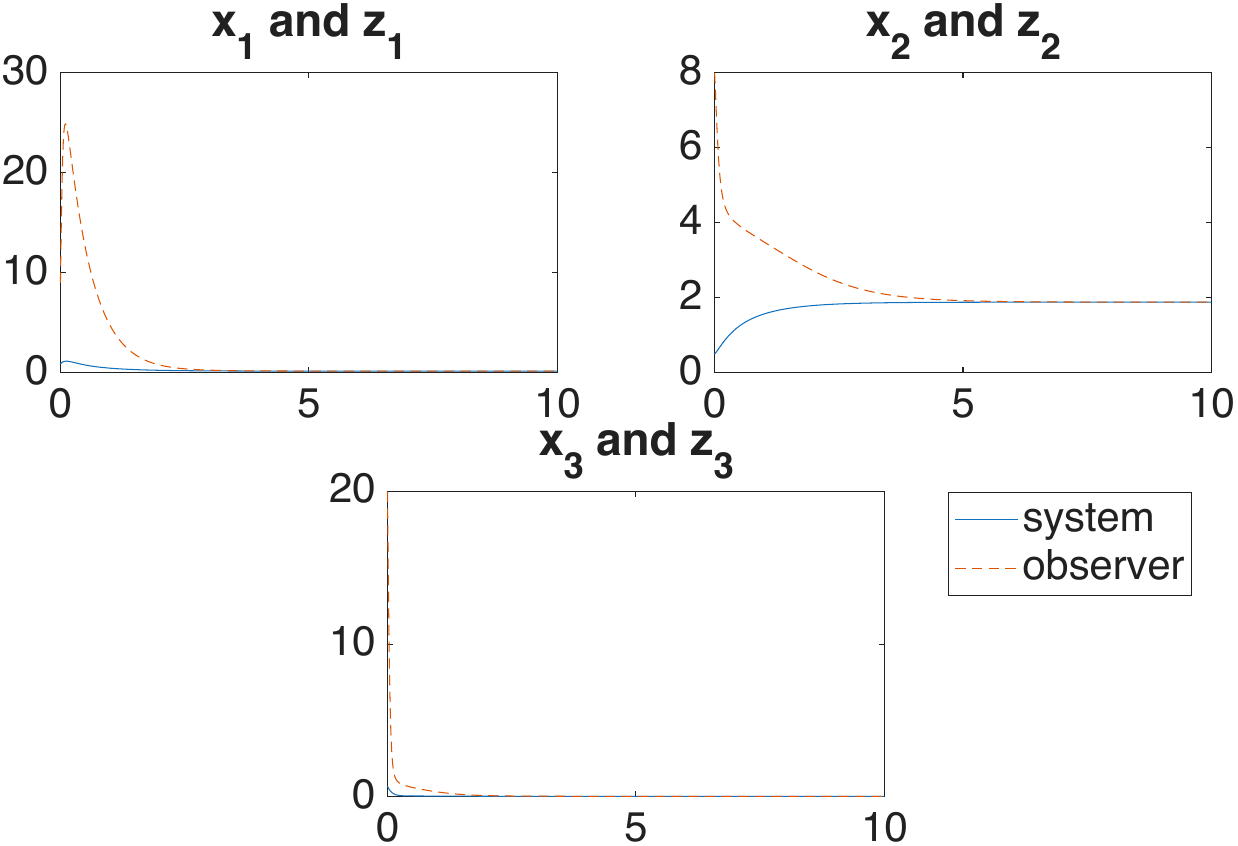}
     \caption{Comparison between each component of system($x$)and observer($z$) with $x_0=[0.9,0.5,0.7]^T, \ z_0 = [9,8,20]^T$ and $\mu=1$ : Model in Example $4$ with proposed observer 1.}
     \label{ex4_each species}
\end{figure}

\begin{figure}[htbp]
     \centering
     \includegraphics[width=0.4\linewidth]{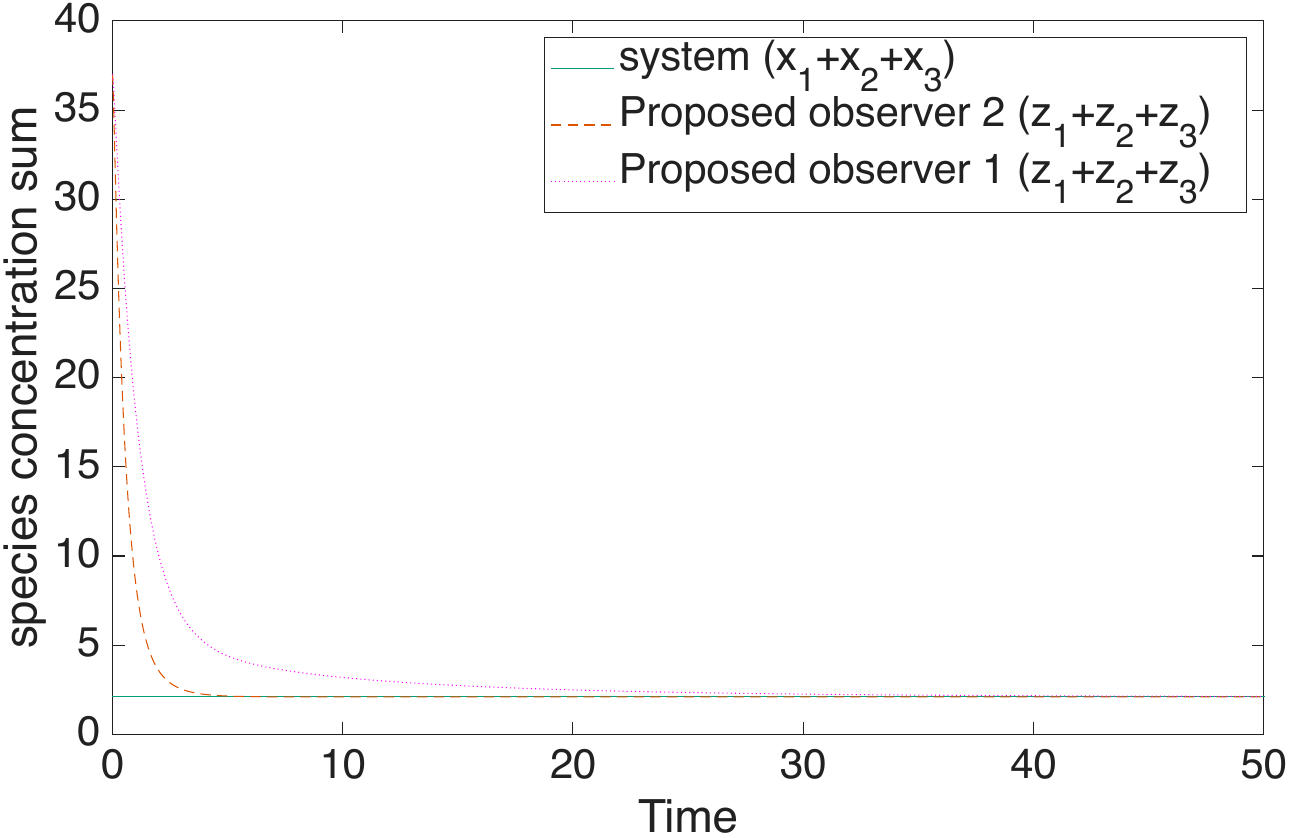}
     \caption{ Comparison between sum of species concentration for system and observers using proposed observer 1 and observer 2 with $x_0=[0.9,0.5,0.7]^T, \ z_0 = [9,8,20]^T$ and $\mu=1$ : Model in Example $4$.}
     \label{ex4_con sum}
\end{figure}

Figure \ref{ex4_err} illustrates the convergence of the error norm for the model described in Example 4 with the comparison between the two proposed observers for parameter values $\mu=1$ and $\mu=10$. This system is initialized with large error using system state $x_0=[0.9,0.5,0.7]^T$ and observer state $z_0=[9,8,20]^T$. The proposed observer 2 converges faster than the proposed observer 1 (same as usual nudging). In Figure \ref{ex4_each species} we show the individual state trajectories of the system and observer using $\mu=1$. Observer states overlap with the system states within the time span $t<8$. Figure \ref{ex4_con sum} compares the sum of species concentrations. Here we can see that the proposed observer 2 (dashed orange line) converges faster than the proposed observer 1 (dotted pink line) with the constant sum of species concentrations shown in solid green line.

\subsection{Noisy observation of the Willamowski-Rössler (WR) model}\label{sec-noisy}

For a dynamical system $\dot{x}=f(x)$, we consider the case of noisy observation $\eta(t)$ of the following form:
\begin{equation}
\eta(t) = \int_0^t C \, x(s) \, ds + \Sigma \, B(t),    
\end{equation}
where $B$ is a $m$ dimensional standard Brownian motion and $\Sigma$ is a symmetric positive definite matrix. In this case, we use the following observer:
\begin{equation}
z(t) = z(0) + \int_0^t f(z(s)) \, ds + \int_0^t G \, (d\eta(s) - C\, z(s) \, ds),      
\end{equation}
where $G$ is the same gain matrix proposed by our method for the case of noiseless observations. When $\Sigma=0$, $\eta(t)$ is differentiable and $\dot{\eta}(t)=y(t)=C \, x(t)$, the deterministic observation. And the above observer reduces to 
the observer considered earlier:
\[
\dot{z}(t) = f(z(t)) + G \, (y(t) - C \, z(t)).
\]

For the WR model where $x_1$ and $x_3$ are observed with noise, we take $\Sigma$ to be a scalar. We implement the same observer $G$ proposed for the case of noiseless 
observation. We also implemented a {\em particle filter} for the purpose of estimating the state. We compare the performance of the proposed observer with that of the particle filter. We refer the reader to \cite{bain2008fundamentals, doucet2001sequential} and references therein for details on particle filtering. 

We used the same parameter values
$   
    k_1 = 30,
    k_2 = 0.5,
    k_3 = 1,
    k_4 = 0.001,
    k_5 = 10,
    k_6 = 0.001,
    k_7 = 1,
    k_8 = 0.001,
    k_9 = 16.5 ,
    k_{10} = 0.5
$.

For the particle filter we used the Euler method with time step $10^{-3}$ to evolve the filter.
The particle weights were normalized after each Euler step. 
We also used the Euler method with same time step to generate the noisy observation and also to evolve the observer. The particle filter and the observer were both applied to the same observation trajectory. 
We took the system initial condition $x_0 = $ and observer initial condition $z_0=(9,8,11)$. 
We initialized the particle filter uniformly inside $[0,18] \times [0, 16] \times [0,22]$. 

Particles were resampled using multinomial resampling \cite{doucet2001sequential} whenever the effective 
sample size fell below $0.65 \, N_p$ where $N_p$ was the number of particles. 
Effective sample size $ESS$ is measured by 
\[
ESS = \frac{\left(\sum_{i=1}^{N_p} W_i\right)^2}{\sum_{i=1}^{N_p} W_i^2}.
\]
The state at time $t$ is estimated by the particle filter as follows:
\[
z_{PF}(t) = \frac{\sum_{i=1}^{N_p} z^{(i)}(t) \, W_i(t)}{\sum_{i=1}^{N_p} },
\]
where $z^{(i)}(t)$ is the state of the $i$th particle and $W_i(t)$ is its weight (at time $t$). 

 In order to avoid particle degeneracy we evolved the particles with 
 artificially added small noise $\sigma$, resulting in the Euler step  
\[
z^{(i)}(t_{k+1}) = z^{(i)}(t_k) + h \, f(z^{(i)}(t_k)) + \sigma \sqrt{h} \, \xi^{(i)}_k    
\]
where $h=t_{k+1}-t_k$ is the step size and $\xi^{(i)}_k$ are iid Gaussians with zero mean and unit variance. We chose $\sigma$ by trial and error to get the best performance. 

We found that the particle filter performed 
best with artificial noise $\sigma=0.1$ and required $N_p=10^4$ particles.
The proposed observer performed poorly compared to the particle filter when $\Sigma =1$ (large observation noise) as seen in Figures
 \ref{wr_PF_err_1} and \ref{wr_PF_each sp_1}.
When $\Sigma = 0.013$ (small noise), the performance of our proposed observer is better than the particle filter with $10^4$ particles and $\sigma = 0.1$ in Figure \ref{wr_PF_err_0.013} and \ref{wr_PF_each spe_0.013}.

It must be noted that the particle filter required $N_p=10^4$ particles for reliable estimates. Hence the computational burden of running a particle filter is significantly higher than implementing our observer.

\begin{figure}[htbp]
     \centering
     \begin{subfigure}[b]{0.35\textwidth}
     \includegraphics[width=\textwidth]{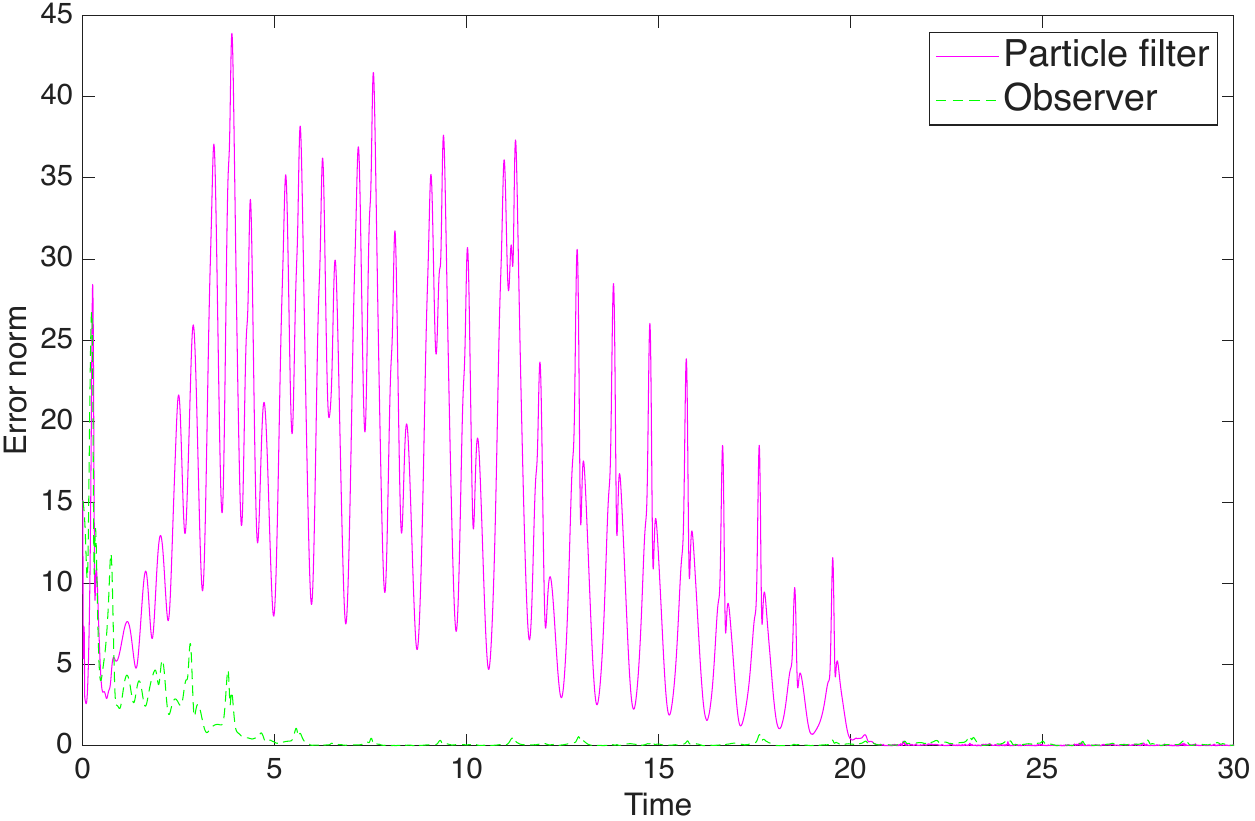}
     \caption{ $\Sigma = 0.013 $, $\mu = 3 $  }
     \label{wr_PF_err_0.013}
     \end{subfigure}
     \hspace{2cm}
     \begin{subfigure}[b]{0.35\textwidth}
         \includegraphics[width=\textwidth]{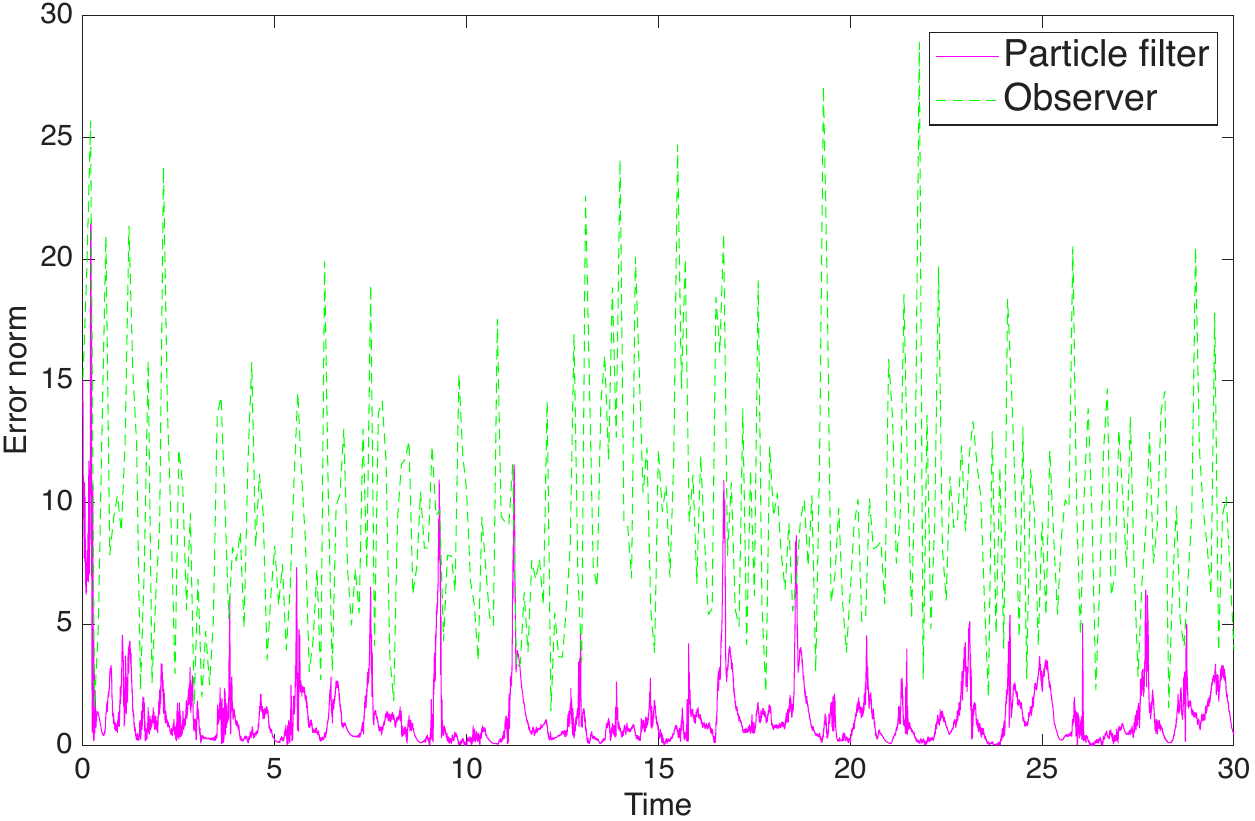}
         \caption{ $\Sigma = 1 $, $\mu = 40$}
         \label{wr_PF_err_1}
     \end{subfigure}
     \caption{Comparison between norm of error with proposed observer and particle filter using $x_0=[0.9,0.5,0.7]^T, \ z_0 = [9,8,11]^T$ with $\Sigma= 0.013, \mu= 3$ and $\Sigma=1, \mu=40$ : WR model.}
     \label{wrn_err}
\end{figure}

\begin{figure}[htbp]

     \centering
     \begin{subfigure}[b]{0.35\textwidth}
     \includegraphics[width=\textwidth]{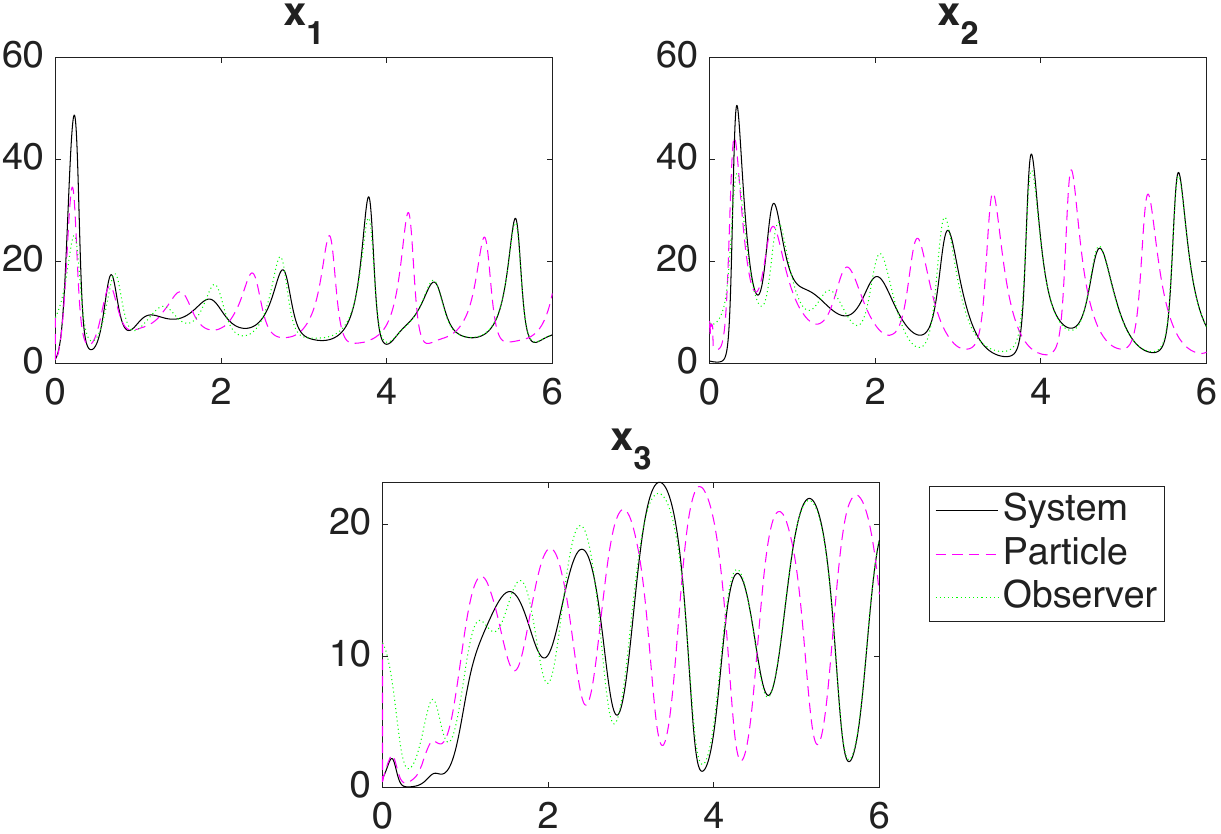}
     \caption{ $\Sigma=0.013, \mu=3 $  }
     \label{wr_PF_each spe_0.013}
     \end{subfigure}
     \hspace{2cm}
     \begin{subfigure}[b]{0.35\textwidth}
         \includegraphics[width=\textwidth]{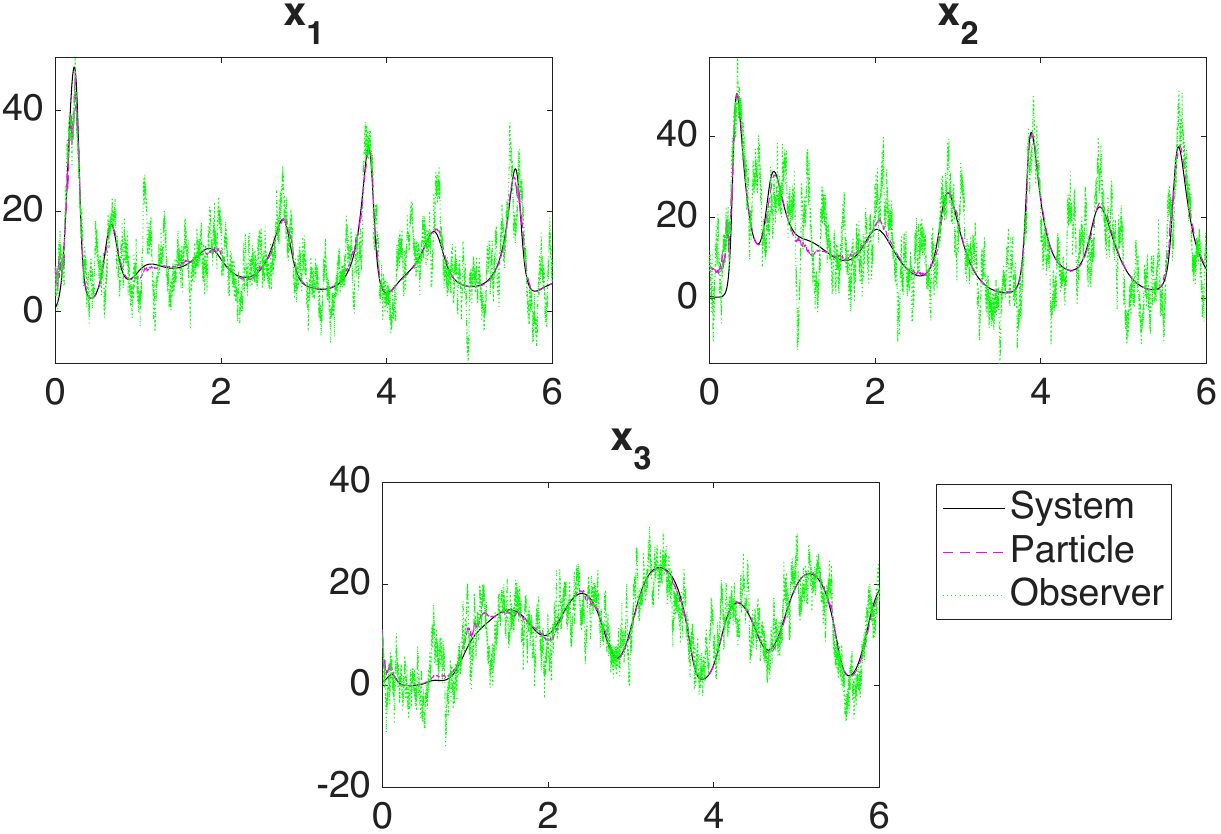}
         \caption{ $\Sigma=1, \mu=40$}
         \label{wr_PF_each sp_1}
     \end{subfigure}
     \caption{Comparison between the species concentration with proposed observer and particle filter using $x_0=[0.9,0.5,0.7]^T, \ z_0 = [9,8,11]^T$ with $\Sigma=0.013, \mu=3$ and $\Sigma=1, \mu= 40$ : WR model.}
     \label{wrn_each sp}
     
\end{figure}

\section{Conclusions}
\label{sec-conclusion}

We presented an approach to data assimilation via a tunable observer and 
provided sufficient conditions (Theorem \ref{genthm}) under which the observer error converged to zero exponentially. 
This result assures that regardless of the initial error one may tune the observer 
to achieve exponential error convergence. We provided two additional results, Propositions \ref{prop1} and \ref{prop2}, which were concerned with chemical reaction network (CRN) models and provided results that imply the contractivity conditions of Theorem \ref{genthm}. We illustrated the application of these propositions to mass action form of CRNs where the concentrations of a subset of species is observed. 

Numerical simulations were provided which confirmed the theory. The numerical results also compared the performance of what refer to as the usual nudging method. In several instances our proposed observer performed better than the usual nudging. 
More notably, in the Lotka-Volterra example, the performance of the usual nudging observer was counterintuitive in that it performed better for a smaller nudging parameter value and poorly for a larger parameter value. 
Even though the usual nudging observer appeared to work in most situations, our proposed observer (under certain conditions) can be tuned to guarantee exponential error convergence regardless of initial error while such a guarantee for the usual nudging method is lacking. 

We also presented numerical results for 
the case of observation noise in the chaotic CRN model. We compared the performance of 
our proposed observer with that of a particle filter. Numerical evidence suggests that the proposed observer performs well when the noise is small, a situation where the particle filter has difficulties. Moreover, the observer is much faster to implement numerically than the particle filter which required $10^4$ particles. However, we do not have any theoretical analysis as yet to justify the accuracy of our proposed observer for the case of noisy observations. This is the subject of future work. 

\appendix
\section{Solvability of the Lie algebra for Example 3 (Oscillator)} 
Given two matrices $X, Y \in \real^{n \times n}$ their Lie bracket $[X,Y] \in \real^{n \times n}$ is defined by $[X,Y]=X Y - Y X$. A Lie (sub) algebra $g \subset \real^{n \times n}$ is a vector subspace of $\real^{n \times n}$ that is closed under Lie brackets. The Lie algebra generated by a set of matrices is the intersection of all Lie (sub) algebras that contain that set, and is itself a Lie algebra. 

Given a Lie algebra $g$ of matrices the derived series $g^{(k)}$ for $k=0,1,2 \dots$ are defined by $g^{(0)}=g$ and 
\[
g^{(k)} = \text{span}\{ [X,Y] \, | X, Y \in g^{(k-1)}\} \quad k \geq 1.
\]
Note that $g^{(k)}$ is a Lie sub algebra of $g^{(k-1)}$. 
A Lie algebra $g$ is said to be {\em solvable} if $g^{(k)}=\{0\}$ for some $k \geq 0$.

\begin{lemma}\label{lie_lem}
    The Lie algebra $\{A, B\}_{LA}$ generated by $A, B \in \real^{3 \times 3}$ is solvable if $A$ and $B$ are of the following form:
            \[A= \begin{pmatrix} a_{11} & a_{12} & 0 \\ 0 & a_{22} & 0 \\ a_{31} & a_{32} & a_{33} \end{pmatrix},\ \ B= \begin{pmatrix} b_{11} & 0 & 0 \\ 0 & b_{22} & 0 \\ 0 & b_{32} & b_{33} \end{pmatrix}.\]
\end{lemma}
\begin{proof}

Let $g$ be the Lie algebra generated by $\{A, B\}$.  

We first show that $g^{(1)} \subseteq S$ where
$S = \left\{\begin{pmatrix} 0 & x & 0 \\ 0 & 0 & 0 \\ y & z & 0 \end{pmatrix} : x,y,z \in \mathbb{R}\right\}$. 

To that end, let $C_n$ be the set of all Lie brackets of $A$ and $B$ of length $n \ge 1$. 
Thus $C_1=\{A, B\}$, $C_2=\{[A,B],0\}$ (note that $[A,A]=[B,B]=0$) and so on.  
We note that $g^{(1)}$ is spanned by $\bigcup_{n \geq 2} C_n$. 

Next, we show via induction that  $C_n \subset S$ for $n \ge 2$. 
For $n=2$,
\[
[A, B] = \begin{pmatrix} 0 & a_{12}(b_{22}-b_{11}) & 0 \\ 0 & 0 & 0 \\ a_{31}(b_{11}-b_{33}) & a_{32}(b_{22}-b_{33}) + b_{32}(a_{33}-a_{22}) & 0 \end{pmatrix} 
\in S. 
\]
Suppose $C_j \subset S$ for all brackets of length $j \leq k$.
Let $J \in C_{k+1}$ be nonzero. We consider cases.

\em{Case 1:} $J=[M, N]$ where 
length of $M$ is 1 (and hence $M=A$ or $M=B$) and length of $N$ is $k$ or vice versa. 
By hypothesis $N \in S$. So, we can write it as
$ N = \begin{pmatrix} 0 & x_1 & 0 \\ 0 & 0 & 0 \\ y_1 & z_1 & 0 \end{pmatrix}$.\\
Now,
\[
[A, N] = \begin{pmatrix} 0 & x_1(a_{11}-a_{22}) & 0 \\ 0 & 0 & 0 \\ y_1(a_{33}-a_{11}) - z_1a_{31} & a_{31}x_1 + z_1(a_{33}-a_{22}) - y_1a_{12} & 0 \end{pmatrix} \\
\in S
\]
and
\[
[B, N] = \begin{pmatrix} 0 & x_1(b_{11}-b_{22}) & 0 \\ 0 & 0 & 0 \\ y_1(b_{33}-b_{11}) & z_1(b_{33}-b_{22}) & 0 \end{pmatrix} 
\in S.
\]
In the alternative, $N=A$ or $N=B$ and $M \in S$, and it is clear that $[M,A], [M,B] \in S$. 

\em{Case 2:} $J=[M,N]$, where lengths of $M$ and $N$ are $\ge 2$.
Thus $M, N \in S$. Then
\begin{align*}
[M, N] &= \left[ \begin{pmatrix} 0 & x_2 & 0 \\ 0 & 0 & 0 \\ y_2 & z_2 & 0 \end{pmatrix}, \begin{pmatrix} 0 & x_1 & 0 \\ 0 & 0 & 0 \\ y_1 & z_1 & 0 \end{pmatrix} \right] \\
&= \begin{pmatrix} 0 & 0 & 0 \\ 0 & 0 & 0 \\ 0 & y_2x_1 - y_1x_2 & 0 \end{pmatrix} \\
&\in S.
\end{align*}
This proves that $C_n \subset S$ for $n \geq 2$. 
Consequently $g^{(1)} \subset S$. The above calculations also show that $S$ is a Lie sub algebra. Now, $g^{(2)} \subset S^{(1)}$.
So, we have
\[
g^{(2)} \subset \left\{\begin{pmatrix} 0 & 0 & 0 \\ 0 & 0 & 0 \\ 0 & x & 0 \end{pmatrix} : x \in \mathbb{R}\right\}, 
\]
and hence
\[
g^{(3)} = \{0\}.
\]
Therefore, $g$ is a solvable Lie algebra.
\end{proof}

\section*{Acknowledgments}
The work of AB was supported in part by NSF grant DMS-2529382.

\bibliographystyle{siamplain}
\bibliography{references}

\end{document}